\documentclass[11pt]{elsarticle}
\usepackage{amssymb,amsfonts,amsmath,mathrsfs,mathtools}
\usepackage{graphicx,epsfig,subfigure}
\usepackage{bm}
\usepackage[margin=1in,papersize={8.5in,11in}]{geometry}
\usepackage{times}
\usepackage{multirow}
\usepackage{color}
\usepackage[vlined,ruled]{algorithm2e}
\usepackage{amsthm}
\usepackage{soul}

\SetCommentSty{mycommfont}
\usepackage[T1]{fontenc}

\renewcommand{\top}{\text{T}}

\def\vs{\vspace{0.2cm}}
\def\T{\text{\tiny TT}}

\DeclareMathOperator*{\argmin}{\arg\!\min}

\newtheorem{lemma}{\bf Lemma}[section]

\newtheorem{proposition}{\bf Proposition}[section]

\renewenvironment{proof}{{\noindent \em Proof.}}{\begin{flushright}
\qed\end{flushright}}

\title{Tensor rank reduction via coordinate flows}

\begin{document}
\begin{frontmatter}

\author[ucsc]{Alec Dektor}
\author[ucsc]{Daniele Venturi\corref{correspondingAuthor}}
\ead{venturi@ucsc.edu}

\address[ucsc]{Department of Applied Mathematics, 
University of California Santa Cruz\\ Santa Cruz (CA) 95064}

\cortext[correspondingAuthor]{Corresponding author}

\journal{the Journal of Computational Physics}

\begin{abstract}
Recently, there has been a growing interest in efficient numerical algorithms based on tensor networks and low-rank techniques to approximate high-dimensional functions and {  solutions to high-dimensional PDEs}. In this paper, we propose a new tensor rank reduction method {  based on coordinate transformations} that can greatly increase the efficiency of high-dimensional tensor approximation algorithms. The idea is simple: given a multivariate function, determine a coordinate transformation so that the function in the new coordinate system has smaller tensor rank. We restrict our analysis to linear coordinate transformations, which gives rise to a new class of functions that we refer to as tensor ridge functions. {  Leveraging Riemannian gradient descent on matrix manifolds we develop an algorithm that determines a quasi-optimal linear coordinate transformation for tensor rank reduction.} 
The results we present for rank reduction via linear coordinate transformations {  open the 
possibility for generalizations to 
larger classes of nonlinear transformations. Numerical applications 
are presented and discussed for linear and nonlinear PDEs.}
\end{abstract}
\end{frontmatter}

\section{Introduction}
\label{sec:intro}
There has been a growing interest in efficient numerical 
algorithms based on tensor networks and low-rank techniques 
to approximate high-dimensional functions and solutions 
to high-dimensional PDEs \cite{Kolda,HeyrimJCP_2014,
tensor_high_dim_pde,Vandereycken_2019,venturi2018numerical,VenturiSpectral}.
A tensor network is a factorization of an entangled 
object such as a multivariate function or an operator, into 
a set of simpler objects (e.g., low-dimensional functions or operators) 
which are amenable to efficient representation and computation. 
The process of building a tensor network relies on a 
hierarchical decomposition of the entangled object, which,  
can be visualized in terms of trees \cite{h_tucker_geom,Falco_2016,Bachmayr}. Such a  decomposition is rooted in the spectral theory for linear operators \cite{kato}, 
and it opens the possibility to approximate high-dimensional functions 
and compute the solution of high-dimensional PDEs at 
a cost that scales linearly with respect to the dimension 
of the object and polynomially with respect to the tensor 
rank.

Given the fundamental importance of tensor rank in computations and 
its non-favorable scaling, in this paper we propose a new tensor 
rank reduction method based on coordinate flows that can
greatly increase the efficiency of high-dimensional tensor 
approximation algorithms.
To describe the method, consider the scalar field $u(\bm x)$, 
where $\bm x \in \Omega \subseteq \mathbb{R}^d$, $d>1$. 
The idea is very simple: determine an invertible  coordinate 
transformation $\bm H: \mathbb{R}^d \to \mathbb{R}^d$ 
so that the function 
\begin{equation}
\label{nonlinear_coordinates}
v(\bm x) = u(\bm H( \bm x))
\end{equation}
has smaller tensor rank than $u(\bm x)$. 
Representing a function on a transformed coordinate system 
has proven to be a successful technique for a wide 
range of applications \cite{Haoxiang,Conjugateflows,GKSS_2005}, 
{  including tensor rank reduction in 
quantum many-body problems \cite{orbital_optimization}.}
%
%
To illustrate the effects of coordinate transformations 
on tensor rank, in Figure \ref{fig:rotating_2D_example} 
we show that a simple two-dimensional rotation can increase the rank 
of fully separated (i.e., rank one) Gaussian function significantly. 
Vice versa, the inverse rotation can transform a rotated Gaussian with 
high tensor rank into a rank one function. Similar results 
hold of course in higher dimensions. 

Under mild assumptions on the function $u(\bm x)$, 
one may argue, e.g. using nonlinear dynamics or 
the theory of optimal mass transport, that there 
always exists a transformation $\bm H$ such that 
$v(\bm x)=u(\bm H(\bm x))$ possesses the 
smallest possible multilinear rank among 
all tensors in a given format\footnote{If we allow for 
nonlinear coordinate flows \cite{Heng2021}, then we can 
of course map any multivariate probability density function (PDF) 
into a target distribution that has rank-one. Similar results 
can be obtained via optimal mass transport, e.g., by suitable approximations of 
the Kn\"{o}the-Rosenblatt rearrangement \cite{Rosenblatt,Marzouk_2018}.
}.   
However, developing a computationally 
tractable algorithm for obtaining 
the transformation $\bm H$ given the 
function $u(\bm x)$ is not an easy task. 
The main objective of this paper is to 
develop a mathematical 
framework and computationally efficient algorithms for 
obtaining {\em quasi-optimal} tensor rank-reducing coordinate 
transformations $\bm H$. In particular, we restrict our analysis to 
linear coordinate transformations. In this setting, $\bm H$ can 
be represented by a matrix $\bm A$, which allows us to 
write \eqref{nonlinear_coordinates} 
in the simplified form 
\begin{equation}
\label{ridge}
v(\bm x) = u(\bm A \bm x).
\end{equation}
The function $v(\bm x)$ is known as a generalized ridge 
function \cite{Pinkus_ridge}. 
If $u(\bm x)$ is represented in a tensor format, i.e., a series of 
products of one-dimensional functions $\psi(x_i)$ called 
tensor modes, then $v(\bm x)$ inherits a similar series expansion.
However, under the action of the linear map $\bm A\bm x$ the tensor 
modes are no longer univariate. Instead, they take the form of ridge 
functions $\psi(\bm a_i \cdot \bm x)$, 
where $\bm a_i$ is the $i$-th row of the matrix $\bm A$. 
Since these ridge tensor modes are now $d$-dimensional, 
the tensor compression which we had for 
$u(\bm x)$ may be lost. 
\begin{figure}[!t]
\centerline{\footnotesize\hspace{.3cm} (a)  \hspace{5.4cm}    (b) \hspace{5.7cm} (c) \hspace{0.1cm}}
\centerline{
\includegraphics[scale=0.3]{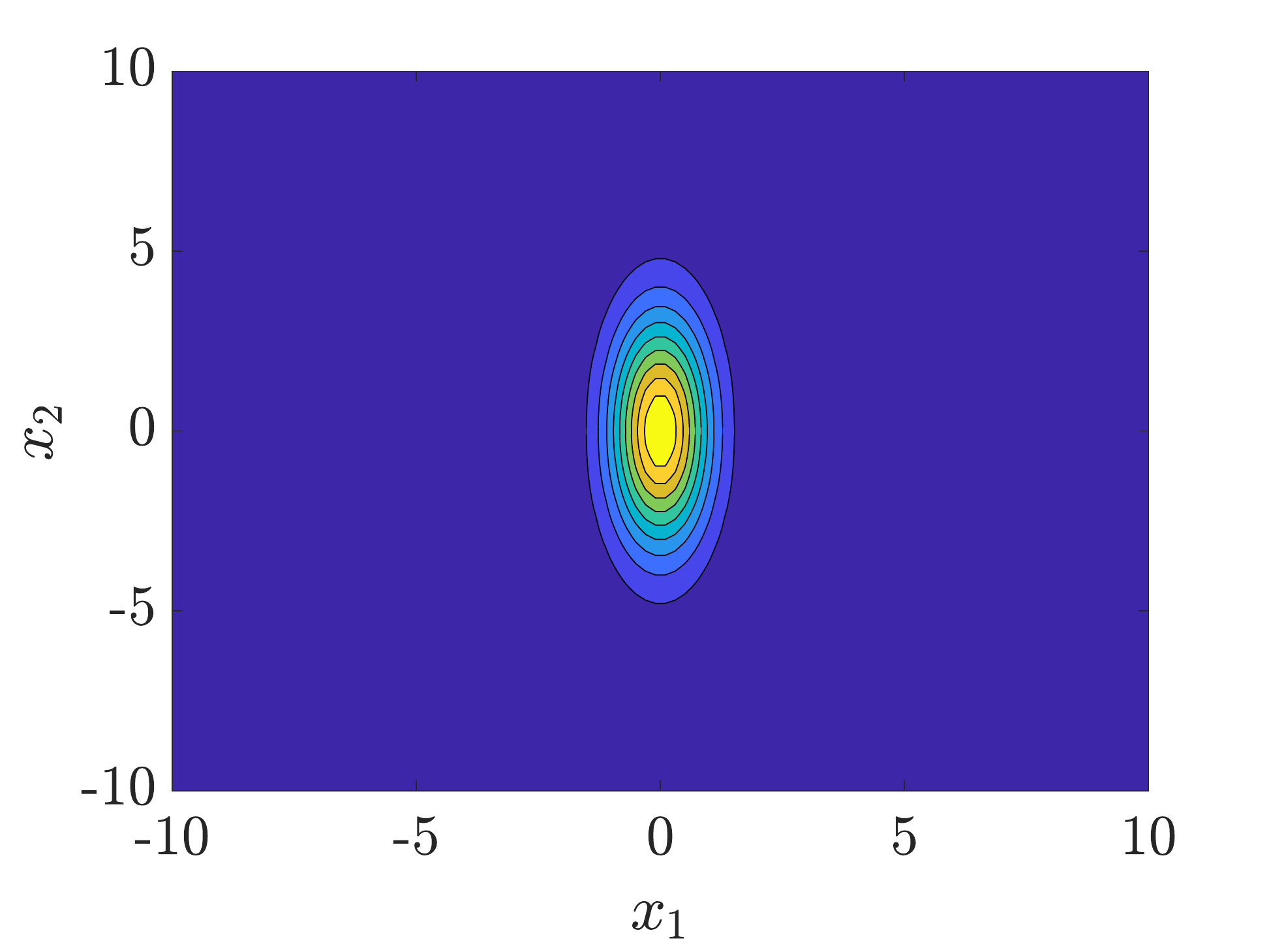}
\includegraphics[scale=0.3]{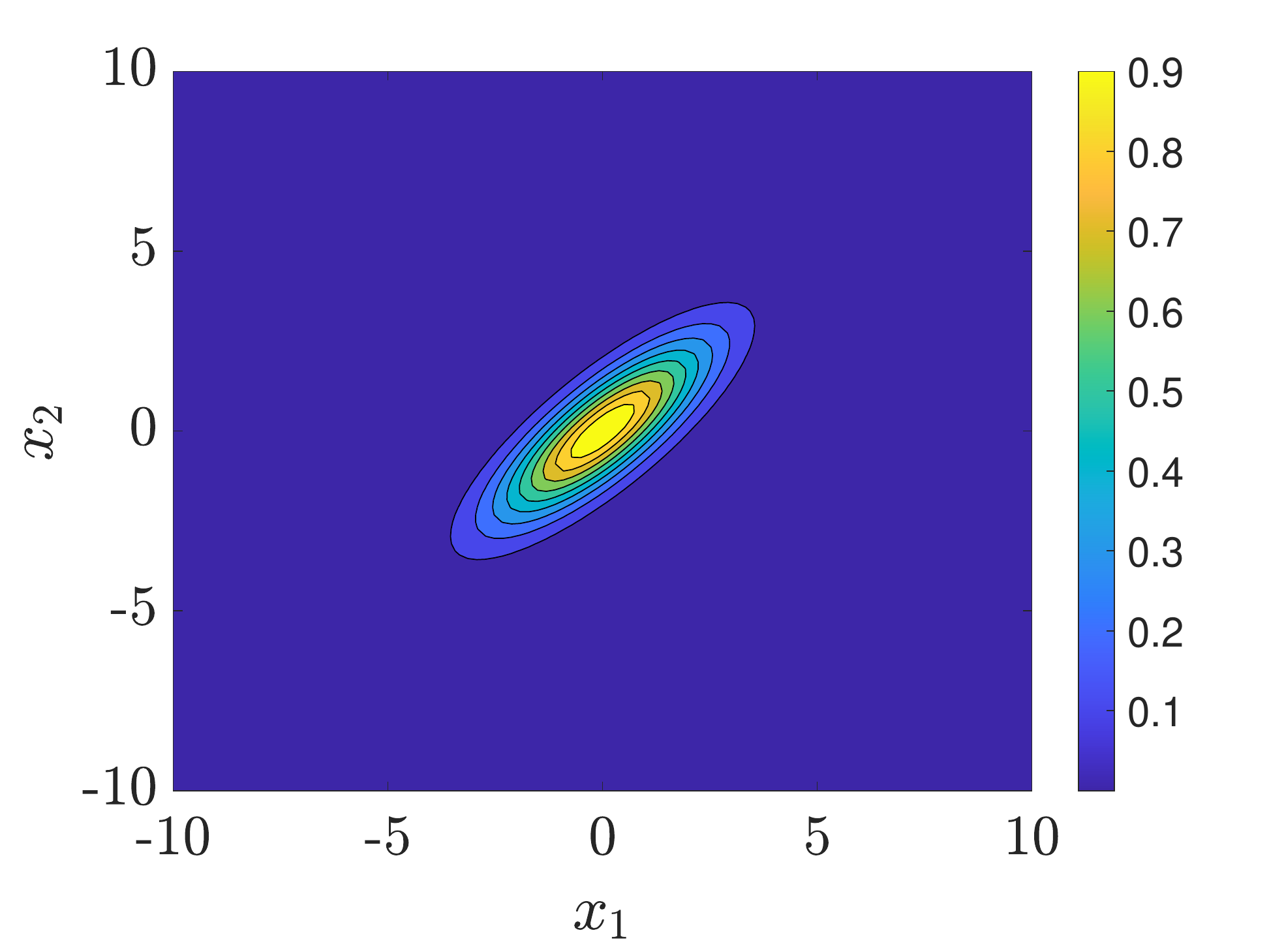}
\includegraphics[scale=0.3]{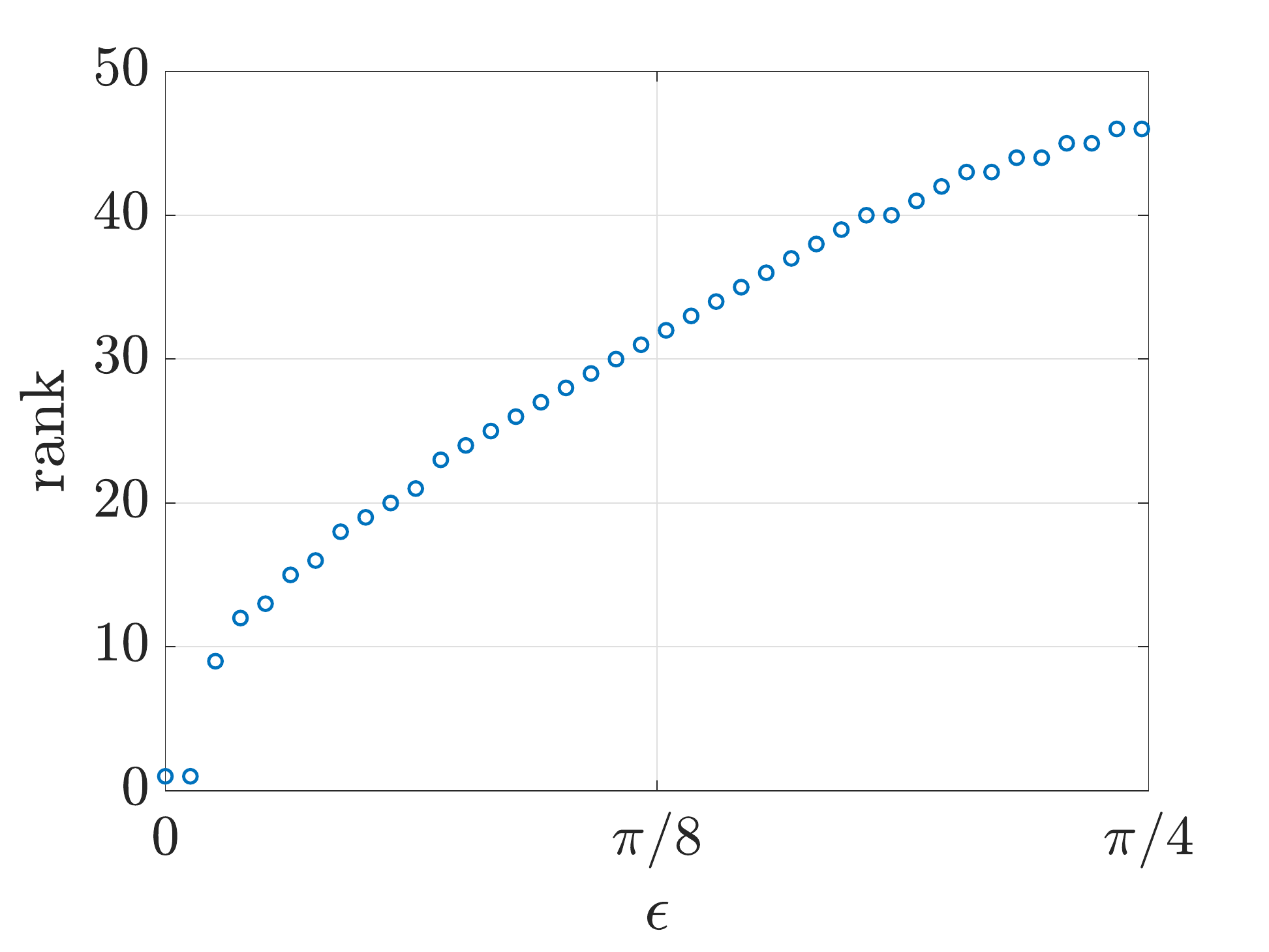}
}
\caption{(a) Two-dimensional Gaussian stretched 
along the $x_2$-axis; (b) Two-dimensional rotated Gaussian 
(clockwise rotation by $\epsilon = \pi/4$ radians);
(c) Rank of the the rotated Gaussian versus the 
clockwise rotation angle $\epsilon$. }
\label{fig:rotating_2D_example}
\end{figure}

Our contributions are as follows. 
First, we introduce a new class of functions that we refer to as 
{\em ridge tensors}, and provide a new method for evaluating $v(\bm x)$ 
without losing the tensor compression of $u(\bm x)$. 
Second, we develop a new Riemannian optimization algorithm 
for determining a linear map $\bm A$ so that 
$v(\bm x)$ has smaller tensor rank than $u(\bm x)$. 
Our approach is to parameterize the transformation $\bm A$ with 
a parameter $\epsilon\geq 0$, i.e., set $\bm A = \bm A(\epsilon)$, and 
consider $\bm A(\epsilon)$ as the flow map of an appropriate 
linear dynamical system.
From this dynamical system perspective we derive a partial 
differential equation (PDE) for the $\epsilon$-dependent 
function $v(\bm x;\epsilon) = u(\bm A(\epsilon) \bm x)$. 
We then integrate the PDE for $v(\bm x; \epsilon)$ 
using a step-truncation tensor integration scheme 
\cite{rodgers2020step-truncation,rodgers2022implicit,Vandereycken_2019} 
to obtain $v(\bm x;\epsilon)$ in a low-rank tensor format for 
all $\epsilon$. 
The quasi-optimal rank-reducing coordinate map $\bm A$ 
is obtained by minimizing a rank-related cost function 
using an appropriate path of transformations $\bm A(\epsilon)$. 
We construct such a path by giving the collection of 
linear transformations $\bm A(\epsilon)$ a Riemannian manifold 
structure and performing Riemannian gradient descent optimization, i.e., 
assigning to the derivative of the path $d \bm A(\epsilon)/d \epsilon$ 
a descent direction for the cost at each $\epsilon$ with respect to the 
chosen geometry. 
Our theoretical contributions provide a new 
framework that be generalized to larger classes of 
nonlinear transformations, and open new pathways for tensor 
rank reduction that are complementary to classical rank reduction 
methods, e.g., based on hierarchical SVD \cite{hsvd_tensors_grasedyk,Grasedyck2017}.   
Building upon our recent work on 
rank-adaptive tensor integrators 
\cite{adaptive_rank}, we also apply the new rank-reduction 
algorithm based on coordinate flows to linear PDEs. Specifically, 
we study advection problems in dimensions 
two, three and five and a reaction-diffusion equation 
in dimension two. 

This paper is organized as follows. 
In section \ref{sec:FTT} we briefly review the functional 
tensor train (FTT) expansion, which, will be the main
tensor format for demonstrating the proposed rank-reduction 
method based on coordinate flows. 
In section \ref{sec:ridge_funs} we introduce ridge tensor 
trains and provide a simple example to illustrate 
the effects of coordinate flows on tensor rank. 
In section \ref{sec:rank_reduction} we formulate 
the tensor rank reduction problem  
as a Riemannian optimization problem on the manifold of 
volume-preserving invertible coordinate 
maps. We also provide an efficient algorithm 
for computing the numerical solution to such optimization problem. 
In section \ref{sec:3Dfunction} we demonstrate the 
proposed Riemannian gradient descent 
algorithm for computing { 
quasi-optimal linear coordinate maps 
on a three-dimensional prototype function.}
In section \ref{sec:PDEs} we apply the proposed rank reduction methods to PDEs in dimensions 
two, three and five and a reaction-diffusion 
equation in dimension two. 
We also include two appendices in which we describe the
Riemannian manifold of coordinate flows, and provide 
theoretical results supporting the proposed mathematical 
framework for ridge tensors.

\section{Functional tensor train (FTT) expansion}
\label{sec:FTT}

The coordinate flow rank reduction technique 
we propose in this paper can be applied to any tensor format, 
in particular to the functional tensor train (FTT) format
\cite{OseledetsTT,Bigoni_2016,Dektor_dyn_approx}. 
In this section we briefly review the construction of 
FTT expansions for multivariate 
function belonging to the weighted Hilbert space 
\begin{equation} 
\label{hilbert_space}
H = L^2_{\mu}(\Omega).
\end{equation} 
Here, $\Omega \subseteq \mathbb{R}^d$ is a separable domain 
such as a $d$-dimensional flat torus $\mathbb{T}^d$ or a 
Cartesian product of $d$ subsets of the real 
line
\begin{equation}
\begin{aligned}
\Omega &= \bigtimes_{i=1}^d \Omega_i,\qquad \Omega_i\subseteq \mathbb{R}, 
\end{aligned}
\label{Omega}
\end{equation}
and $\mu$ is a 
finite product measure on 
$\Omega$
\begin{equation}
\mu(\bm x) = \prod_{i=1}^d \mu_i(x_i).
\end{equation}
As is well-known (e.g., \cite{OseledetsTT,Bigoni_2016,Dektor_dyn_approx}), 
each element $u \in L^2_{\mu}(\Omega)$ admits a FTT expansion of the form 
\begin{equation}
\label{FTT}
u(\bm x) = \sum_{\alpha_1,\ldots,\alpha_{d-1}=1}^{\infty} \sqrt{\lambda(\alpha_{d-1})}
\psi_1(1;x_1;\alpha_1) \psi_2(\alpha_1;x_2;\alpha_2) \cdots 
\psi_d(\alpha_{d-1};x_d;1),
\end{equation}
where $\{\psi_i(\alpha_{i-1};x_i;\alpha_i)\}_{\alpha_{i}}$ 
are orthonormal eigenfunctions of a hierarchy of 
compact self-adjoint operators. Orthonormality of 
$\{\psi_i(\alpha_{i-1};x_i;\alpha_i)\}_{\alpha_{i}}$ 
is relative to the inner products  
\begin{equation}
\begin{aligned}
    \sum_{\alpha_{i-1}=1}^\infty \int_{\Omega_i} 
    \psi_i(\alpha_{i-1}; x_i; \alpha_i)
    \psi_i(\alpha_{i-1}; x_i; \beta_i) d \mu_i(x_i) &=
     \delta_{\alpha_i \beta_i}, \qquad i = 1,\ldots, d-1, \qquad\alpha_0=1,\\
     \int_{\Omega_d} \psi_d(\alpha_{d-1};x_d;1) 
     \psi_d(\beta_{d-1};x_d;1) d \mu_d(x_d) &= 
      \delta_{\alpha_{d-1} \beta_{d-1}}.
\end{aligned}
\end{equation}
The sequence of positive real numbers 
$\left\{\sqrt{\lambda(\alpha_{d-1})}\right\}_{\alpha_{d-1}=1}^{\infty}$ appearing 
in \eqref{FTT} represents the product of all spectra of the 
compact self-adjoint operators mentioned above, 
and it has a single accumulation point at zero.
By truncating such spectra so that only the largest 
eigenvalues and corresponding eigenfunctions 
are retained, we obtain the following FTT approximation of $u(\bm x)$
\begin{equation}
\label{FTT_finite}
u_{\T}(\bm x) = \sum_{\alpha_0 = 1}^{r_0} 
\sum_{\alpha_1 = 1}^{r_1} \cdots  \sum_{\alpha_{d}=1}^{r_{d}} 
\sqrt{\lambda(\alpha_{d-1})}
\psi_1(\alpha_0;x_1;\alpha_1) \psi_2(\alpha_1;x_2;\alpha_2) \cdots 
\psi_d(\alpha_{d-1};x_d;\alpha_d),
\end{equation}
where $\bm r = (r_0, r_1, \ldots, r_{d-1}, r_d)$ is the 
tensor rank, $r_0=1$ and $r_d=1$.
At this point it is convenient to define the following
matrix-valued functions $\bm\Psi_i(x_i)$ (known as tensor cores) 
\begin{align}
\bm \Psi_1(x_1)=&
\begin{bmatrix}
\psi_1(1;x_1;1) & \cdots &
\psi_1(1;x_1;r_1) 
\end{bmatrix},\nonumber\\
\bm \Psi_i(x_i)=&
\begin{bmatrix}
\psi_i(1;x_i;1) & 
\cdots &
\psi_i(1;x_i;r_i)\\
\vdots & \ddots& \vdots\\  
\psi_i(r_{i-1};x_i;1) & 
\cdots &
\psi_i(r_{r-1};x_i;r_i)
\end{bmatrix}
\qquad i=2,\ldots, d-1,\nonumber\\
\bm \Psi_d(x_d)=&
\begin{bmatrix}
\psi_d(1;x_d;1) \\ \vdots \\
\psi_1(r_{d-1};x_d;1) 
\end{bmatrix},\nonumber
\end{align}
and write \eqref{FTT_finite} in a compact 
matrix product form 
\begin{equation}
\label{FTT_core_rep}
    u_{\T}(\bm x) = \bm \Psi_1(x_1) \bm \Psi_2(x_2) \cdots \sqrt{\bm \Lambda} \bm \Psi_d(x_d),
\end{equation}
where $\bm \Lambda$ is a diagonal matrix with entries 
$\lambda(\alpha_{d-1})$ ($\alpha_{d-1}=1,\ldots,r_{d-1}$). 
To simplify notation further, we will often suppress explicit 
tensor core dependence on the spatial variable $x_i$, allowing us 
to write $\bm\Psi_i=\bm\Psi_i(x_i)$ and $\psi_i(\alpha_{i-1},\alpha_i)=\psi_i(\alpha_{i-1};x_i;\alpha_i)$ as the spatial dependence is indicated 
by the tensor core subscript ``$i$''.

\section{Tensor ridge functions}
\label{sec:ridge_funs}
Let us introduce a new class of functions, which we call  
{\em tensor ridge functions}, that will be used in subsequent sections 
to build a tensor rank-reduction theory via linear coordinate mappings. 
To this end, consider the following invertible linear 
coordinate transformation
\begin{equation}
\label{x_to_y}
\bm y = \bm A \bm x, \qquad \bm A : \mathbb{R}^d \to \mathbb{R}^d. 
\end{equation}
If we evaluate a function $u \in L^2_{\mu}(\Omega)$ at 
$\bm y$, we obtain the generalized ridge function  
$u(\bm A \bm x)$ (e.g., \cite{Pinkus_ridge}).
Although the coordinate transformation $\bm A$ is linear, 
the evaluation of $u$ on $\bm A \bm x$ 
defines a {\em nonlinear map} of functions 
\begin{equation}
\label{evaluation_map}
\begin{aligned}
u(\bm x) &\mapsto v(\bm x) = u(\bm A \bm x).
\end{aligned}
\end{equation}
The image of a FTT tensor \eqref{FTT_finite} 
under such a map has the form 
\begin{equation}
\label{FTT_ridge}
\begin{aligned}
u_{\T}(\bm A \bm x)= \sum_{\alpha_0 = 1}^{r_0} 
\sum_{\alpha_1 = 1}^{r_1} \cdots  \sum_{\alpha_{d}=1}^{r_{d}} 
\sqrt{\lambda(\alpha_{d-1})}
\psi_1(\alpha_0;\bm a_1 \cdot \bm x;\alpha_1) 
\psi_2(\alpha_1;\bm a_2 \cdot \bm x;\alpha_2) \cdots 
\psi_d(\alpha_{d-1};\bm a_d \cdot \bm x ;\alpha_d),
\end{aligned}
\end{equation}
which we call a {\em tensor ridge function}. 
In \eqref{FTT_ridge}, $\bm a_i$ denotes the $i$-th row of the matrix $\bm A$ and ``$\cdot$'' is the 
standard dot product on $\mathbb{R}^d$, and $r_0=r_d=1$.
When we consider 
$u_\T(\bm A \bm x) = 
u_\T(\bm y)$ in coordinates 
$\bm y$, equation \eqref{FTT_ridge} is the 
FTT expansion for $u_\T(\bm y)$. 
However, when we consider $v(\bm x)=u_\T(\bm A \bm x)$ in 
coordinates $\bm x$ we have that each mode 
$\psi_i(\alpha_{i-1};\bm a_i \cdot \bm x;\alpha_i)$ in the 
tensor ridge function \eqref{FTT_ridge} 
is no longer 
a univariate function of $x_i$ as in \eqref{FTT_finite}, 
but rather a $d$-variate ridge function,
which, has the property of being constant 
in all directions orthogonal to the vector 
$\bm a_i$ (e.g., \cite{Constantine_2014,Pinkus_ridge}). 
An important problem is determining the FTT expansion 
\begin{equation}
\label{FTT_new}
v_\T(\bm x) = \sum_{\alpha_0 = 1}^{s_0} 
\sum_{\alpha_1 = 1}^{s_1} \cdots  \sum_{\alpha_{d}=1}^{s_{d}} 
\sqrt{\theta(\alpha_{d-1})}
\varphi_1(\alpha_0;x_1;\alpha_1) \varphi_2(\alpha_1;x_2;\alpha_2) \cdots 
\varphi_d(\alpha_{d-1};x_d;\alpha_d)
\end{equation}
given the FTT expansion \eqref{FTT_ridge} 
for $u_\T(\bm A \bm x)$.
A naive approach to solve this problem would be to 
recompute the FTT expansion from scratch using the methods 
of section \ref{sec:FTT}, i.e., treat $v_{\T}(\bm x)$ as 
a multivariate function and solve a sequence of 
hierarchical eigenvalue problems. 
This is not practical even for a moderate number of  
dimensions $d$ since the evaluation of $v(\bm x)$ requires 
constructing a tensor product grid in 
$d$-dimensions, and each eigenvalue problem 
requires the computation of $d$-dimensional integrals. 
Another approach is to use TT-cross approximation 
\cite{TT-cross}, which provides 
an algorithm for interpolating $d$-dimensional 
black-box tensors in the tensor train format 
with computationally complexity that scales 
linearly with the dimension $d$. 
Hereafter, we develop a new approach to compute 
the FTT expansion \eqref{FTT_new} from \eqref{FTT_ridge} 
based on coordinate flows.

\subsection{Computing tensor ridge functions via coordinate flows}
\label{sec:coord_flows}
Consider the non-autonomous linear dynamical system
\begin{equation}
\label{coordinate_dynamical_system}
\begin{cases}
\displaystyle\frac{d \bm y(\epsilon)}{d \epsilon} = 
\bm B (\epsilon) \bm y(\epsilon), \\
\bm y(0) = \bm x,
\end{cases}
\end{equation}
where $\bm y(\epsilon) \in \mathbb{R}^d$,
and $\bm B(\epsilon)$ is a given $d \times d$ matrix 
with real entries for all $\epsilon\geq 0$.
It is well-known that the solution 
to \eqref{coordinate_dynamical_system} 
can be written as 
\begin{equation}
\label{epsilon_coordinates}
\bm y(\epsilon) = \bm \Phi(\epsilon)\bm x,
\end{equation}
where 
\begin{equation}
\bm \Phi(\epsilon) = e^{\bm M(\epsilon)}
\end{equation} 
is an invertible linear mapping on $\mathbb{R}^d$ 
for each $\epsilon \geq 0$.
The matrix $\bm M(\epsilon)$ can be represented 
by the Magnus series (e.g., \cite{MSE}) 
\begin{equation}
\label{MSE}
\bm M(\epsilon) = \int_0^{\epsilon} \bm B(\epsilon_1) d \bm \epsilon_1 - \frac{1}{2	} \int_0^{\epsilon} \left\{ \int_0^{\epsilon_1} \bm B(\epsilon_2) d \epsilon_2, \bm B(\epsilon_1) \right\} d \epsilon_1 + \cdots ,
\end{equation}
where $\{ \cdot, \cdot\}$ denotes the matrix commutator
\begin{equation}
\{\bm P, \bm Q\} = \bm P \bm Q - \bm Q \bm P.
\end{equation} 
%
%
Now that we have introduced 
coordinate flows and their connection 
to linear coordinate transformations, let us 
consider the problem of determining 
the FTT expansion of 
$u_\T(\bm A \bm x)$ 
when $\bm A$ is generated by a coordinate 
flow, i.e., $\bm A = \bm \Phi(\epsilon)$ for 
some $\bm \Phi$ and some $\epsilon\geq 0$.
Differentiating $v(\bm x ; \epsilon) = u_\T(\bm \Phi(\epsilon) \bm x)$ 
with respect to $\epsilon$ yields the hyperbolic PDE 
\begin{equation}
\begin{aligned}
\displaystyle\frac{\partial v(\bm x;\epsilon)}{\partial \epsilon} & = \displaystyle\frac{\partial u_\T(\bm \Phi(\epsilon) \bm x)}{\partial \epsilon} \\
&= \nabla u_\T(\bm \Phi(\epsilon) \bm x) \cdot 
\left( \frac{\partial \bm \Phi(\epsilon)}{\partial \epsilon} \bm x\right) \\
&= \nabla u_\T(\bm \Phi(\epsilon)\bm x) \cdot \left( \bm B(\epsilon) \bm \Phi(\epsilon) \bm x \right),
\end{aligned}
\end{equation}
where in the second line we used the chain 
rule and in the third line we used equations 
\eqref{coordinate_dynamical_system} and 
\eqref{epsilon_coordinates}. 
Recalling $\bm \Phi(0) = \bm I_{d \times d}$ (identity matrix), 
we see that the initial state $v(\bm x;0) = u_\T(\bm x)$ is in 
FTT format. Thus, we have 
derived the following hyperbolic initial value 
problem for the tensor ridge function 
$v(\bm x; \epsilon)$ 
\begin{equation}
\label{PDE_for_tensor}
\begin{cases}
\displaystyle\frac{\partial v(\bm x;\epsilon)}{\partial \epsilon} =  \nabla v(\bm x;\epsilon)\cdot \left( \bm B(\epsilon) \bm \Phi(\epsilon)\bm x \right), \\
v(\bm x;0) = u_\T(\bm x),
\end{cases}
\end{equation}
with an initial condition that is given in an FTT format. 

Integrating the PDE \eqref{PDE_for_tensor} 
forward in $\epsilon$ on a FTT tensor manifold, e.g. using rank-adaptive 
step-truncation methods \cite{rodgers2020stability,rodgers2020step-truncation,Vandereycken_2019} or dynamic tensor approximation methods \cite{Dektor_dyn_approx,Dektor_2020,
adaptive_rank,Lubich_2015,Lubich_2010}, 
results in a FTT approximation $v_\T(\bm x; \epsilon)$ of the 
function $v(\bm x;\epsilon)$ for all $\epsilon \geq 0$. 
The computational cost of this 
approach for computing the FTT expansion 
of a FTT-ridge function is precisely the 
same cost as solving the hyperbolic PDE \eqref{PDE_for_tensor} 
in the FTT format, which 
in the case of step-truncation or 
dynamic approximation has computational 
complexity that scales linearly 
with $d$. Note that the accuracy of 
$v_\T(\bm x ; \epsilon)$ as an approximation 
of $v(\bm x; \epsilon)$ depends 
on the $\epsilon$ step-size and order of integration scheme 
used to solve the PDE \eqref{PDE_for_tensor}. 

Using coordinate flows, it is straightforward to 
compute the FTT expansion of a tensor ridge 
function $u_{\T}(\bm A\bm x)$ when the matrix $\bm A$ admits a 
real matrix logarithm $\bm L$. In this case, setting 
$\bm B(\epsilon)=\bm L$ yields 
$\bm \Phi(\epsilon) = e^{\epsilon \bm L}$, and therefore  
$\bm A = \bm \Phi(1)$. This means that we need 
to integrate \eqref{PDE_for_tensor} with $\bm B(\epsilon)=\bm L$ 
up to $\epsilon=1$ to obtain the FTT approximation of the 
tensor ridge function $u_{\T}(\bm A\bm x)$. Let us provide a 
simple example.

\paragraph{An example}
Consider the two-dimensional Gaussian function depicted in 
Figure \ref{fig:rotating_2D_example}(a), i.e., 
\begin{equation}
\label{2D_gaussian}
u_\T(\bm x) = e^{-x_1^2- x_2^2/10 }. 
\end{equation}
Clearly, \eqref{2D_gaussian} is the product 
of two univariate {functions} and therefore the 
 FTT tensor representation coincides with
\eqref{2D_gaussian} and has rank equal to one.
Next, consider a simple linear coordinate 
transformation $\bm \Phi(\epsilon)$ which rotates 
the $(x_1,x_2)$-plane by an angle of 
$\epsilon$ radians. 
It is well-known that 
\begin{equation}
\bm \Phi(\epsilon)=e^{\epsilon \bm L} ,
\label{rotation}
\end{equation}
where 
\begin{equation}
\label{2D_matrix_log}
\bm L =  
\begin{bmatrix}
0 & -1 \\
1 & 0
\end{bmatrix}
\end{equation}
is the infinitesimal generator of the two-dimensional rotation. The dynamical 
system \eqref{coordinate_dynamical_system} defining the coordinate 
flow $\bm y(\epsilon)=\bm \Phi(\epsilon) \bm x$ can be written as
\begin{equation}
\begin{cases}
\displaystyle \frac{d \bm y(\epsilon)}{d \epsilon} = \begin{bmatrix}
0 & -1 \\
1 & 0
\end{bmatrix} \bm y(\epsilon), \\
\bm y(0) = \bm x.
\end{cases}
\end{equation}
The tensor ridge function corresponding to the coordinate 
map $\bm \Phi(\epsilon)$ is given analytically by 
\begin{equation}
\label{rotated_gaussian}
\begin{aligned}
v(\bm x ; \epsilon) &= u_\T(\bm \Phi(\epsilon) \bm x) \\
&= e^{ -(\Phi_{11}(\epsilon) x_1 + \Phi_{12}(\epsilon) x_2)^2 } e^{-(\Phi_{21}(\epsilon) x_1 + \Phi_{22}(\epsilon) x_2)^2/10}. 
\end{aligned}
\end{equation}
The hyperbolic PDE \eqref{PDE_for_tensor}  for $v(\bm x; \epsilon)$
in this case is given by 
\begin{equation}
\label{PDE_for_tensor_rot}
\begin{cases}
\displaystyle\frac{\partial v(\bm x ; \epsilon)}{\partial \epsilon} = 
- x_2\displaystyle\frac{\partial v(\bm x ; \epsilon)}{\partial x_1}  + 
x_1\displaystyle\frac{\partial v(\bm x ; \epsilon)}{\partial x_2}, \\
v(\bm x ; 0) = u_\T( \bm x).
\end{cases}
\end{equation}
It is straightforward to verify that \eqref{rotated_gaussian} 
satisfies \eqref{PDE_for_tensor_rot}. 
Note that $v(\bm x; \epsilon)$ in 
\eqref{rotated_gaussian} 
is not an FTT tensor if 
$\epsilon \neq \pi k/2$ and $k\in \mathbb{N}$.
To compute the FTT representation of 
$v_\T(\bm x;\epsilon)$ we can solve the 
PDE \eqref{PDE_for_tensor_rot} on a tensor manifold 
using step-truncation or dynamic tensor approximation 
methods \cite{adaptive_rank,Dektor_dyn_approx,rodgers2020step-truncation,rodgers2022implicit}. Given 
the low dimensionality of the spatial domain in this 
example ($d=2$), we can also evaluate \eqref{rotated_gaussian} 
directly, and compute its FTT decomposition by solving 
an eigenvalue problem. 
In Figure \ref{fig:rotating_2D_example} (b) 
we provide a contour 
plot of $v(\bm x; \pi/4)$. 
To demonstrate the effect of rotations on tensor rank, 
in Figure \ref{fig:rotating_2D_example}(c) we plot 
the rank of $v(\bm x; \epsilon)$ versus $\epsilon$ for all 
$\epsilon\in[0,\pi/4]$. {Of} course such a plot can be mirrored 
to obtain the rank for $\epsilon\in[\pi/4,\pi/2]$, $[\pi/2,3\pi/4]$, and 
$[3\pi/4,\pi]$.

\section{Tensor rank reduction via coordinate flows}
\label{sec:rank_reduction}

The coordinate flows we introduced in the previous 
section can be used to morph a given function into 
another one that has a faster decay rate of FTT singular values, i.e., 
a FTT tensor with lower rank (after truncation). 
A simple example is the coordinate flow that rotates 
the Gaussian function in Figure \ref{fig:rotating_2D_example}(b) back to the 
rank-one state depicted in Figure \ref{fig:rotating_2D_example}(a).
This example and the examples documented in 
subsequent sections indicate that symmetry of the function relative to the new coordinate 
system plays an important role in reducing the tensor rank. 

The problem of tensor rank reduction via linear coordinate flows 
can be formulated as follows: how do we choose an invertible 
linear coordinate transformation $\bm A$ 
so that $v_\T(\bm x)\approx u_{\T}(\bm A\bm x)$ 
has smaller rank than $u_\T(\bm x)$ (eventually minimum rank)? 
The mathematical statement of this optimization problem is
\begin{equation}
\label{rank_min_problem}
\bm A= \argmin_{\bm A \in {\rm GL}_d(\mathbb{R})} {\rm rank} \left[ v_\T( \bm x) \right],
\end{equation}
where ${\rm GL}_d(\mathbb{R})$ denotes the 
set of $d \times d$ real invertible matrices, 
${\rm rank}[\cdot]$ is a metric 
related to the FTT rank, and $v_\T(\bm x)$ is a FTT 
approximation of $u_\T(\bm A\bm x)$.
In equation \eqref{rank_min_problem} we have purposely 
left the cost function unspecified, 
as some care must be taken in its 
definition to ensure that the optimization 
problem is both feasible and computationally 
effective in reducing rank. 
One possibility is to define 
${\rm rank}[v_\T(\bm x)]$ to return the 
sum of all $d$ entries of the 
multilinear rank vector $\bm r$ corresponding 
to the FTT tensor $v_\T(\bm x)$. 
While such a cost function is 
effective in measuring tensor rank it yields 
a NP-hard rank optimization problem \cite{non_con_nuclear}.

\subsection{Non-convex relaxation for the rank minimization problem}
\label{sec:non-convex-relaxation}
A common relaxation for rank minimization 
problems is to replace the rank cost 
function with the sum of the singular values.
To describe this relaxation in the context of FTT 
tensors we first recall that any FTT 
tensor $u_\T(\bm x)$ can be orthogonalized in the 
$i$-th variable as (e.g., \cite{adaptive_rank})
\begin{equation}
\label{orthog_at_i}
u_\T( \bm x) = \bm Q_{\leq i} \bm \Sigma_i \bm Q_{>i}, 
\end{equation}
where 
\begin{equation}
\label{multilinear_spec}
\bm \Sigma_i = {\rm diag}(\sigma_i(1), \sigma_i(2), \ldots, \sigma_i(r_i)),
\end{equation}
is a diagonal matrix with real entries (singular values of $u_\T$). 
The matrices $\bm Q_{\leq i}$ and $\bm Q_{>i}$ are 
defined as partial products
\begin{equation}
\bm Q_{\leq i} = \bm Q_1 \bm Q_2 \cdots \bm Q_i,\quad \quad
\bm Q_{>i} = \bm Q_{i+1} \bm Q_{i+2} \cdots \bm Q_d, 
\end{equation}
and they satisfy the orthogonality conditions 
\begin{equation}
\label{orthogonal_cores}
\left\langle \bm Q_{\leq i}^{\top} \bm Q_{\leq i} \right\rangle_{\leq i} 
= \bm I_{r_i \times r_i}, \qquad 
\left\langle \bm Q_{>i} \bm Q_{>i}^{\top} \right\rangle_{> i} 
= \bm I_{r_i \times r_i}. 
\end{equation}
Here, $\langle \cdot \rangle_{\leq i}$ and 
$\langle \cdot \rangle_{> i}$ are the averaging operators
\begin{equation}
\left\langle \bm W \right\rangle_{\leq i} (j,k) = 
\int_{\Omega_{\leq i}} w(j; \bm x; k) d \mu_{\leq i}(\bm x_{\leq i}), 
\qquad
\left\langle \bm W \right\rangle_{> i} (j,k) = 
\int_{\Omega_{>i}} w(j; \bm x; k) d \mu_{>i}(\bm x_{> i}),
\end{equation}
which map an arbitrary  $r_i \times r_i$ matrix-valued 
function $\bm W(\bm x)$  with entries 
$w(j; \bm x; k)$ into another $r_i \times r_i$ matrix-valued 
function depending on a smaller number of variables.  
Using the orthogonalization 
\eqref{orthog_at_i} for each $i = 1,2,\ldots,d-1$, 
we define the functions 
\begin{equation}
\label{multi_sings}
\begin{aligned}
S_i : L^2_{\mu}(\Omega) &\to \mathbb{R} \\
 u_\T(\bm x) &\mapsto \sum_{\alpha_i=1}^{r_i} \sigma_i(\alpha_i),
\end{aligned}
\end{equation}
which returns the sum of the singular 
values corresponding to the $i$-th 
component of the multilinear rank vector. 
Using these functions we define 
\begin{equation}
\label{sum_of_multi_sings}
\begin{aligned}
S : L^2_{\mu}(\Omega) &\to \mathbb{R} \\
 u_\T(\bm x) &\mapsto \sum_{i=1}^{d-1} \sum_{\alpha_i=1}^{r_i} \sigma_i(\alpha_i),
\end{aligned}
\end{equation}
which is a relaxation of the rank cost
function appearing in \eqref{rank_min_problem}. 
An analogous relaxation of the rank cost function 
called the matrix nuclear norm has been studied 
extensively for matrix rank minimization problems 
\cite{Fazel_2010,non_con_nuclear}. 
The function \eqref{sum_of_multi_sings} has also 
been used as a relaxation of ${\rm rank}[\cdot]$ 
in tensor completion \cite{tensor_comp_nuclear}.
Next, we proceed by selecting an appropriate 
search space for the rank cost function. 
The largest search space we may choose is 
${\rm GL}_d(\mathbb{R})$ as we have done in 
\eqref{rank_min_problem}. 
This, however, is not a good choice since transformations 
in ${\rm GL}_d(\mathbb{R})$ are not volume-preserving 
and hence do not preserve $L^2$ norm, i.e., 
$\| u_\T(\bm A \bm x) \|_{L^2_{\mu}} \neq \| u_\T(\bm x) \|_{L^2_{\mu}}$. 
Transformations which do not preserve the norm 
of $u_\T(\bm x)$ can reduce the rank cost function 
while having no impact on the tensor rank 
relative to the tensor's norm. 
Therefore we choose  ${\rm SL}_d(\mathbb{R})\subset {\rm GL}_d(\mathbb{R})$ 
as the search space, i.e., the collection of invertible matrices with determinant 
equal to one. These transformations are obviously volume-preserving, and therefore 
they preserve\footnote{It is straightforward to show with a simple 
change of variables that volume-preserving transformations 
preserve many quantities that are defined via an integral. 
For example, for any $u_\T(\bm x) \in L^2_{\mu}(\Omega)$ 
and $\bm A \in {\rm SL}_d(\mathbb{R})$ 
we have 
\begin{equation}
\| u_\T(\bm A \bm x) \|_{  L^p_{\mu}\left(\bm A^{-1}\Omega\right)} = \|u_\T(\bm x)\|_{L^p_{\mu}(\Omega)} , \qquad p=1,2.
\end{equation}} the $L^2$ norm of $u_\T$.

Since the domain of $S$ in \eqref{sum_of_multi_sings} is $L^2_{\mu}(\Omega)$ 
and the cost function must be defined on the search space ${\rm SL}_d(\mathbb{R})$, 
we define an evaluation map $E$ corresponding to $u_\T(\bm x)$ 
\begin{equation}
\label{eval_map}
\begin{aligned}
E : {\rm SL}_d(\mathbb{R}) &\to L^2_{\mu}(\Omega) \\
\bm A &\mapsto  u_\T(\bm A \bm x).
\end{aligned}
\end{equation}
Composing $E$ with $S$ yields the 
following (non-convex) relaxation of 
the cost function ${\rm rank}[\cdot]$ in \eqref{rank_min_problem} 
\begin{equation}
\label{relaxed_cost}
C = S \circ E : {\rm SL}_d(\mathbb{R}) \to \mathbb{R}.
\end{equation}
The function $C(\bm A)$ returns the sum of the singular values of the 
tensor ridge function $u_\T(\bm A\bm x)$, where $\bm A$ is 
an invertible matrix with determinant equal to one. 
The optimization problem corresponding to 
the cost function and search space discussed above 
is 
\begin{equation}
\label{relaxed_rank_min}
\bm A=\argmin_{\bm A  \in {\rm SL}_d(\mathbb{R})} C(\bm A) . 
\end{equation}
While \eqref{relaxed_rank_min} is simpler 
than \eqref{rank_min_problem}, it is still a computationally 
challenging problem for a number of reasons: First, it is 
non-convex. 
Second, when considered as 
a subset of all $d \times d$ matrices, the search space 
${\rm SL}_d(\mathbb{R})$ is subject to a non-trivial set 
of constraints, e.g., $\det(\bm A)=1$.
{  Third, we note that the set 
${\rm SL}_d(\mathbb{R})$ is unbounded 
(matrices with determinant $1$ can have 
entries that are arbitrarily large). This 
can yield convergence issues when looking for 
the minimizer \eqref{relaxed_rank_min}. 
In order to overcome this problem one may introduce 
linear inequality constraints on the entries 
of the matrix $\bm A$, i.e., optimize over bounded 
subsets of ${\rm SL}_d(\mathbb{R})$ such as the set of 
rotation matrices with determinant $1$.} 

Hereafter we develop a new method 
for obtaining a local minimum 
of \eqref{relaxed_rank_min}. 
To handle the search space constraints, we give 
${\rm SL}_d(\mathbb{R})$ a Riemannian 
manifold structure and perform gradient 
descent on this manifold {  \cite{optimization_on_matrices}}. 
To obtain the gradient of $C(\bm A)$ efficiently 
we build its computation into the FTT truncation 
procedure at a negligible additional computational cost. 

\subsection{Riemannian gradient descent path}
\label{sec:gradient_descent}

In Lemma \ref{lemma:continuity_of_cost} we 
show that if the FTT singular values of  
$u_\T(\bm A\bm x)$ are simple (i.e., distinct) 
then the cost function $C(\bm A)$ defined in \eqref{relaxed_cost} 
is differentiable in $\bm A$.
With this sufficient condition for the smoothness of the 
cost function established, we can use the machinery of Riemannian 
geometry summarized in \ref{app:Riemann} to construct a 
gradient descent path for the minimization of \eqref{relaxed_rank_min} 
on the search space ${\rm SL}_d(\mathbb{R})$. 
To this end, let us denote by $\bm \Gamma(\epsilon)$ such gradient descent path. 
To build $\bm \Gamma(\epsilon)$, we start at the identity 
$\bm \Gamma(0)=\bm I$ and consider the matrix ordinary differential equation
\begin{equation}
\label{descent_path_diffEQ}
\begin{cases}
\displaystyle \frac{d \bm \Gamma(\epsilon)}{d\epsilon} = 
-{\rm grad } \left[C \left(\bm \Gamma\left(\epsilon\right)\right)\right], \\
\bm \Gamma(0) = \bm I,
\end{cases}
\end{equation}
where ${\rm grad } \left[C \left(\bm \Gamma\left(\epsilon\right)\right)\right]$ is 
the {\em Riemannian gradient} of the cost function 
$C$ defined in \eqref{relaxed_cost}. 
By construction, the vector 
$-{\rm grad}\left[ C(\bm \Gamma(\epsilon)) \right]$ is  tangent to the manifold 
${\rm SL}_d(\mathbb{R})$ at each 
point $\bm \Gamma\left(\epsilon\right)$, 
and thus $\bm \Gamma(\epsilon) \in {\rm SL}_d(\mathbb{R})$ 
for all $\epsilon\geq 0$. 
Since $-{\rm grad }\left[C \left(\bm \Gamma\left(\epsilon\right)\right)\right]$ 
points in the direction of steepest descent of the cost function 
$C$ at the point $\bm \Gamma(\epsilon)$, 
the cost function is guaranteed to decrease 
along the path $\bm \Gamma(\epsilon)$ (or remain constant 
which implies we have obtained a local minimum). 
In Figure \ref{fig:gradient_descent} 
we provide an illustration of the 
Riemannian gradient descent path $\bm \Gamma(\epsilon)$. 
For computational efficiency, it is 
essential to have a fast method 
for computing the Riemannian gradient of the 
cost function $C$ at an arbitrary point 
$\bm A \in {\rm SL}_d(\mathbb{R})$. 
The following Proposition provides an 
expression for such Riemannian gradient in terms of 
orthogonal FTT cores which we will 
use to efficiently compute the descent 
direction. The proof is given in \ref{sec:appdx_proof}.

\begin{proposition}
\label{prop:riemannian_gradient}
The Riemannian gradient of the cost function \eqref{relaxed_cost} 
at the point $\bm A \in {\rm SL}_d(\mathbb{R})$ is given by 
\begin{equation}
\label{gradient_of_cost}
{\rm grad }\left[C \left(\bm A\right)\right] = \bm D \bm A,
\end{equation}
where 
\begin{equation}
\label{descent_generator}
{\bm D} = \sum_{i=1}^{d-1} \int_{\Omega} \bm Q_{\leq i} \bm Q_{>i} \left( \nabla v(\bm x)  \left( \bm A \bm x\right)^{\top} - \frac{\nabla v(\bm x)^{\top} \bm A \bm x}{d} \bm I_{d \times d} \right) d \mu(\bm x) ,
\end{equation}
$v(\bm x) = u_\T(\bm A \bm x)$ and $\bm Q_{\leq i}, \bm Q_{>i}$ are tensor 
cores of the orthogonalized FTT 
\begin{equation}
v(\bm x) = \bm Q_{\leq i} \bm \Sigma_i \bm Q_{>i}.
\end{equation} 
\end{proposition}
\begin{figure}[!t]
\centerline{\footnotesize\hspace{0.1cm}  (a) \hspace{6.7cm} (b)  }
\centerline{\hspace{1cm}
\includegraphics[scale=0.4]{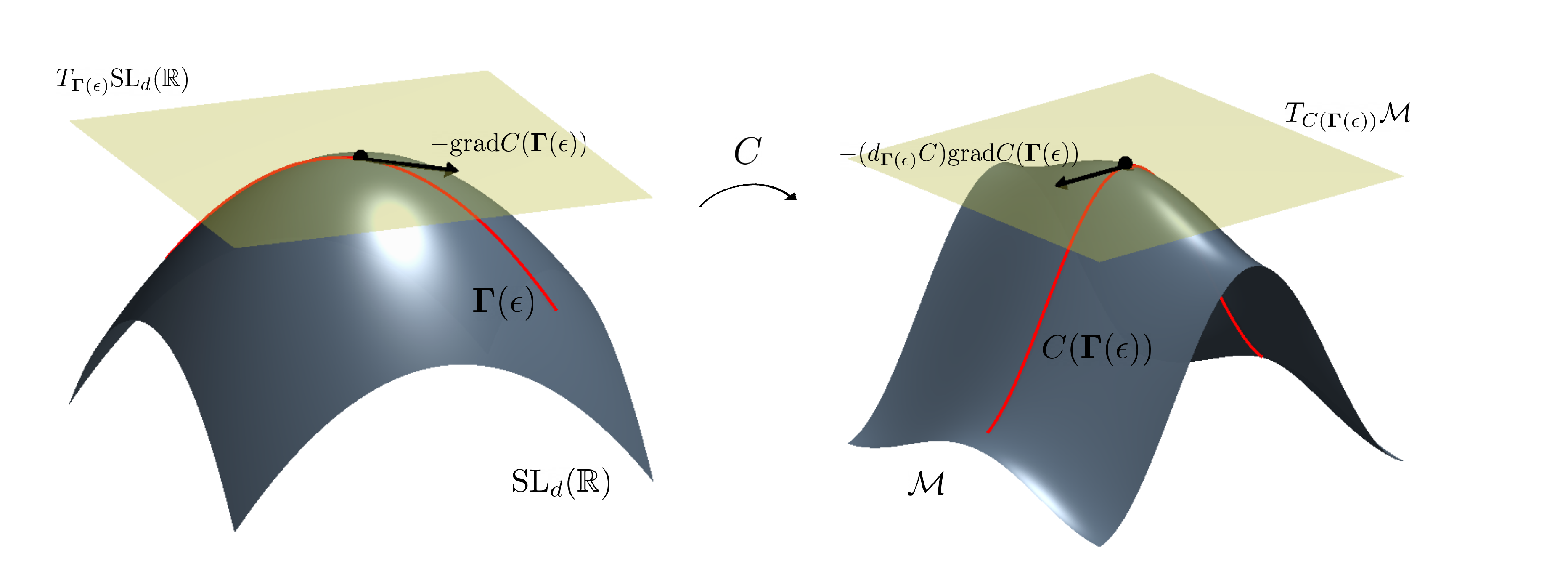}
}
\caption{An illustration of 
a descent path $\bm \Gamma(\epsilon)$ 
for the cost function $C$ defined in 
\eqref{relaxed_cost}. 
At each point of the 
path 
$\bm \Gamma(\epsilon) \in {\rm SL}_d(\mathbb{R})$ 
we assign the tangent 
vector $-{\rm grad} \left[C( \bm \Gamma(\epsilon))\right] \in 
T_{\bm \Gamma(\epsilon)} {\rm SL}_d(\mathbb{R})$ 
which points in the direction 
of steepest descent 
$-(d_{\bm \Gamma(\epsilon)} C ) 
{\rm grad} \left[C(\bm \Gamma(\epsilon))\right] \in 
T_{C(\bm \Gamma(\epsilon))} \mathcal{M}$ 
of the function 
$C$ at the point $C(\bm \Gamma(\epsilon)) \in \mathcal{M}$.}
\label{fig:gradient_descent}
\end{figure}

\noindent
Using the gradient descent path defined 
by \eqref{descent_path_diffEQ} we differentiate 
the coordinate transformation 
\begin{equation}
\label{descent_coordinates}
\bm y(\epsilon) = \bm \Gamma(\epsilon) \bm x
\end{equation}
with respect to $\epsilon$ and 
use \eqref{descent_path_diffEQ}-\eqref{gradient_of_cost} to 
obtain 
\begin{equation}
\begin{cases}
\label{coordinate_descent_ODE}
\displaystyle\frac{d \bm y(\epsilon)}{d \epsilon} = -\bm D(\epsilon) \bm y(\epsilon), \\
\bm y(0) = \bm x.
\end{cases}
\end{equation}
Note that \eqref{coordinate_descent_ODE} has the same 
form as the ODE \eqref{coordinate_dynamical_system} 
we used to define coordinate flows. 
Hence, by construction, the flow map generated by 
\eqref{coordinate_descent_ODE} 
is the descent path $\bm \Gamma(\epsilon)$ 
on the manifold ${\rm SL}_d(\mathbb{R})$, 
which, converges to a local 
minimum of $C$ as 
$\epsilon$ increases. 
By defining the function 
$v(\bm x; \epsilon) = 
u_\T( \bm \Gamma(\epsilon) \bm x)$ 
and differentiating it with respect to 
$\epsilon$ we obtain the hyperbolic PDE 
\begin{equation}
\label{PDE_for_tensor_descent}
\begin{cases}
\displaystyle\frac{\partial v(\bm x ; \epsilon)}{\partial \epsilon} = -
\nabla v(\bm x;\epsilon) \cdot \left[\bm D(\epsilon) 
\bm \Gamma(\epsilon) \bm x \right], \\
v(\bm x ;0) = u_\T(\bm x).
\end{cases}
\end{equation}
Note that the evolution of $\bm y(\epsilon)=\bm \Gamma(\epsilon)\bm x$ 
at the right hand side of  \eqref{PDE_for_tensor_descent} is 
defined by \eqref{coordinate_descent_ODE}.
Integrating \eqref{coordinate_descent_ODE}-\eqref{PDE_for_tensor_descent} forward in $\epsilon$
yields a rank-reducing linear coordinate transformation $\bm \Gamma(\epsilon)$ 
and the reduced rank function $v(\bm x; \epsilon) = u_\T\left(\bm \Gamma(\epsilon) \bm x\right)$.

\subsection{Numerical integration of the gradient descent equations}

It is convenient to use a step-truncation method 
\cite{rodgers2020step-truncation,rodgers2022implicit,Vandereycken_2019} 
to integrate the initial value problem  
\eqref{PDE_for_tensor_descent} on a FTT tensor manifold. 
This is because applying the FTT truncation operation to 
$v(\bm x ;\epsilon)$ requires computing 
the orthogonalized tensor cores 
$\bm Q_{\leq i}(\epsilon), \bm Q_{>i}(\epsilon)$ $(i = 1,2,\ldots,d-1)$ 
which can then be readily used to evaluate the matrix $\bm D(\epsilon)$ 
defining the Riemannian gradient \eqref{descent_generator}. 
Moreover, the FTT truncation operation applied to $v(\bm x ;\epsilon)$ 
yields an expansion of the form \eqref{FTT_new}, which is the desired 
tensor format.  
To describe the integration algorithm in more detail, 
let us discretize the interval $[0, \epsilon_f]$ into $N+1$ 
evenly-spaced points\footnote{The right end-point 
$\epsilon_f$ will ultimately be determined 
by the stopping criterion for gradient descent.}
\begin{equation}
\label{discrete_epsilon}
\epsilon_i = i \Delta \epsilon, \qquad 
\Delta \epsilon = \frac{\epsilon_f}{N}, \qquad i=0,1,\ldots,N, 
\end{equation}
and let $v_i, \bm \Gamma_i, \bm D_i$ denote 
$v(\bm x;\epsilon_i), \bm \Gamma(\epsilon_i), \bm D( \epsilon_i)$, respectively.
Let 
\begin{align}
\label{one_step_scheme}
    v_{i+1} &= v_i + \Delta \epsilon \Phi\left(v_i,\bm D_i, \Delta \epsilon \right),\\
    \label{discrete_matrix_ODE}
\bm \Gamma_{i+1} &= \bm \Gamma_i + \Delta \epsilon \widehat{\bm \Phi} 
\left( \bm \Gamma_i , \bm D_i, \Delta \epsilon \right)
\end{align} 
be one-step explicit integration schemes approximating the 
solution to the initial value problem \eqref{PDE_for_tensor_descent} 
and the matrix ODE \eqref{descent_path_diffEQ}, respectively. 
In order to guarantee that the solution 
$v_{i+1}$ is a low-rank FTT tensor, 
we apply a truncation 
operator to the right hand 
side of \eqref{one_step_scheme}.
This yields the step-truncation method \cite{rodgers2020step-truncation}
\begin{equation}
\label{step_truncation_scheme}
v_{i+1} = \mathfrak{T}_{\delta}
\left( v_i + 
\Delta \epsilon \Phi\left(v_i,\bm D_i, \Delta \epsilon \right) \right). 
\end{equation}
Here, $\mathfrak{T}_{\delta}$ denotes 
the standard FTT truncation operator 
with relative accuracy $\delta$ proposed 
in \cite{OseledetsTT} 
modified to return the matrix $\bm D_{i+1}$. 
A detailed description of the modified 
tensor truncation algorithm is provided 
in  section \ref{sec:modified_truncation}. 

At this point, a few remarks regarding the integration scheme 
\eqref{one_step_scheme}-\eqref{step_truncation_scheme} are in order. 
First, we notice that at each step of the 
gradient descent algorithm we are computing 
the gradient \eqref{gradient_of_cost} at the 
identity matrix since the current tensor $v_i$ 
is the tensor ridge function 
$v_i(\bm x) = u_\T(\bm \Gamma_i \bm x)$.
Second, the accuracy of the 
tensor $v_i$ approximating 
$u_\T(\bm \Gamma_i \bm x)$ is determined by the 
chosen integration scheme and its relevant 
parameters 
(i.e., the function 
$\Phi$, the step size $\Delta \epsilon$, and 
the accuracy of $\mathfrak{T}_{\delta}$). 
In a standard gradient descent algorithm, 
convergence can be expedited with a 
line-search routine that determines 
an appropriate step-size to take 
in the descent direction. 
However, in the proposed integration scheme 
the step-size determines the accuracy of the 
final tensor, thus we keep the step-size 
$\Delta \epsilon$ fixed during gradient descent. 
We set a stopping criterion for the integration of 
\eqref{one_step_scheme}-\eqref{step_truncation_scheme} 
based on the empirical observation that the cost 
function $C$ does not decrease substantially along the descent 
path $\bm \Gamma(\epsilon)$ when $\bm \Gamma_{i}$ is close 
to a local minimum. Mathematically, this translates 
into the condition
\begin{equation}
\label{stopping_criterion_1}
\frac{d S(v_i)}{d \epsilon}  >
- {  \eta},
\end{equation}
where ${  \eta}$ is some predetermined 
tolerance. 
Since the solution history $v_i,v_{i-1},\ldots$ 
is available during gradient descent, a simple 
method for approximating $dS(v_i)/d\epsilon $ 
is a $p$-point backwards difference stencil 
\begin{equation}
\frac{d S(v_i)}{d \epsilon} \approx {\rm BD}^{(p)}(S(v_{i}),{  S(v_{i-1})},\ldots,S(v_{i-p})).
\label{BD}
\end{equation}
In addition to the stopping criterion 
\eqref{stopping_criterion_1}, we also 
set a maximum number of iterations $M_{\rm iter}$ 
to ensure that the Riemannian 
gradient descent algorithm halts 
within a reasonable amount of time.
We summarize the proposed Riemannian gradient descent method to compute 
a local minimum of \eqref{relaxed_rank_min} 
in Algorithm \ref{alg:gradient_descent}.

\vs
\begin{algorithm}[!t]
\SetAlgoLined
 \caption{Riemannian gradient descent for computing rank-reducing linear coordinate maps.}
\label{alg:gradient_descent}
\vspace{0.3cm}
 \KwIn{\\
  $u_\T$ $\rightarrow$ initial FTT tensor, \\
  $\Delta \epsilon$ $\rightarrow$ step-size for gradient descent, \\
  ${  \eta}$ $\rightarrow$ stopping tolerance, \\
  $M_{\rm iter}$ $\rightarrow$ maximum number of iterations.
 }
 \vspace{0.3cm}
 \KwOut{\\
   $\bm \Gamma$ $\rightarrow$ rank-reducing linear coordinate transformation, \\
   $v_\T$ $\rightarrow$ reduced rank FTT tensor on transformed coordinates 
   $v_\T(\bm x) = u_\T(\bm \Gamma \bm x)$ .}
   \vspace{0.3cm}
 { \bf Runtime:}\\
\begin{itemize}
\item[]   $[v_0, S(v_0),\bm D_0] = \mathfrak{T}_{\delta} (u_\T)$,\\
   $\bm \Gamma_0 = \bm I$, \\
   $\dot{S}(v_0) = -\infty$, \\ 
   $i = 0$. \\
   \vspace{0.3cm}
{{\bf while} $\dot{S}(v_i) < -{  \eta}$ {\bf and} $i \leq M_{\rm iter}$}
\begin{itemize}
\item[]$v_{i+1} = v_i + \Delta \epsilon \Phi(v_i, \bm D_i, \Delta \epsilon)$, 
\item []$[v_{i+1},S(v_{i+1}),\bm D_{i+1}] = \mathfrak{T}_{\delta}(v_{i+1})$, 
\item []$\bm \Gamma_{i+1} = \bm \Gamma_i + \Delta \epsilon \widehat{\bm \Phi}(\bm \Gamma_i, \bm D_i, \Delta \epsilon)$,
\item []$\dot{S}(v_{i+1}) = {\rm BD}^{(p)}(S(v_{i+1}),S(v_i),\ldots,S(v_{i+1-p}))$, 
\item[] $i = i+1$,
\item[]$\bm \Gamma = \bm \Gamma_i$, 
\item []$v_\T = v_i$.
\end{itemize}
{\bf end}
\end{itemize}
\end{algorithm}
\vs

\subsection{Modified tensor truncation algorithm}
\label{sec:modified_truncation}
\begin{figure}[!t]
\centerline{\hspace{1cm}
\includegraphics[scale=0.42]{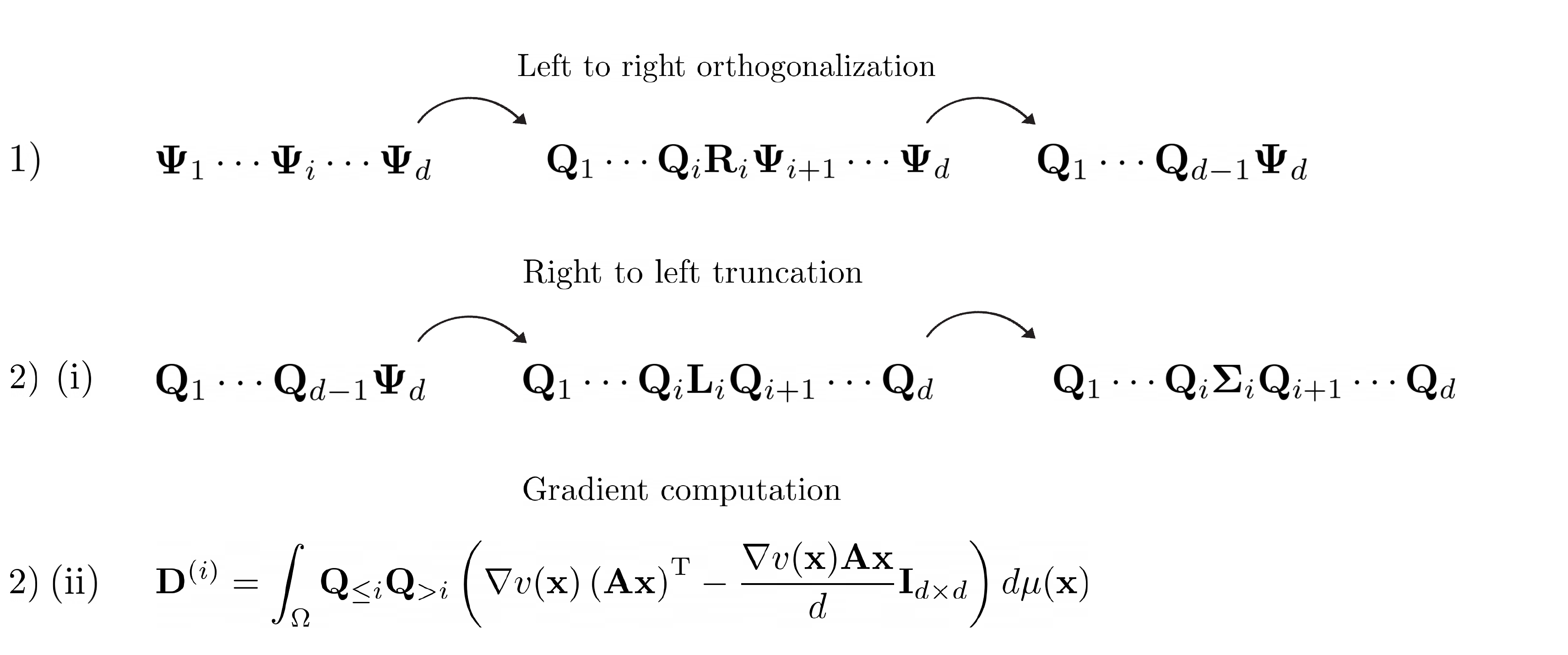}
}
\caption{A summary of the modified FTT truncation algorithm  for computing the Riemannian gradient \eqref{descent_generator}.}
\label{fig:truncation_summary}
\end{figure}

\begin{algorithm}[!t]
\DontPrintSemicolon
\SetAlgoLined
 \caption{Modified FTT truncation algorithm.}
\label{alg:truncation_and_gradient}
\vspace{.3cm}
 \KwIn{\\
 $v$ $\rightarrow$ FTT tensor with cores 
 $\bm \Psi_1,\bm \Psi_2 ,\ldots, \bm \Psi_d$,\\
  $\delta$ $\rightarrow$ desired accuracy.
 }
\vspace{.3cm}
 \KwOut{ \\
$v_\T$ $\rightarrow$ truncated FTT tensor satisfying $\| v - v_\T \|_{L^2_{\mu}(\Omega)} \leq \delta \|v\|_{L^2_{\mu}(\Omega)}$, \\ 
$S\left(v_\T\right)$ $\rightarrow$ sum of all multilinear singular values of $v_\T$, \\
$\bm D$ $\rightarrow$ left factor of the Riemannian gradient \eqref{gradient_of_cost}.
}
\vspace{.3cm}
{\bf Runtime:} 
\begin{itemize}
\item []$\hat{\delta} = \displaystyle\frac{\delta}{\sqrt{d-1}} \|v\|_{L^2}$\hspace{1cm} (Set truncation parameter) 

\item [] $S(v) = 0$ \hspace{2.5cm}(Initialize $S(v)$)
\item [] {\bf for  $i=1$ to $d-1$}  \hspace{1cm} (Left-to-right orthogonalization)
\begin{itemize}
\item [] $[\bm Q_i , \bm R_i] = {\rm QR}(\bm \Psi_i)$
\item [] $\bm \Psi_{i+1} = \bm R_i \bm \Psi_{i+1}$ \\
\end{itemize}
{\bf end}
  \vspace{0.1cm}
  \item [] {\bf for  $i=d$ to $2$} \hspace{1cm} (Right-to-left truncation and gradient computation)
 \begin{itemize}
\item [] $[\bm L_i , \bm Q_i] = {\rm LQ}(\bm \Psi_i)$ 
\item [] $[\bm U_i, \bm \Sigma_i, \bm V_i] = {\rm SVD}_{\hat{\delta}}(\bm L_i)$
\item [] $\bm Q_{i} = \bm V_i^{\top} \bm Q_{i} $ 
\item [] $\bm Q_{i-1} = \bm Q_{i-1} \bm U_i$, 
\item [] $S(v) = S(v) + {\rm sum}\left(\Sigma_i \right)$, 
\item [] $\bm D^{(i-1)} = \int \bm Q_{\leq i-1}  \bm Q_{>i-1}\left( \nabla v(\bm x) \bm x^{\top} \right) d \mu(\bm x)$
\end{itemize}
{\bf end}
\item [] $\bm D = \displaystyle\sum_{i=1}^{d-1} \bm D^{(i)}$. \\
\item [] $v_\T = \bm Q_1 \bm \Sigma_2 \bm Q_2 \ldots \bm Q_d$.
 \end{itemize}
\end{algorithm}
To speed up numerical integration of the gradient descent equations
\eqref{discrete_matrix_ODE}-\eqref{step_truncation_scheme} it is 
convenient to compute the matrix $\bm D_i$ defined in \eqref{descent_generator} during 
FTT truncation. The reason being that standard FTT truncation requires 
the computation of all orthogonal FTT cores $\bm Q_{\leq j}$ and $\bm Q_{>j}$, which, 
can be readily used to compute $\bm D_i$.
To describe the modified tensor truncation algorithm, 
let $v_{\T}(\bm x) = \bm \Psi_1 \cdots \bm \Psi_d$ be an 
FTT tensor with non-optimized rank, e.g., $v_{\T}(\bm x)$ is 
the result of adding two FTT tensors together. 
Denote by $\mathfrak{T}_{\delta}$ the modified truncation operator, 
where $\delta$ is the required relative accuracy, i.e., 
\begin{equation}
\|v_{\T}(\bm x) - \mathfrak{T}_{\delta} (v_{\T}(\bm x)) \|_{L^2_{\mu}(\Omega)} \leq 
\delta \| v_{\T}(\bm x) \|_{L^2_{\mu}(\Omega)}. 
\end{equation} 
We also define 
\begin{equation}
\hat{\delta} = \frac{\delta}{\sqrt{d-1}} 
\| v_{\T}(\bm x) \|_{L^2_{\mu}(\Omega)},
\end{equation}
which is the required accuracy for each 
SVD in the FTT truncation algorithm. 
As described in \cite{adaptive_rank}, we may perform a 
functional analogue of the QR decomposition on the 
FTT tensor cores of $v_{\T}(\bm x)$
\begin{equation}
\label{functional_QR}
\bm \Psi_i = \bm Q_i \bm R_i,
\end{equation}
where $\bm Q_i$ is a $r_{i-1} \times r_i$ matrix with elements 
in $L^2_{\mu_i}(\Omega_i)$ satisfying 
$\left\langle \bm Q_i^{\top}\bm Q_i \right\rangle_i = 
\bm I_{r_{i} \times r_i}$, and $\bm R_i$ 
is an upper triangular $r_i \times r_i$ matrix with real entries.
Similarly, we can perform an LQ-factorization 
\begin{equation}
\bm \Psi_i = \bm L_i \bm Q_i,
\end{equation}
where $\bm Q_i$ is a $r_{i-1} \times r_i$ matrix with elements 
in $L^2_{\mu_i}(\Omega_i)$ satisfying 
$\left\langle \bm Q_i^{\top}\bm Q_i \right\rangle_i = 
\bm I_{r_{i} \times r_i}$, and $\bm L_i$ 
is a lower-triangular $r_i \times r_i$ matrix with real entries.
The first procedure in the modified 
truncation routine is a 
left-to-right orthogonalization sweep in which 
we first perform the QR decomposition 
\begin{equation}
\label{functional_QR_d}
\bm \Psi_1 = \bm Q_1 \bm R_1, 
\end{equation}
and then update the core $\bm \Psi_2$ 
\begin{equation}
\bm \Psi_2 = \bm R_1 \bm \Psi_2.
\end{equation}
This process is repeated recursively 
\begin{equation}
\begin{aligned}
\bm \Psi_i = \bm Q_i \bm R_i, \qquad 
\bm \Psi_{i+1} = \bm R_i \bm \Psi_{i+1}, \qquad i = 2,\ldots, d-1,
\end{aligned}
\end{equation}
resulting in the orthogonalization 
\begin{equation}
\label{left_orthog}
v_{\T}(\bm x) = \bm Q_1 \bm Q_2 \cdots \bm Q_{d-1} \bm \Psi_d.
\end{equation}
Next, we perform a right-to-left sweep which 
compresses $v_\T(\bm x)$ 
and simultaneously computes 
each term appearing in the summation of $\bm D$ in equation 
\eqref{descent_generator}. 
The first step of this procedure is to compute a 
LQ decomposition of $\bm \Psi_d$ 
\begin{equation}
\bm \Psi_d = \bm L_d \bm Q_d,
\label{LQ}
\end{equation}
and then perform a truncated singular value decomposition of $\bm L_d$ 
with threshold $\hat{\delta}$ 
\begin{equation}
\bm L_d = \bm U_d \bm \Sigma_d \bm V_d^{\top}. 
\label{Lqq}
\end{equation}
Substituting \eqref{LQ} and \eqref{Lqq}  into \eqref{left_orthog} yields
\begin{equation}
v_{\T}(\bm x) = \bm Q_1 \bm Q_2 \cdots \bm Q_{d-1} \bm \Sigma_d \bm Q_d,
\end{equation}
where we re-defined 
\begin{equation}
\label{truncated_d}
\bm Q_{d-1} = \bm Q_{d-1} \bm U_d, \qquad \bm Q_d = \bm V_d^{\top} \bm Q_d.
\end{equation}
At this point we have truncated the FTT tensor $v_\T(\bm x)$  in the $d$-{th} variable. 
The expansion \eqref{truncated_d} provides the FTT orthogonalization needed 
to compute the $(d-1)$-th term in the sum \eqref{descent_generator} 
\begin{equation}
\label{D_d-1}
\bm D^{(d-1)} =  \int_{\Omega} \bm Q_{\leq d-1} \bm Q_{>d} \left( \nabla v_{\T}(\bm x) 
\bm x ^{\top} \right)   d \mu(\bm x),
\end{equation}
which, can be computed efficiently by applying one-dimensional differentiation matrices 
and quadrature rules to the FTT cores.
We proceed in a recursive manner ($i = d-1,\ldots, 2$) 
with the same steps described above for $\bm \Psi_d$. 
First compute the LQ decomposition 
\begin{equation}
\left( \bm Q_i \bm \Sigma_{i+1} \right) = \bm L_i \bm Q_i, 
\end{equation}
and then perform a singular value decomposition 
with threshold $\hat{\delta}$ 
\begin{equation}
\bm L_i = \bm U_i \bm \Sigma_i \bm V_i^{\top}. 
\end{equation}
Then rewrite $v_\T(\bm x)$ as 
\begin{equation}
\label{orthog_at_i_2}
v_{\T}(\bm x) = \bm Q_1 \bm Q_2 \cdots \bm Q_{i-1} \bm \Sigma_i \bm Q_i \cdots \bm Q_d,
\end{equation}
where we re-defined 
\begin{equation}
\label{truncated_i}
\bm Q_{i-1} = \bm Q_{i-1} \bm U_i, \qquad \bm Q_i = \bm V_i^{\top} \bm Q_i.
\end{equation}
The expansion \eqref{orthog_at_i_2} provides 
the FTT orthogonalization required 
to compute the $(i-1)$-th term in the sum \eqref{descent_generator} 
\begin{equation}
\bm D^{(i-1)} =  \int_{\Omega} \bm Q_{\leq i-1} \bm Q_{>i} \left( \nabla v_{\T}(\bm x) \bm x ^{\top} \right)   d \mu(\bm x),
\end{equation}
which, can be computed efficiently 
by applying one-dimensional 
differentiation matrices 
and quadrature rules to the FTT cores.
Finally with all of the terms in \eqref{descent_generator} computed we simply 
sum them to obtain the matrix 
\begin{equation}
\bm D = \displaystyle \sum_{i=1}^{d-1} \bm D^{(i)}.
\end{equation}
We summarize the main steps of the modified tensor truncation 
in Figure \ref{fig:truncation_summary} and in Algorithm \ref{alg:truncation_and_gradient}. Clearly, the matrix $\bm D$ needs 
to be recomputed at each step $\epsilon_i$, resulting in 
the matrix $\bm D_i$ appearing in the gradient descent equations
\eqref{discrete_matrix_ODE}-\eqref{step_truncation_scheme}.

\subsection{Computational cost}

One step of a first-order step-truncation integrator 
of the form \eqref{step_truncation_scheme} 
(e.g. Euler forward) without computing 
the left factor of the Riemannian gradient 
\eqref{descent_generator} requires one multiplication between 
a scalar and a TT tensor ($\mathcal{O}(n r^2)$ FLOPS), 
one addition between two TT tensors, and 
one FTT truncation ($\mathcal{O}(dnr^3)$ FLOPS) for 
a total  computational complexity $\mathcal{O}(dnr^3)$.
In the above estimate we assumed that 
each entry of the rank of 
the FTT tensor $v_i + \Delta \epsilon \Phi(v_i , \bm  D_i,  \Delta \epsilon)$ 
is bounded by $r$ and $v_i$ is discretized 
on a grid with $n$ points in each variable.
In addition to the cost of 
step-truncation, we must 
also compute the matrix 
$\bm D_i$ defined in 
\eqref{descent_generator} 
at each step. 
The orthogonal FTT cores $\bm Q_{\leq j}, \bm Q_{>j}$ 
needed for the computation of $\bm D_i$ are 
readily available during the tensor truncation 
procedure. 
For the computation of $\bm D_i$ we must 
compute the gradient of $\nabla v_i$, which 
requires $d$ matrix multiplications 
between a differentiation matrix of size 
$n \times n$ and a FTT core $\bm \Psi_i$ 
of size $n \times r n$ for total 
complexity of $\mathcal{O}(dn^3r)$. 
We also need to compute the outer 
product $(\nabla v_i) \bm x^{\top}$ where 
each entry of $\nabla v_i$ is in FTT format, 
and, for each of the $d^2$ 
entries in the matrix $(\nabla v_i) \bm x^{\top}$ 
compute an integral of FTT tensors requiring $\mathcal{O}(dnr^3)$ FLOPS. 
The final estimate for computing the matrix $\bm D_i$ 
is $\mathcal{O}(d^2n^3r^3)$, which dominates the 
cost of performing one-step of \eqref{step_truncation_scheme}.
We point out that the computation of $\bm D_i$ 
may be incorporated into high performance computing algorithms for tensor train rounding, e.g.,  \cite{Hussam_parr_TT,Townsend_21}.

\section{Application to multivariate functions}
\label{sec:3Dfunction}

We now demonstrate the Riemannian gradient 
descent algorithm for generating rank-reducing linear 
coordinate transformations. 
%
Consider the Gaussian mixture 
\begin{equation}
\label{Gaussian_mixture}
u(\bm x) = 
\sum_{i=1}^{N_g} w_i \exp\left(- \sum_{j=1}^d 
\frac{1}{\beta_{ij}} \left( \bm R^{(i)}_{j} \cdot \bm x + t_{ij} \right)^2 \right) , 
\end{equation}
where $N_g$ is the number of Gaussians, $\beta_{ij}$ are positive real numbers, 
$w_i$ are positive weights satisfying $$\displaystyle\sum_{i=1}^d w_i = 1,$$
$\bm R_{j}^{(i)}$ is the $j$-th row of a $d\times d$ rotation matrix $\bm R^{(i)}$, 
and $t_{ij}$ are translations.
\begin{figure}[!t]
\centerline{\footnotesize\hspace{.5cm} (a)   \hspace{7cm} (b) \hspace{1cm}}
\vspace{.0cm}
	\centering
\includegraphics[scale=0.39]{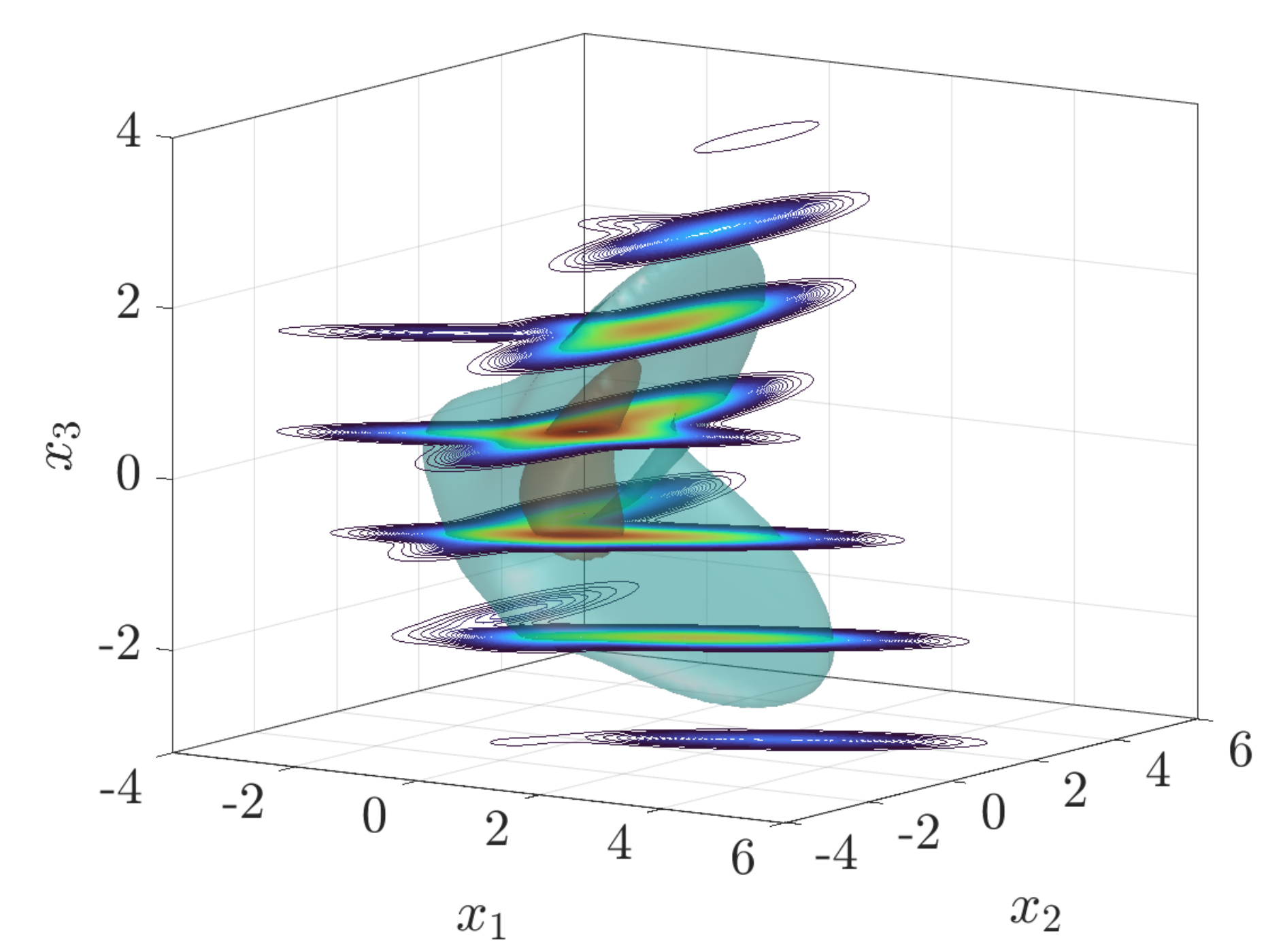} \hspace{.1cm}
\includegraphics[scale=0.39]{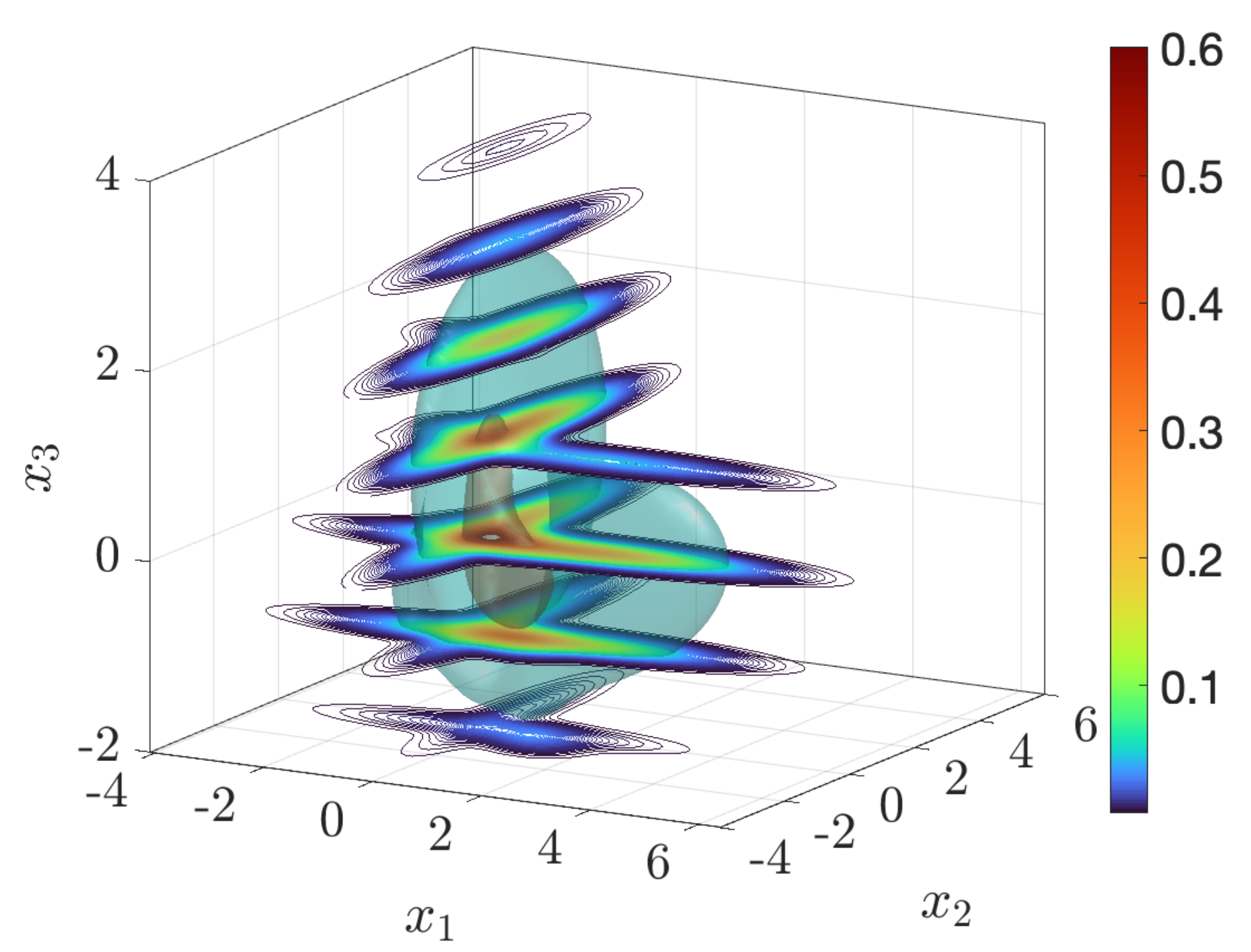} \\
\caption{(a) Volumetric plot of the three-dimensional Gaussian mixture \eqref{Gaussian_mixture}. 
(b) Volumetric plot of the corresponding reduced 
rank ridge tensor $v(\bm x ;\epsilon_f)$.}
\label{fig:contour_slices}
\end{figure}
\begin{figure}[!t]
\centerline{\footnotesize\hspace{.5cm} (a)  \hspace{1.3cm}    \hspace{6.5cm} (b) }
\vspace{.1cm}
	\centering
\includegraphics[scale=0.39]{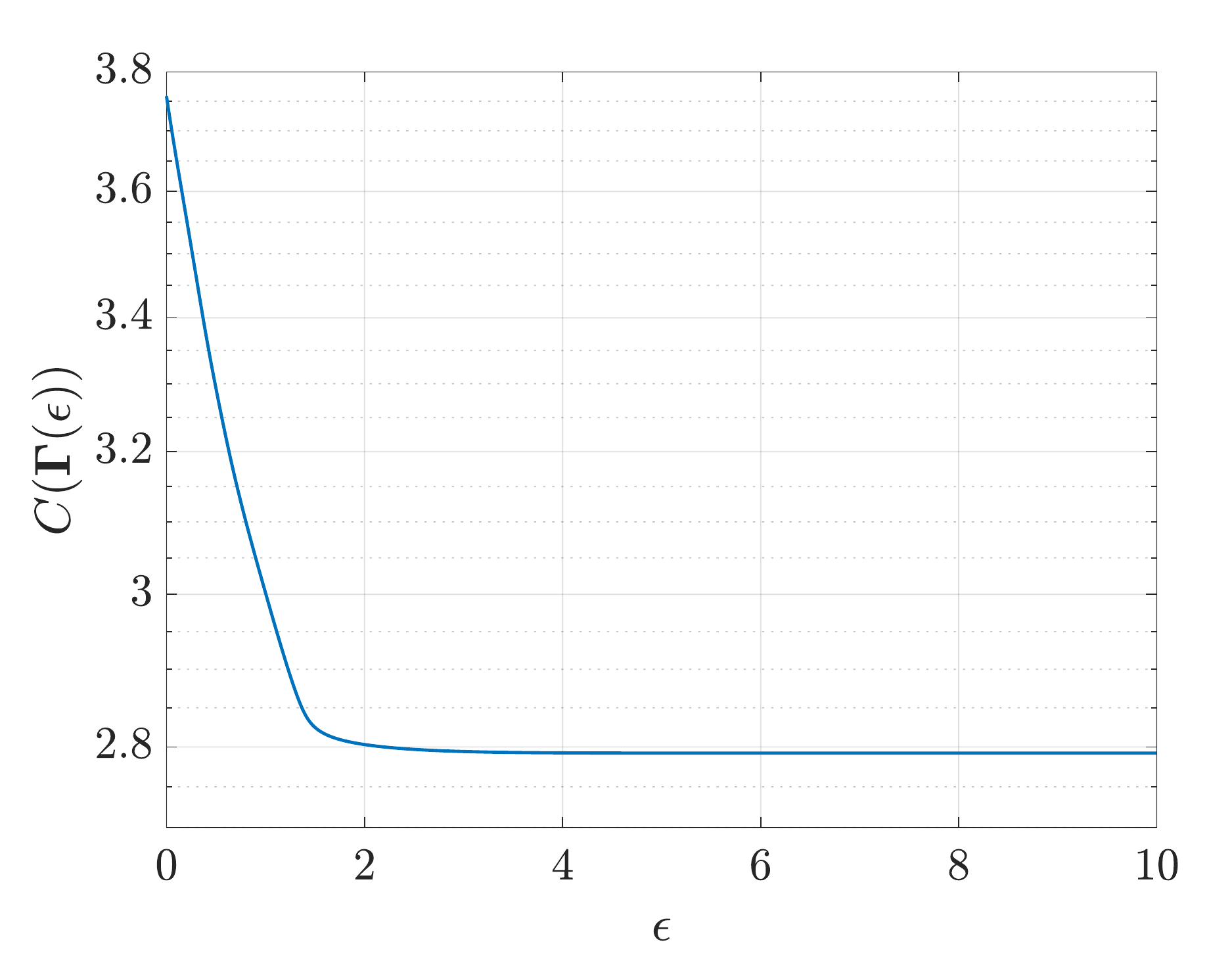} \hspace{.4cm}
\includegraphics[scale=0.39]{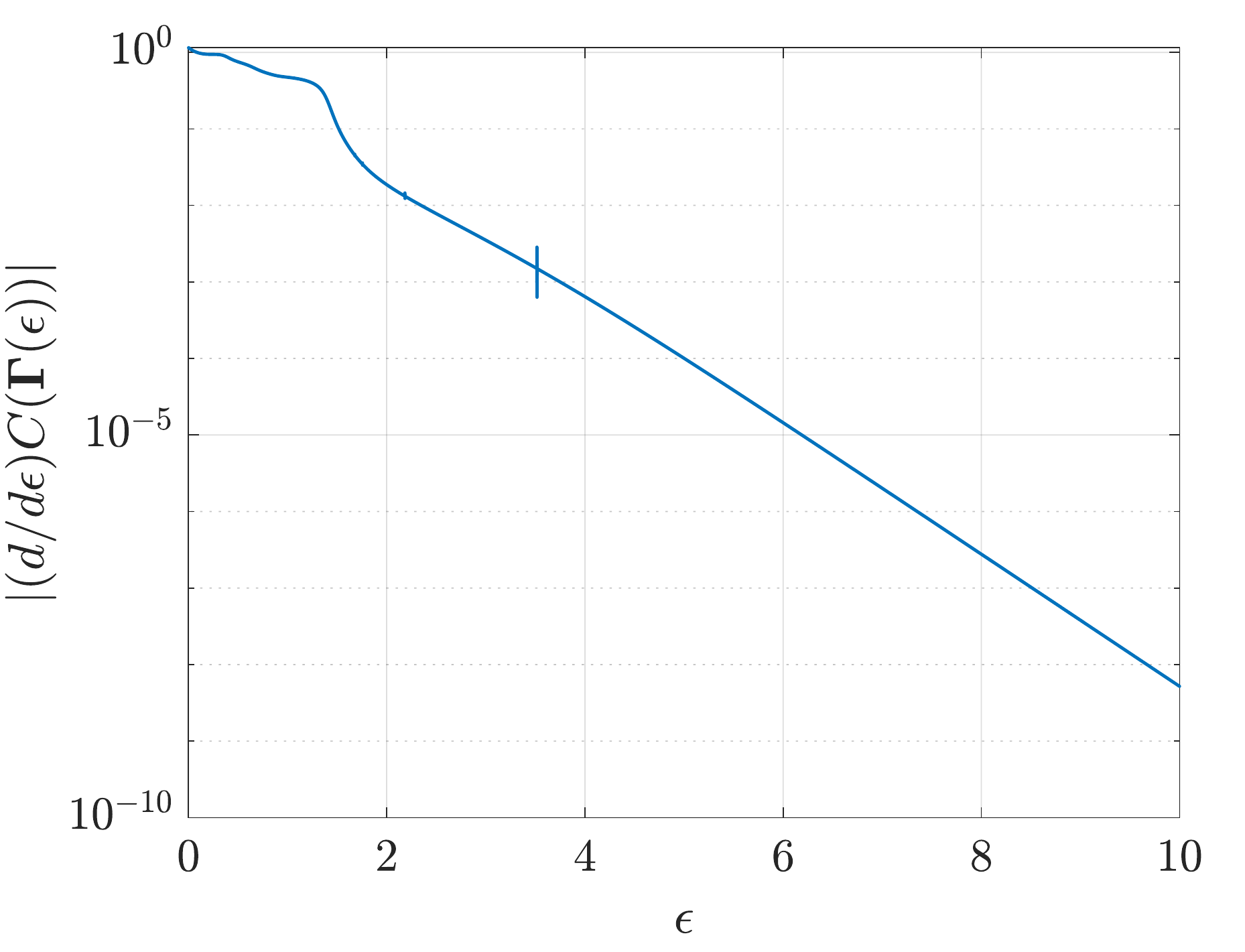} \\
\caption{Rank reduction problem via coordinate flow 
for the three-dimensional Gaussian mixture \eqref{Gaussian_mixture}. 
(a) Cost function \eqref{relaxed_rank_min} evaluated along a 
steepest descent mapping $\bm \Gamma(\epsilon)$ versus $\epsilon$. 
(b) Absolute value of the derivative of the cost function in \eqref{relaxed_rank_min} versus $\epsilon$. The derivative is computed with a second-order backwards finite difference stencil \eqref{BD}. }
\label{fig:cost}
\end{figure}
{ 
For our demonstration we consider 
three spatial dimensions ($d=3$)} and 
set $N_g = 3$, $w_i = 1/3$,  
\begin{equation}
\begin{array}{l l l}
\beta_{11} = 2,  & \beta_{12} = 1/3, &  \beta_{13} = 1/2, \\
\beta_{21} = 3,  & \beta_{22} = 4,  & \beta_{23} = 1/6, \\
\beta_{31} = 1,  & \beta_{32} = 1/5,   & \beta_{33} = 5, 
\end{array} \qquad 
\begin{array}{l l l}
t_{11} = 0, & t_{12} = 0, & t_{13} = 0, \\
t_{21} = -1, & t_{22} = 1/2, & t_{23} = -1/3, \\
t_{31} = 1/2, & t_{32} = -1/4, & t_{33} = 1,
\end{array}
\end{equation}
and the rotation matrices 
\begin{equation}
\bm R^{(i)} = \exp\left( \begin{bmatrix}
0 & \theta_i(1) & \theta_i(2) \\
-\theta_i(1)  & 0 & \theta_i(3) \\
  -\theta_i(2) & -\theta_i(3) & 0
\end{bmatrix} \right), 
\end{equation}
with 
\begin{equation}
\bm \theta_1 = \begin{bmatrix}
\pi/4 \\
\pi/3 \\
\pi/5
\end{bmatrix},  \quad
\bm \theta_2 = \begin{bmatrix}
\pi/3 \\
\pi/6 \\
\pi/4
\end{bmatrix},  \quad
\bm \theta_3 = \begin{bmatrix}
\pi/3 \\
\pi/3 \\
\pi/7
\end{bmatrix}. 
\end{equation}
We discretize the Gaussian mixture \eqref{Gaussian_mixture} 
on the computational domain $[-12,12]^3$
(which is large enough to enclose the numerical support 
of \eqref{Gaussian_mixture}) using $200$ evenly-spaced points 
in each variable. 
From the discretization of 
\eqref{Gaussian_mixture}
we compute the FTT decomposition 
$u_\T(\bm x)$ using recursive SVDs. 
%
%
%
\begin{figure}[!t]
\centerline{\footnotesize\hspace{-0.5cm} (a)  \hspace{7.2cm}   (b)  }
\vspace{0cm}
	\centering
\hspace{-.8cm}\includegraphics[scale=0.4]{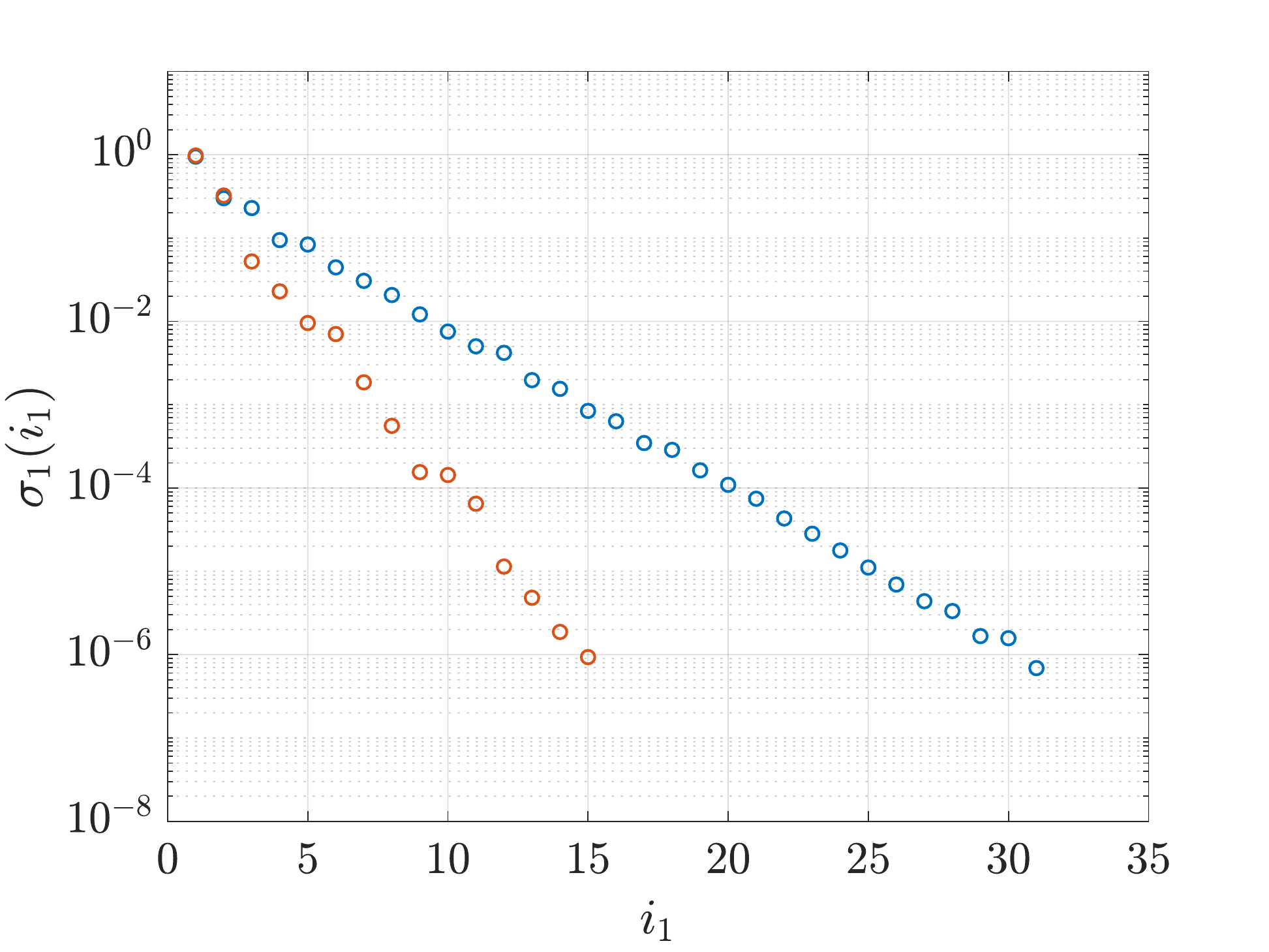} \hspace{-.4cm}
\includegraphics[scale=0.4]{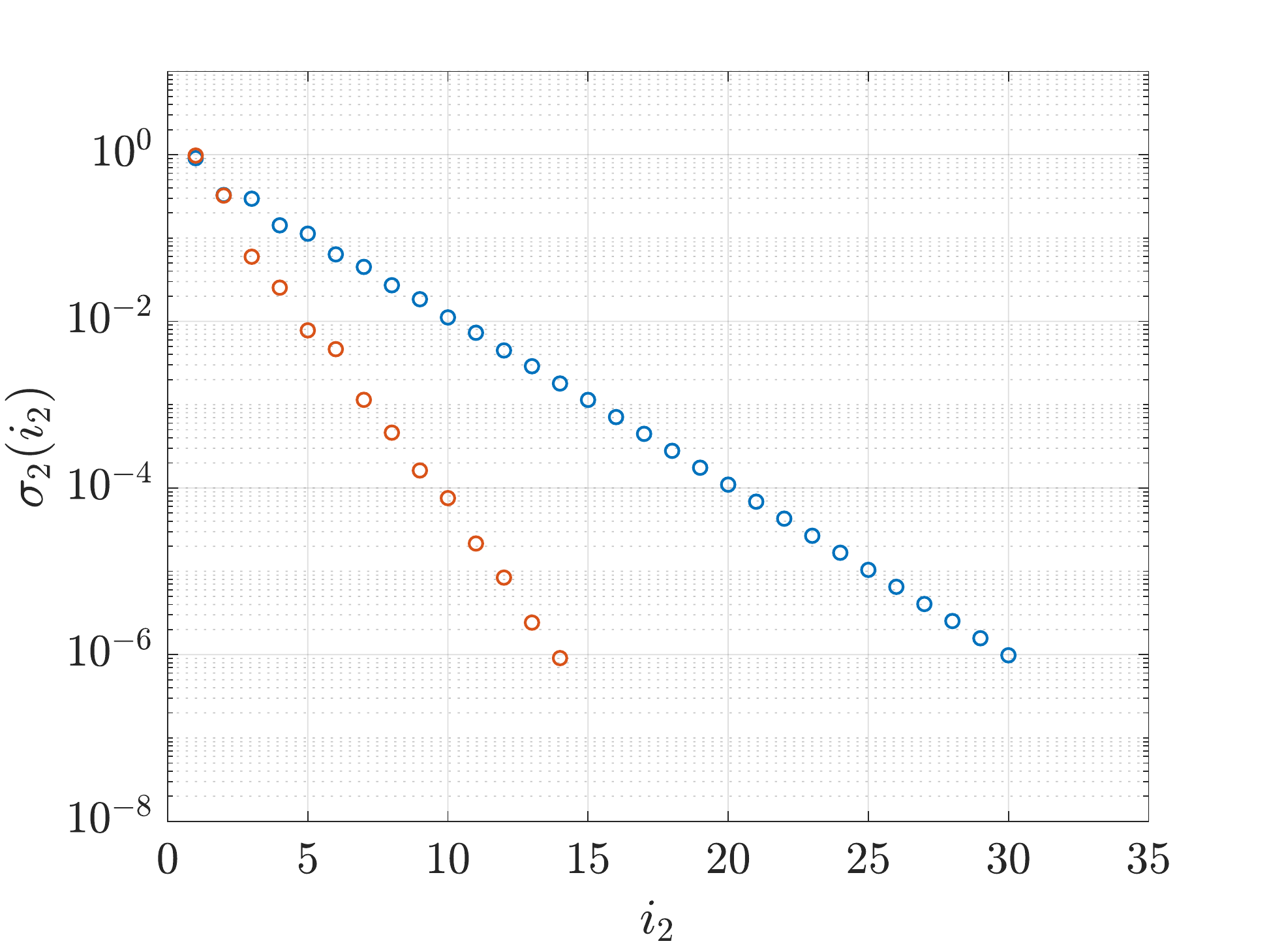} \\ 
\vspace{-.3cm}
\includegraphics[scale=0.48]{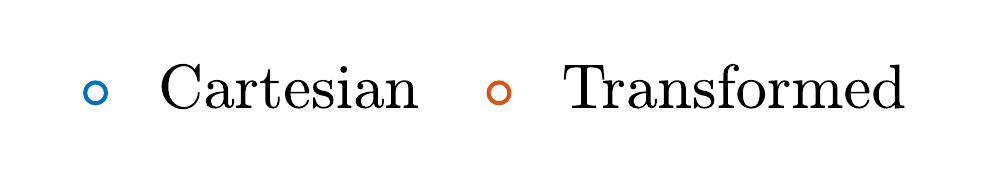} 
\hspace{-.4cm}
\centerline{\footnotesize\hspace{0cm} (c)  }
\includegraphics[scale=0.4]{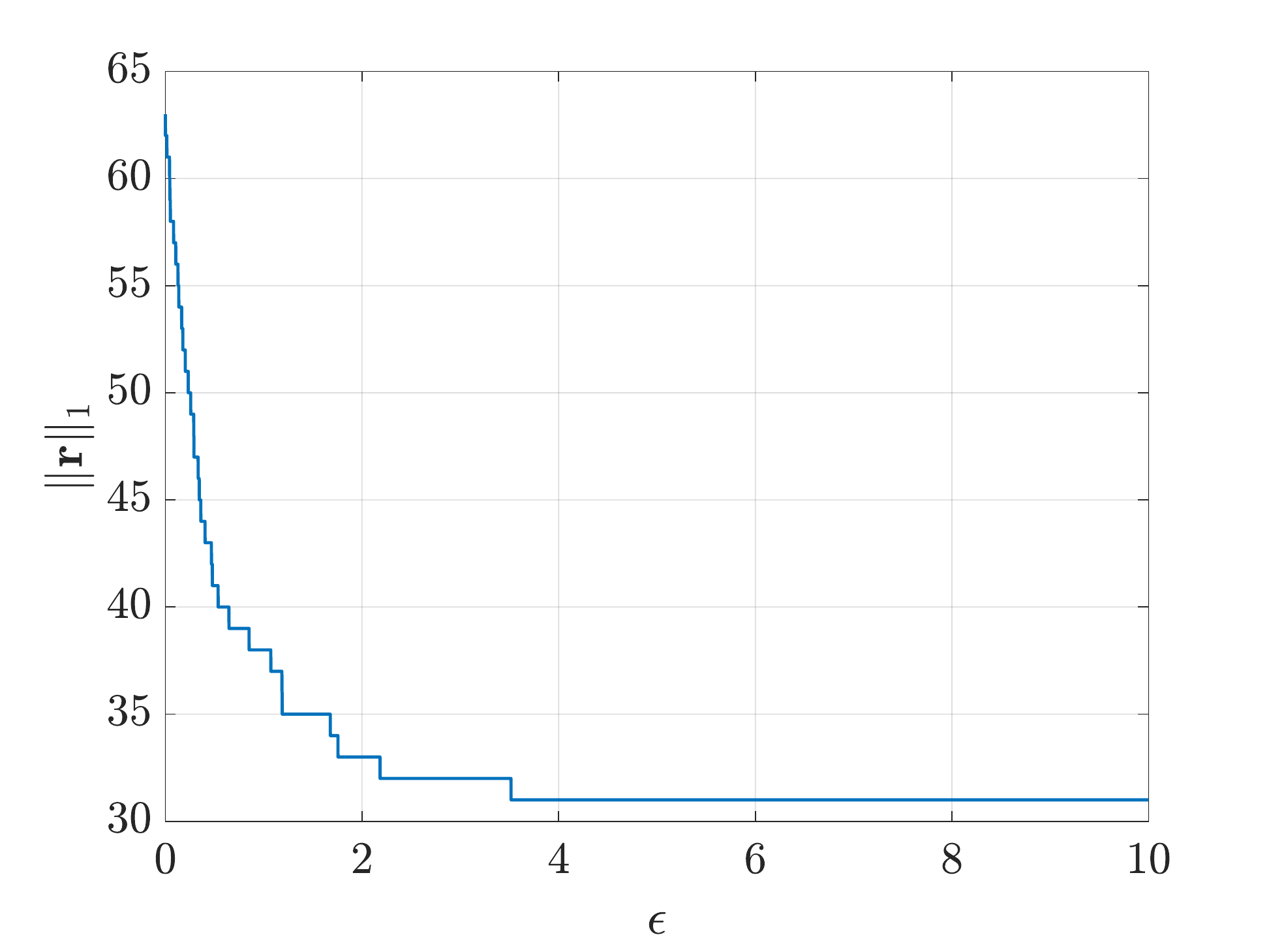} 
\caption{Multilinear spectra of the 3-dimensional Gaussian mixture $v(\bm x;0)$ defined in \eqref{Gaussian_mixture} (Cartesian coordinates) and 
the corresponding reduced-rank ridge tensor $v(\bm x ;\epsilon_f)$ (transformed coordinates). 
(a) Spectra $\sigma_1$ corresponding to multilinear rank $r_1$. 
(b) Spectra $\sigma_2$ corresponding to multilinear rank $r_2$. 
(c) $1$-norm of the multilinear rank vector of 
$v_{\T}(\bm x; \epsilon)$ versus $\epsilon$. It is seen that the coordinate flow 
$\bm \Gamma(\epsilon)\bm x$ reduces the multilinear rank of the Gaussian mixture
\eqref{Gaussian_mixture} from about $63$ (Cartesian coordinates) to $31$ (transformed coordinates).}
\label{fig:spectra}
\end{figure}
To integrate the PDE \eqref{PDE_for_tensor_descent} for 
the reduced rank ridge tensor 
$v(\bm x; \epsilon)$, we use the explicit two-step 
Adams-Bashforth step-truncation method 
with $\Delta \epsilon = 10^{-3}$ and 
relative FTT truncation accuracy $\delta = 10^{-6}$. 
All spatial derivatives and integrals are 
computed by applying 
one-dimensional pseudo-spectral Fourier 
differentiation matrices 
and quadrature weights \cite{spectral_methods_book} 
to the tensor modes. 
We integrate 
up to $\epsilon_f = 10$ 
which is sufficient to demonstrate 
convergence of the gradient descent method. 
In Figure \ref{fig:contour_slices} we 
provide volumetric plots of the 
of the function $v(\bm x; \epsilon)$ for 
$\epsilon = 0$ and $\epsilon = \epsilon_f$. 
We observe that the reduced rank ridge 
tensor $v(\bm x; \epsilon_f)$ appears to be more 
symmetrical with 
respect to the $(x_1,x_2,x_3)$-axes 
than the higher rank function $v(\bm x; 0)$. 
In Figure \ref{fig:cost}(a) we plot the cost 
function $C (\bm \Gamma(\epsilon))$ versus 
$\epsilon$ and in Figure 
\ref{fig:cost}(b) we plot the absolute value of
the derivative of the cost function versus $\epsilon$. 
Observe in Figure \ref{fig:cost} 
that $\bm \Gamma(\epsilon)$ appears to be 
converging to a local minimum of $C$, and, 
the majority of the decrease in the cost function 
occurs in the interval $\epsilon \in [0,2]$. 
Correspondingly, we notice that in Figure \eqref{fig:spectra}(c) 
the majority of rank increase 
occurs in the interval $\epsilon \in [0,2]$. 
Finally, in Figure \ref{fig:spectra}(a)-(b) we plot 
the multilinear spectra of the function $v_\T(\bm x;\epsilon)$ 
for $\epsilon = 0$ and $\epsilon = \epsilon_f$. 
The decay rate of the multilinear spectra 
corresponding to $v(\bm x;\epsilon_f)$ is 
significantly faster than the decay rate of the 
multilinear corresponding to the initial 
function $v(\bm x; 0)$, resulting in a multilinear 
rank of about half the one in Cartesian coordinates. 
Thus, for any truncation tolerance $\delta$, the FTT-ridge 
tensor $v(\bm x;\epsilon_f)$ can be stored at a significantly 
lower cost than the original function. Intuitively, the savings 
that can be obtained by the coordinate flow in 
higher-dimensions are even more pronounced, since 
hierarchical SVDs with steeper spectra yield a much 
smaller number of tensor modes.

\section{Application to PDEs}
\label{sec:PDEs}
We now apply the proposed rank 
tensor reduction method to initial-value problems of the form 
\begin{equation}
\begin{cases}
\displaystyle\frac{\partial u(\bm x,t) }{\partial t} = G(u(\bm x,t), {  \bm x}), \\
u(\bm x,0) = u_0 (\bm x), 
\label{nonlinear-ibvp0} 
\end{cases}
\end{equation}
where $G$ is a nonlinear operator that may
incorporate boundary conditions. 
{  As is well-known, solving \eqref{nonlinear-ibvp0}
numerically involves repeated application 
of $G$. In particular, if the approximate 
solution of the given PDE is represented as a FTT tensor 
$u \approx u_\T$ then the operator 
$G$ must be represented in a form that can 
take $u_\T$ as an input and output 
another FTT tensor.
If $G$ is a linear operator then such a representation 
is given by the rank $\bm g$ FTT-operator 
(or TT-matrix after disceretization \cite{OseledetsTT}) 
\begin{equation}
\label{separated_operator}
G(\cdot,\bm x) \approx G_\T(\cdot,\bm x) = 
\sum_{\alpha_1=1}^{g_1} \sum_{\alpha_2=1}^{g_2} \cdots 
\sum_{\alpha_{d-1}=1}^{g_{d-1}} \bm A_1(x_1;\alpha_1) \otimes \bm A_2(\alpha_1;x_2;\alpha_2) \otimes 
\cdots \otimes \bm A_d(\alpha_{d-1};x_d), 
\end{equation}
where, for fixed 
$\alpha_{i-1}$ and $\alpha_i$, $\bm A_j$ is a 
one-dimensional operator acting only on 
functions of $x_j$. 
The representation \eqref{separated_operator} is also 
known as matrix product operator (MPO) \cite{Hubig}.
%
After applying $G_\T$ to a FTT tensor 
$u_\T$ with rank $\bm r$, the 
new FTT tensor $G_\T (u_\T,\bm x)$ has rank 
given by the element-wise (Hadamard) product of the two ranks 
$\bm g \circ \bm r$, which then has to be truncated. 
The computational cost of 
such a truncation scales cubically in the new FTT rank. 
If the product rank $\bm g \circ \bm r$ is 
prohibitively large then the FTT operator 
$G_\T$ can be split into sums of low 
rank operators
\begin{equation}
\label{split_operator}
G_\T = \sum_{k=1}^n G^{(k)}_\T, 
\end{equation}
where each $G_\T^{(k)}$ has FTT operator rank $\bm g^{(k)}$,
which is less than $\bm g$. 
Hence, instead of applying $G_\T$ directly to the solution 
tensor, we can apply each $G_\T^{(k)}$ ($k=1,2,\ldots,n$) 
to $u_\T$, truncate each $G^{(k)}_\T (u_\T,\bm x)$, and 
then add them together. 
After the addition, one more truncation procedure must 
be performed to ensure the result of the FTT 
addition has optimal ranks. 
This procedure can be written mathematically as 
\begin{equation}
\label{split_operator_applied}
G_\T (u_\T,\bm x) \approx \mathfrak{T}_{\delta} \left[ \sum_{k=1}^n \mathfrak{T}_{\delta} \left( G_\T^{(k)} (u_\T,\bm x) \right) \right], 
\end{equation}
where $\mathfrak{T}_{\delta}$ is a truncation 
(or rounding) operator for FTT tensors with 
relative accuracy $\delta$. 
Alternatively, one can use randomized 
algorithms, e.g., based on tensor sketching, for 
computing sums of many TT tensors \cite{Daas_randomized_TT_rounding}.
This can increase efficiency of 
applying high rank FTT operators 
to FTT tensors. 
For the PDEs 
considered hereafter we employ 
the truncation algorithm 
\eqref{split_operator_applied} to 
mitigate the cost of applying high rank operators.

\subsection{Coordinate transformation}
For a given PDE operator $G(\cdot,\bm x)$ and 
linear coordinate transformation 
$\bm y = \bm \Gamma \bm x$, it is always possible to 
write $G$ in coordinates $\bm y$ resulting in a 
new operator $G_{\bm \Gamma}$. 
If $G$ acts on $u_\T(\bm x,t_k)$ then 
$G_{\bm \Gamma}$ acts on the transformed 
tensor $u_\T(\bm y,t_k)=v_\T(\bm x,t_k)$. 
Such an operator can be constructed 
using standard tools of 
differential geometry \cite{Aris,Conjugateflows}, 
and usually has different FTT-operator 
rank than $G$. 
For example, consider the variable coefficient 
advection operator 
\begin{equation}
\label{advection_operator}
G(u(\bm x,t),\bm x) = \sum_{i=1}^d f_i(\bm x) \frac{\partial u}{\partial x_i}.
\end{equation}
The scalar field 
$u(\bm x,t_k)$ can be written in the new coordinate system 
$\bm y$ as
\begin{equation}
u(\bm x,t_k) =  u(\bm \Gamma^{-1} \bm y,t_k) = U(\bm y,t_k) 
= U(\bm \Gamma \bm x,t_k)  
\end{equation}
which implies that 
\begin{equation}
\frac{\partial u(\bm x,t)}{\partial x_j} = \sum_{k=1}^d \Gamma_{kj}\frac{\partial U(\bm y,t_k)}{\partial y_k}.
\end{equation}
In this way, we can rewrite the operator \eqref{advection_operator} 
in coordinates $\bm y=\bm \Gamma \bm x$ as
\begin{equation}
\label{advection_operator_transformed}
\begin{aligned}
G_{\bm \Gamma}\left(U(\bm y,t),\bm y\right) &= \sum_{i,j=1}^d \Gamma_{ij} f_j\left(\bm \Gamma^{-1}\bm y\right) \frac{\partial U(\bm y,t)}{\partial y_i} \\
&= \sum_{i=1}^d h_i\left(\bm y\right) \frac{\partial U(\bm y,t)}{\partial y_i}, 
\end{aligned}
\end{equation}
where 
\begin{equation}
h_i\left(\bm y\right) = 
\sum_{j=1}^d \Gamma_{ij} f_j\left(\bm \Gamma^{-1}\bm y\right).
\end{equation}
Note that $G_{\bm \Gamma}$ has a relatively simple 
form due to the linearity\footnote{For more general nonlinear coordinate 
transformations $\bm y=\bm H(\bm x)$, the operator 
$G_{\bm \Gamma}$ includes the metric tensor of the coordinate change, 
which can significantly 
complicate the form of $G_{\bm \Gamma}$ 
(e.g., \cite{Aris,Haoxiang}).} of the coordinate transformation. 
In this case, the rank of the operators 
$G$ and $G_{\bm \Gamma}$ are determined by 
the FTT ranks of the variable coefficients 
$f_i(\bm x)$ and $h_i(\bm y)$ ($i=1,2,\ldots,d$), respectively.  

}

{ 
\subsection{Time integration} 
\label{sec:time_dep_funs}
For the time integration of 
\eqref{nonlinear-ibvp0} using low-rank 
tensors we discretize 
the temporal domain of interest $[0, T]$ into $N+1$ 
evenly-spaced time instants,
\begin{equation}
\label{discrete_time}
t_k = k \Delta t, \qquad 
\Delta t = \frac{T}{N}, \qquad k=0,1,\ldots,N,
\end{equation}
%
%
and consider the rank-adaptive step-truncation scheme  
\cite{adaptive_rank,rodgers2020step-truncation,rodgers2022implicit}
\begin{equation}
\label{step_truncation_scheme_time}
u_\T(\bm x,t_{k+1}) = \mathfrak{T}_{\delta} 
\left( u_\T(\bm x,t_k) + 
\Delta t \Phi\left(G_\T,u_{\T}(\bm x,t_k), \Delta t\right) \right),
\end{equation}
where $\Phi$ is {  a iteration function associated with a temporal discretization scheme and $G_\T$ is a FTT-operator approximation \eqref{separated_operator} of the given operator $G$.
For example, a step-truncation Euler forward 
scheme is 
\begin{equation}
\label{Euler_ST}
u_\T(\bm x,t_{k+1}) = \mathfrak{T}_{\delta} 
\left[ u_\T(\bm x,t_k) + 
\Delta t \mathfrak{T}_{\delta} \left( G_\T(u_{\T}(\bm x,t_k),\bm x)\right) \right],
\end{equation}
while a step-truncation Adams-Bashforth 2 (AB2) 
scheme\footnote{  Variants of these 
step-truncation schemes can be obtained by inserting 
or removing truncation operations 
between summations, 
changing 
truncation tolerances $\delta$ in each of the 
truncation operators, or by 
using operator splitting 
\eqref{split_operator}.} can be written as 
\begin{equation}
\label{AB2_ST}
u_\T(\bm x,t_{k+1}) = \mathfrak{T}_{\delta} 
\left[ u_\T(\bm x,t_k) + 
\Delta t \left( 
\frac{3}{2} \mathfrak{T}_{\delta} \left[ G_\T(u_{\T}(\bm x,t_k),\bm x) \right]
- \frac{1}{2} \mathfrak{T}_{\delta} \left[ G_\T(u_{\T}(\bm x,t_{k-1}),\bm x)\right] \right) \right].
\end{equation}
} 

\vs
\begin{algorithm}[!t]
\SetAlgoLined
 \caption{PDE integrator with {  adaptive} rank-reducing coordinate transformations}
\label{alg:coord_adaptive_integrator}
\vspace{0.3cm}
 \KwIn{\vspace{0.1cm}\\\\
   $u_0$ $\rightarrow$ initial condition in FTT tensor format, \\
   $\Delta t$ $\rightarrow$ temporal step size, \\
   $N_t$ $\rightarrow$ total number of time steps, \\
   {  ${\rm max\ rank}$ $\rightarrow$ maximum rank during FTT integration before attempting rank reduction,} \\
   {  ${\rm max\ time}$ $\rightarrow$ maximum computational time for one time step before attempting rank reduction,} \\
   {  $k_r$ $\rightarrow$ increase for maximum rank after performing coordinate transformation,} \\
   {  $k_t$ $\rightarrow$ increase for maximum time after performing coordinate transformation,}\\
   $\Delta \epsilon$ $\rightarrow$ gradient descent step-size, \\
   ${  \eta}$ $\rightarrow$ tolerance for coordinate gradient descent, \\
  $M_{\rm iter}$ $\rightarrow$ maximum number of iterations for gradient descent routine.
 }
 \vspace{0.3cm}
 \KwOut{\vspace{0.1cm}\\
  $\bm \Gamma$ $\rightarrow$ rank-reducing linear coordinate transformation for PDE solution, \\
  $v_{\T}(\bm x, t_f)= u_\T(\bm \Gamma \bm x, t_f)$ $\rightarrow$ FTT solution tensor at time $t_f$ on rank-reducing coordinate system.}
\vspace{0.3cm}
{\bf Runtime:} 
\begin{itemize}
\item [] $\bm \Gamma = \bm I$,
\item [] $v_0 = u_\T$, \\
\item [] {\bf for $k = 0$ to $N_t$ }
\begin{itemize}
\item [] { \bf if ${\rm time} > {\rm max\ time}$ or ${\rm rank} > {\rm max\ rank}$}
\begin{itemize}
\item [] $ [ v_{k}, \bm \Gamma_{\rm new} ] = {\rm gradient \ descent}( v_{k}, \Delta \epsilon, {  \eta}, M_{\rm iter})$
\item [] $\bm \Gamma = \bm \Gamma_{\rm new} \bm \Gamma$ 
\item[] {  ${\rm max\ rank} = {\rm max\ rank} + k_r$}
\item[] {  ${\rm max\ time} = {\rm max\ time} + k_t$}
\end{itemize}
\item [] {\bf end} 
 \item [] $[v_{k+1}, {  {\rm time}, {\rm rank}}] = \mathfrak{T}_{\delta} \left( v_k + \Delta t \Phi(v_k, G_{\T,\bm \Gamma}, \Delta t) \right)$ \\
\end{itemize}
  {\bf end}
\end{itemize}
\end{algorithm}
\vs

At any time step $t_k$ during temporal integration of 
\eqref{nonlinear-ibvp0} we may compute 
a rank-reducing coordinate transformation 
$\bm y = \bm \Gamma \bm x$ and obtain a 
tensor ridge representation of the solution 
at time $t_k$. 
To integrate the initial boundary 
value problem \eqref{nonlinear-ibvp0} 
using the tensor ridge representation of 
the solution at time $t_k$, 
the operator $G_\T$ may be rewritten 
as a new (FTT) operator $G_{\T,\bm \Gamma}$ 
acting in the transformed coordinate system.
With the operator $G_{\T,\bm \Gamma}$ available, we 
can write the following PDE for $U(\bm y,t)$ 
corresponding to \eqref{nonlinear-ibvp0} 
\begin{equation}
\begin{cases}
\displaystyle\frac{\partial U(\bm y,t)}{\partial t} = G_{\T,\bm \Gamma}(U(\bm y,t),\bm y), \qquad t\geq t_k, \\
U(\bm y,t_{k}) =  v_{\T}(\bm y,t_{k}), 
\label{nonlinear-ibvp_y} 
\end{cases}
\end{equation}
with initial condition given at time $t_k$. 
Time integration can proceed in the 
transformed coordinate system by 
applying a step-truncation 
scheme \eqref{step_truncation_scheme_time}-\eqref{AB2_ST} to 
the transformed PDE \eqref{nonlinear-ibvp_y} resulting 
in a step-truncation scheme in coordinates $\bm y = \bm \Gamma \bm x$
\begin{equation}
\label{step_truncation_scheme_time_y}
v_\T(\bm y,t_{k+1}) = \mathfrak{T}_{\delta} 
\left[ v_\T(\bm y,t_k) + 
\Delta t \Phi\left(G_{\T,\bm \Gamma},v_{\T}(\bm y,t_k), \Delta t\right) \right]. 
\end{equation}
It is well-known that the computational cost of the scheme 
\eqref{step_truncation_scheme_time} 
(or \eqref{step_truncation_scheme_time_y}) 
scales linearly in the problem dimension $d$ and 
polynomially in the tensor rank of the solution 
and the operator. 
To determine an optimal 
coordinate transformation for 
reducing the overall cost of temporal integration it is 
necessary to have more precise estimates 
on the computational cost of one time step. 
Such computational cost depends on many factors, e.g., 
the increment function $\Phi$, the separation rank 
of the PDE operator $G_{\bm \Gamma}$, the 
operator splitting \eqref{split_operator} used, 
the FTT rank of the PDE solution 
$U(\bm y,t_k)$ at time $t_k$, the rank 
of the operator applied to the 
solution after truncation 
$\mathfrak{T}_{\delta}\left( G_{\bm \Gamma}(U(\bm y,t_k))\right)$, etc.
From this observation it is clear that in order to obtain 
a optimal coordinate transformation for reducing the overall 
computational cost of temporal integration with step-truncation, 
we must take into consideration all these factors, 
in particular the solution rank and the operator rank. 
We emphasize that determining a coordinate 
transformation $\bm \Gamma$ that controls the separation 
rank of a general nonlinear operator $G_{\T,\bm \Gamma}$ 
is a non-trivial problem that we do not address 
in the present paper. 
\vs
\begin{algorithm}[!t]
\SetAlgoLined
 \caption{  PDE integrator with coordinate corrections at each time step.}
\label{alg:one_step_one_step}
\vspace{0.3cm}
 \KwIn{\vspace{0.1cm}\\\\
   $u_0$ $\rightarrow$ initial condition in FTT tensor format, \\
   $\Delta t$ $\rightarrow$ temporal step size, \\
   $N_t$ $\rightarrow$ total number of time steps, \\
   $\Delta \epsilon$ $\rightarrow$ gradient descent step-size, \\
   ${  \eta}$ $\rightarrow$ tolerance for coordinate gradient descent, \\
  $M_{\rm iter}$ $\rightarrow$ maximum number of iterations for gradient descent routine.
 }
 \vspace{0.3cm}
 \KwOut{\vspace{0.1cm}\\
  $\bm \Gamma$ $\rightarrow$ rank-reducing linear coordinate transformation for PDE solution, \\
  $v_{\T}(\bm x, t_f)= u_\T(\bm \Gamma \bm x, t_f)$ $\rightarrow$ FTT solution tensor at time $t_f$ on rank-reducing coordinate system.}
\vspace{0.3cm}
{\bf Runtime:} 
\begin{itemize}
\item [] $\bm \Gamma = \bm I$,
\item [] $v_0 = u_\T$, \\
\item [] {\bf for $k = 0$ to $N_t$ }
\begin{itemize}
\item [] $ [ v_{k}, \bm \Gamma_{\rm new} ] = {\rm gradient \ descent}( v_{k}, \Delta \epsilon, {  \eta}, M_{\rm iter})$
\item [] $\bm \Gamma = \bm \Gamma_{\rm new} \bm \Gamma$ 
 \item [] $v_{k+1} = \mathfrak{T}_{\delta} \left( v_k + \Delta t \Phi(v_k, G_{\T,\bm \Gamma}, \Delta t) \right)$ \\
  \end{itemize}
  {\bf end}
\end{itemize}
\end{algorithm}
\vs
%
\subsection{Coordinate-adaptive time integration}
Next we develop coordinate-adaptive time 
integration schemes for PDEs on FTT manifolds 
that are designed to control the solution rank, 
the PDE operator rank, or the rank of the right
hand side of the PDE.}
{  
\noindent
The first coordinate-adaptive algorithm (Algorithm \ref{alg:coord_adaptive_integrator}) 
is designed to attempt a rank 
reducing coordinate transformation if 
the computational cost of time integration 
in the current coordinate system 
exceeds a predetermined threshold. 
The computational cost of one time step 
may be measured in different ways, e.g., by the 
CPU-time it takes to perform one 
time step, by the rank of the 
solution, or by the rank of the right 
hand side of the PDE (operator applied to the solution).
%
%
%
In the second to last line of Algorithm 
\ref{alg:coord_adaptive_integrator}, 
``${\rm time}$'' denotes the computational time 
it takes to compute one time step 
and ``${\rm rank}$'' denotes either the solution rank, 
the rank of the PDE right hand side, or the maximum 
of the two.

The second algorithm (Algorithm \ref{alg:one_step_one_step}) 
we propose for coordinate-adaptive tensor integration of PDEs 
is based on computing a small correction 
of the 
coordinate system at every time step. 
In practice, we compute one $\epsilon$-step 
of \eqref{PDE_for_tensor_descent} at every 
time step during temporal integration of the given PDE.
This yields a PDE in which the 
operator (which depends on the coordinate system) 
changes at every time step, i.e., a time-dependent 
operator induced by the time-dependent 
coordinate change. 

Hereafter we apply these coordinate-adaptive algorithms 
to five different PDEs and compare the 
results with conventional FTT integrators in 
fixed Cartesian coordinates. 
All numerical 
simulations were run in 
Matlab 2022a on a 2021 MacBook Pro 
with M1 chip and 16GB RAM, 
spatial derivatives and 
integrals were approximated 
with one-dimensional 
Fourier pseudo-spectral 
differentiation matrices and quadrature 
rules \cite{spectral_methods_book},
and various explicit step-truncation 
time integration schemes were used. 
}
{  
\subsection{2D linear advection equations}
\label{sec:time_dependent_function}
First we apply coordinate-adaptive 
tensor integration to the 
2D linear advection equation 
\begin{equation}
\label{prototype-transport-PDE0}
\begin{cases}
\displaystyle \frac{\partial u(\bm x,t)}{\partial t} = 
f_1(\bm x) \frac{\partial u(\bm x,t)}{\partial x_1} + 
f_2(\bm x) \frac{\partial u(\bm x, t)}{\partial x_2},  \\
u(\bm x,0) = u_0(\bm x),
\end{cases}
\end{equation}
with two different sets of coefficients $f_i(\bm x)$ 
specified hereafter.  
Each example is designed to demonstrate different features of the 
proposed coordinate-adaptive algorithms (Algorithm \ref{alg:coord_adaptive_integrator} and 
Algorithm \ref{alg:one_step_one_step}). 

\begin{figure}[!t]
\centerline{\footnotesize\hspace{-.4cm} (a) 
\hspace{5cm}  (b)  \hspace{4.75cm} (c) }
\vspace{-.3cm}
\centering
\includegraphics[scale=0.29]{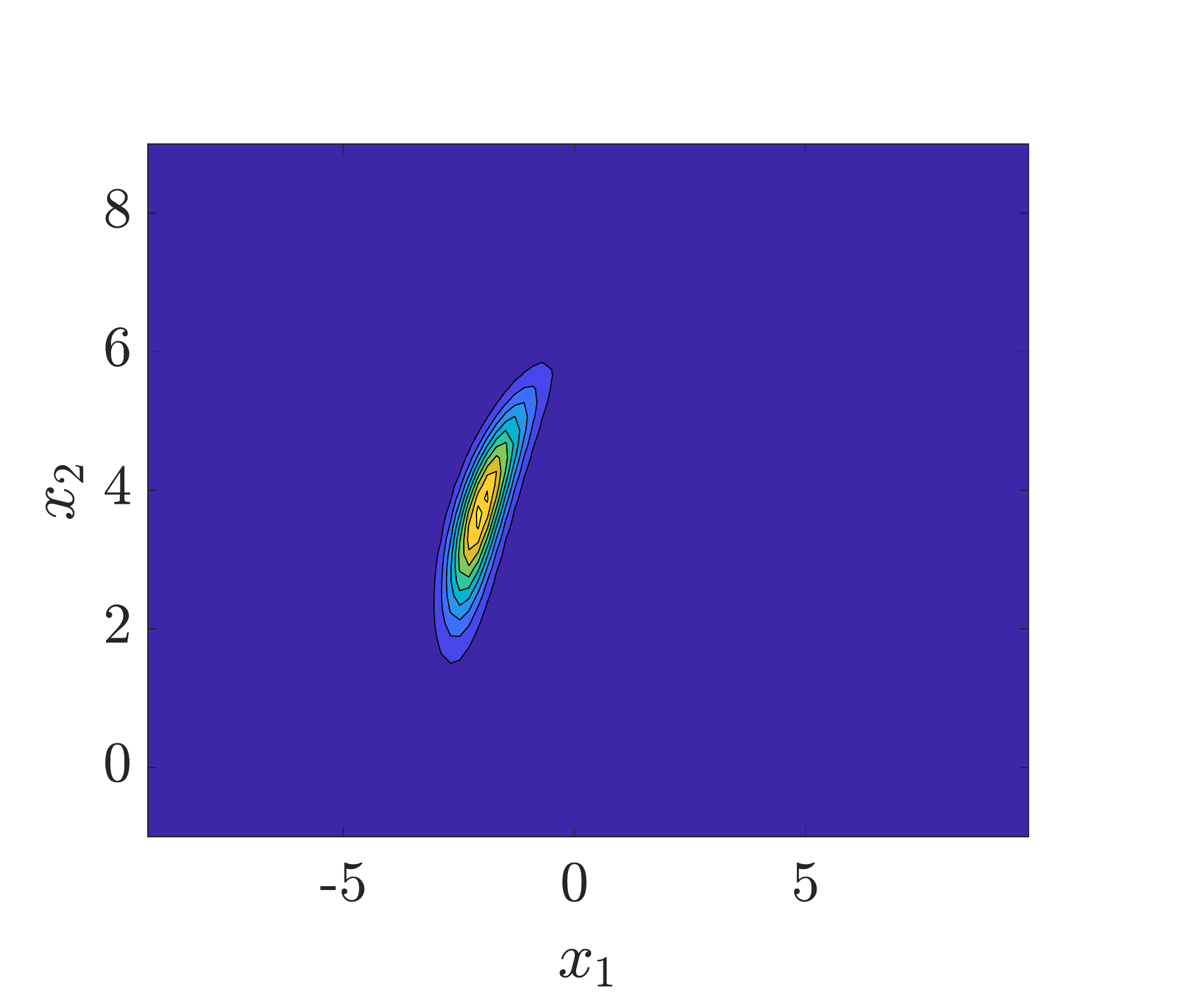} 
\includegraphics[scale=0.29]{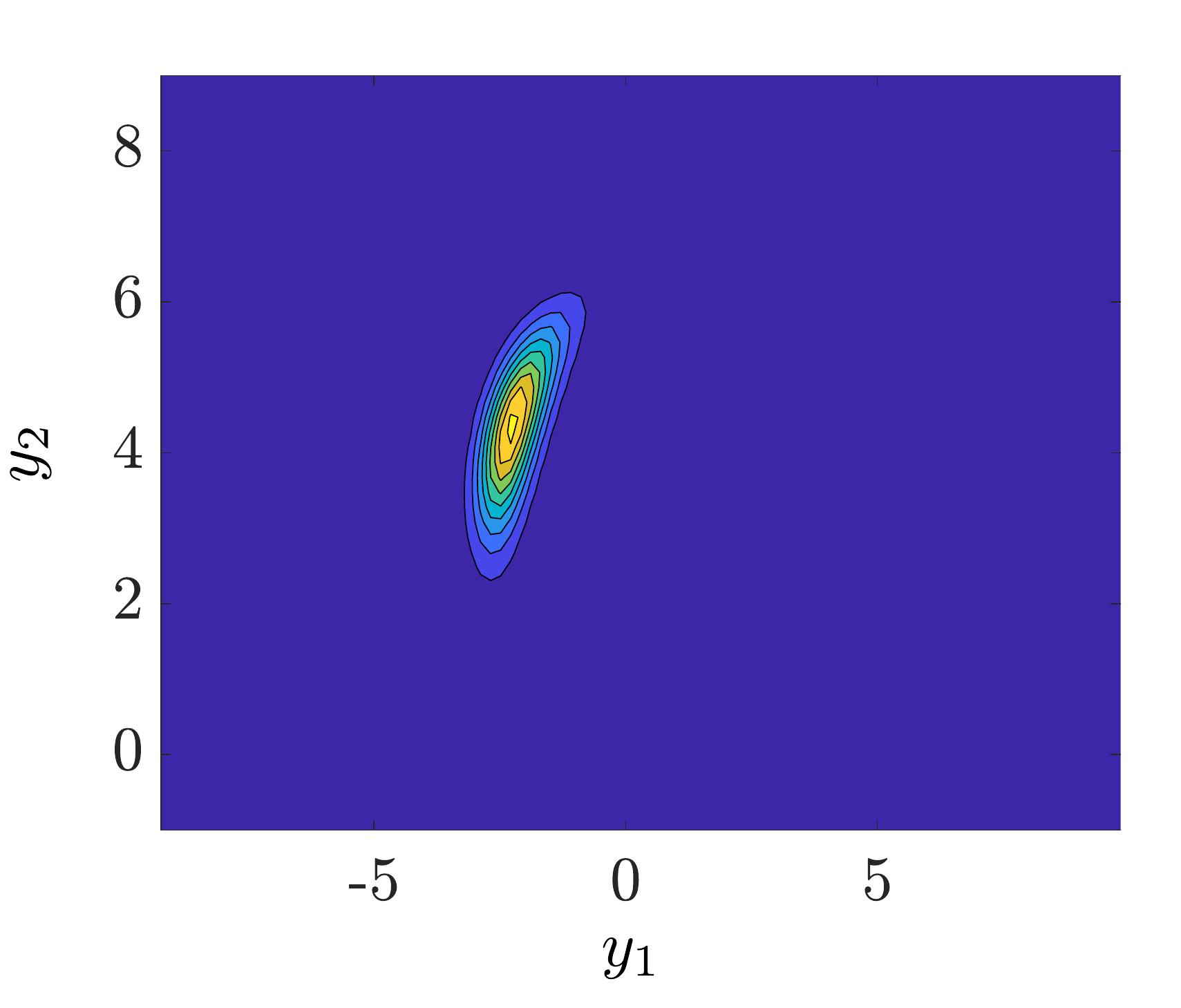} 
\includegraphics[scale=0.29]{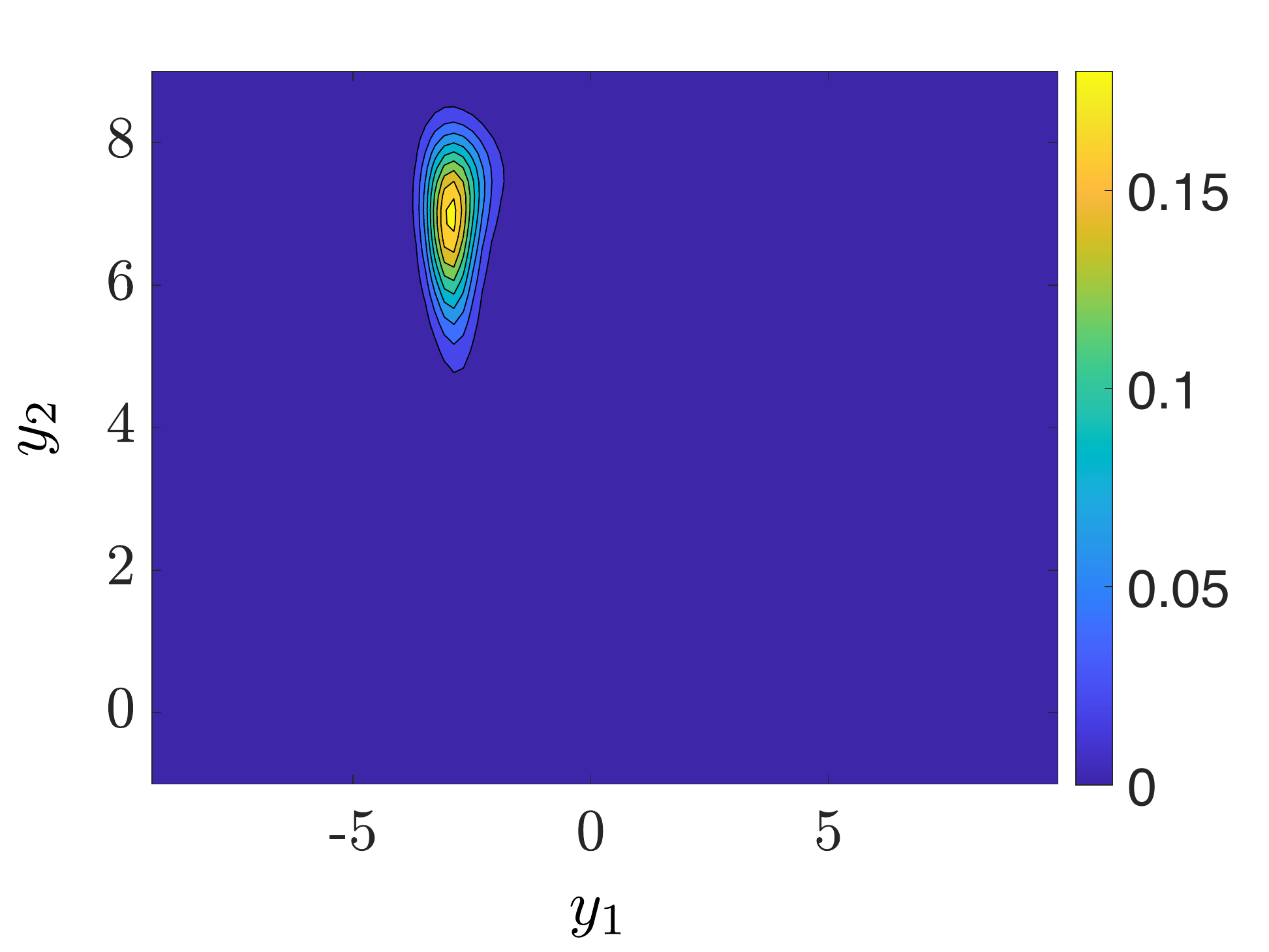} 
\caption{  Solution to the linear advection equation 
\eqref{prototype-transport-PDE0} with coefficients 
\eqref{vrtx_vf} at time $t = 30$. (a) Cartesian coordinates, 
(b) tensor ridge simulation 1 using on Algorithm \ref{alg:coord_adaptive_integrator} with $\Delta \epsilon = 10^{-4}$, 
$M_{\rm iter} = 2000$, and ${  \eta} = 10^{-1}$,
(c) tensor ridge simulation 2 
using Algorithm \ref{alg:one_step_one_step} with $M_{\rm iter} = 1$ and 
$\Delta \epsilon = 10^{-3}$. 
In all simulations the truncation tolerance is 
set to $10^{-6}$.}
\label{fig:vortex_PDE}
\end{figure}
In the first example, we generate the vector field 
$\bm f(\bm x)=(f_1(\bm x),f_2(\bm x))$ via the two-dimensional 
stream function \cite{Venturi_2012}
\begin{equation}
\psi(x_1,x_2) = \Theta(x_1) \Theta(x_2)
\end{equation}
with 
\begin{equation}
\label{ns_eig_fun}
\Theta(x) = \frac{\cos(\alpha x/L)}{\cos(\alpha/2)} - \frac{\cosh(\alpha x/L)}{\cosh(\alpha/2)},
\end{equation}
$L=30$ and  $\alpha = 4.73$.
Such a stream function generates the divergence-free 
vector field (see Figure \ref{fig:vf}(a))
\begin{equation}
\label{vrtx_vf}
f_1(\bm x) = \frac{\partial \psi}{\partial x_2}, \qquad
f_2(\bm x) = -\frac{\partial \psi}{\partial x_1}. 
\end{equation}
We set the initial condition 
\begin{equation}
\label{initial_vortex_fun}
u_0(\bm x) = \frac{1}{m} \exp\left(-4(x_1 - 2)^2\right)
\exp\left(-\frac{(x_2 - 2)^2}{2}\right),
\end{equation}
where 
\begin{equation}
\nonumber
m=\left\|\exp\left(-4(x_1 - 2)^2\right)
\exp\left(-\frac{(x_2 - 2)^2}{2}\right)\right\|_{L^2(\Omega)}
\end{equation}
is a normalization constant.
%
%

We first ran one rank-adaptive tensor simulation 
in fixed Cartesian coordinates. 
We then ran two coordinate-adaptive simulations 
that use rank-reducing coordinate transformations 
during time integration. 
In the first coordinate-adaptive simulation we use 
Algorithm \ref{alg:coord_adaptive_integrator} 
with ${\rm max\ rank} = 15$ and $k_r=0$ 
to initialize coordinate transformations 
during time integration. 
For the Riemannian gradient descent algorithm 
that computes the rank reducing coordinate transformation 
we set step-size $\Delta \epsilon = 10^{-4}$, 
maximum number of iterations $M_{\rm iter} = 
2000$ and stopping tolerance 
${  \eta} = 10^{-1}$. 
In the second coordinate-adaptive simulation we use 
Algorithm \ref{alg:one_step_one_step} 
with $M_{\rm iter} = 1$ and 
$\Delta \epsilon = 10^{-3}$, i.e., 
the integrator performs one step of time integration 
followed by one step of the Riemannian gradient descent 
algorithm \ref{alg:gradient_descent}. 
In both coordinate-adaptive simulations 
we set the truncation threshold $\delta = 10^{-6}$. 

In Figure \ref{fig:vortex_PDE} we plot the solutions 
obtained from each of the three simulations 
at time $t=30$. 
In order to check the accuracy of integrating 
the PDE solution in the low-rank coordinate 
system we mapped the transformed solution back to
Cartesian coordinates using a two-dimensional 
trigonometric interpolant and compared with the 
solution computed in Cartesian 
coordinates. 
In both coordinate-adaptive simulations we 
found that the global $L^{\infty}$ error is bounded 
by $8 \times 10^{-4}$, 
suggesting that the coordinate transformation 
does not affect accuracy significantly. 
In Figure \ref{fig:vrtx_rank} we plot the 
solution rank and the rank of the of the right hand 
side of the PDE \eqref{prototype-transport-PDE0} 
versus time for all three tensor simulations.
We observe that in the coordinate-adaptive 
tensor ridge simulations the ranks of both 
the solution and the PDE right hand side are 
less than or equal to the corresponding 
ranks in Cartesian coordinates. 
We also observe that the adaptive simulation based on 
Algorithm \ref{alg:one_step_one_step} 
(denoted by ``tensor ridge 2'' in Figure \ref{fig:vrtx_rank}) 
has significantly smaller rank than the adaptive simulation 
based on Algorithm \ref{alg:coord_adaptive_integrator} 
(denoted by ``tensor ridge 1'' in Figure \ref{fig:vrtx_rank}).

\begin{figure}[!t]
	\centerline{\footnotesize\hspace{.3cm} (a) \hspace{5.8cm} (b) }
\centering
\includegraphics[scale=0.35]{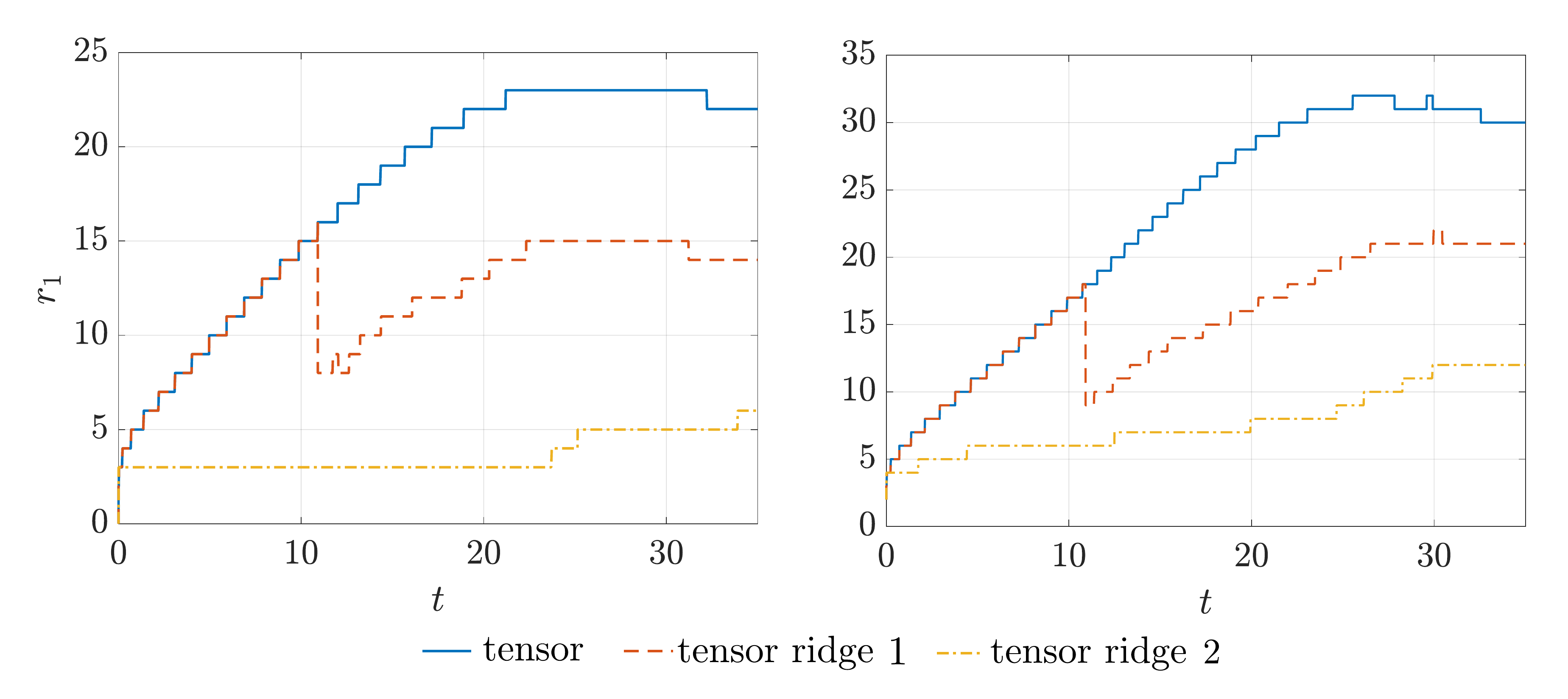}
\caption{  2D Linear advection equation
\eqref{prototype-transport-PDE0} with coefficients 
\eqref{vrtx_vf}. Rank (number of singular values larger than $\delta = 10^{-6}$) of the PDE solution (a) and the right hand side the PDE (b).
Tensor ridge 1 was computed using Algorithm \ref{alg:coord_adaptive_integrator} 
with a maximum solution rank threshold of $15$. 
Tensor ridge 2 was computed using 
Algorithm \ref{alg:one_step_one_step}, 
which performs coordinate corrections at 
each time step. 
}
\label{fig:vrtx_rank}
\end{figure}
}

\vs
{  
\vs

%
Next, we demonstrate that linear coordinate transformations can 
be used to reduce the rank of a PDE operator and reduce the 
overall computational cost of temporal integration. 
To this end, consider again the two-dimensional 
advection equation \eqref{prototype-transport-PDE0} 
this time with advection coefficients 
\begin{equation}
\label{vf2}
\begin{aligned}
f_1(\bm x) &= \exp\left(-a_1 \left(\bm R_1 \cdot \bm x \right)^2\right)
\exp\left(-a_2\left(\bm R_2 \cdot \bm x\right)^2\right), 
\\ 
f_2(\bm x) &= \exp\left(-b_1 \left(\bm R_1 \cdot \bm x \right)^2\right)
\exp\left(-b_2\left(\bm R_2 \cdot \bm x\right)^2\right), 
\end{aligned}
\end{equation}
where $\bm R_i$ is the $i^{\text{th}}$ row 
of the matrix $\bm R$, 
\begin{equation}
\begin{aligned}
&\bm R = \begin{bmatrix}
\cos(\theta) & -\sin(\theta) \\
\sin(\theta) & \cos(\theta)
\end{bmatrix}. 
\end{aligned}
\end{equation}
We set parameters $\theta = \pi/4$, 
\begin{equation}
a_1 = 1/20, \quad a_2 = 1/10 , \quad b_1 = 1/10, \quad b_2 = 1/20, 
\end{equation}
and initial condition 
\begin{equation}
\label{advec_2_IC}
u_0(\bm x) = \exp\left(- x_1^2/3 \right) \exp\left(- x_2^2/3 \right).
\end{equation}
In Figure \ref{fig:vf}(b) we plot the vector field
$\bm f(\bm x)=(f_1(\bm x),f_2(\bm x))$ defined 
by \eqref{vf2}. 
Note that the initial condition $u_0(\bm x)$ is 
rank $1$. 
The rank of the linear advection operator 
defined on the right hand side 
of \eqref{prototype-transport-PDE0} depends 
on the FTT truncation tolerance used to compress 
the multivariate functions $f_1(\bm x)$ and $f_2(\bm x)$. 
If we choose the coordinate transformation 
$\bm y = \bm \Gamma \bm x$, where 
$\bm \Gamma = \bm R^{-1}$, then 
the initial condition remains rank 1, but the 
rank of the advection operator at the right hand side 
of \eqref{prototype-transport-PDE0} becomes 2, 
regardless of the FTT truncation tolerance used. 
\begin{figure}[!t]
	\centerline{\footnotesize\hspace{.2cm} (a) \hspace{6.cm} (b) }
\centering
\includegraphics[scale=0.33]{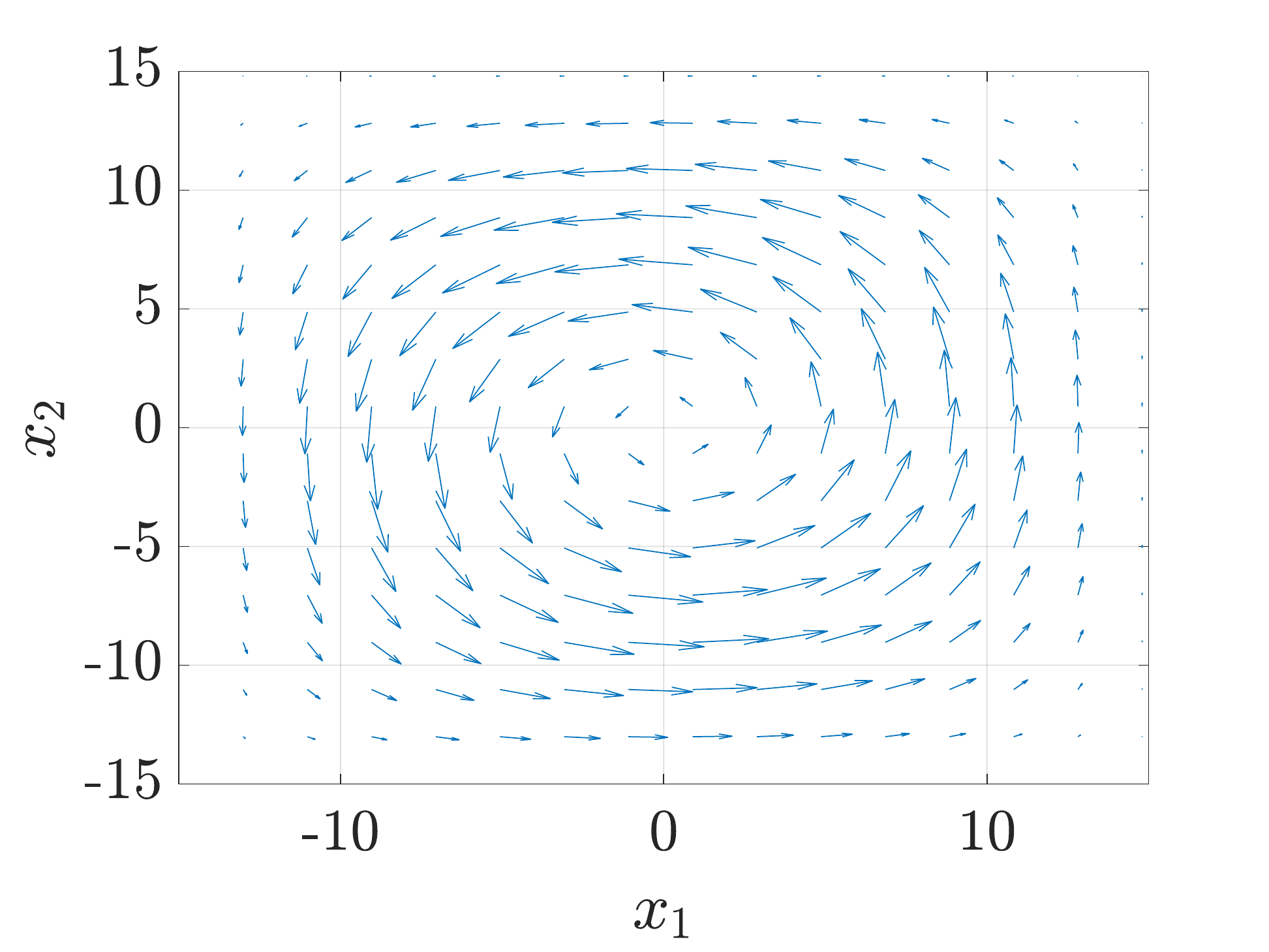} 
\includegraphics[scale=0.33]{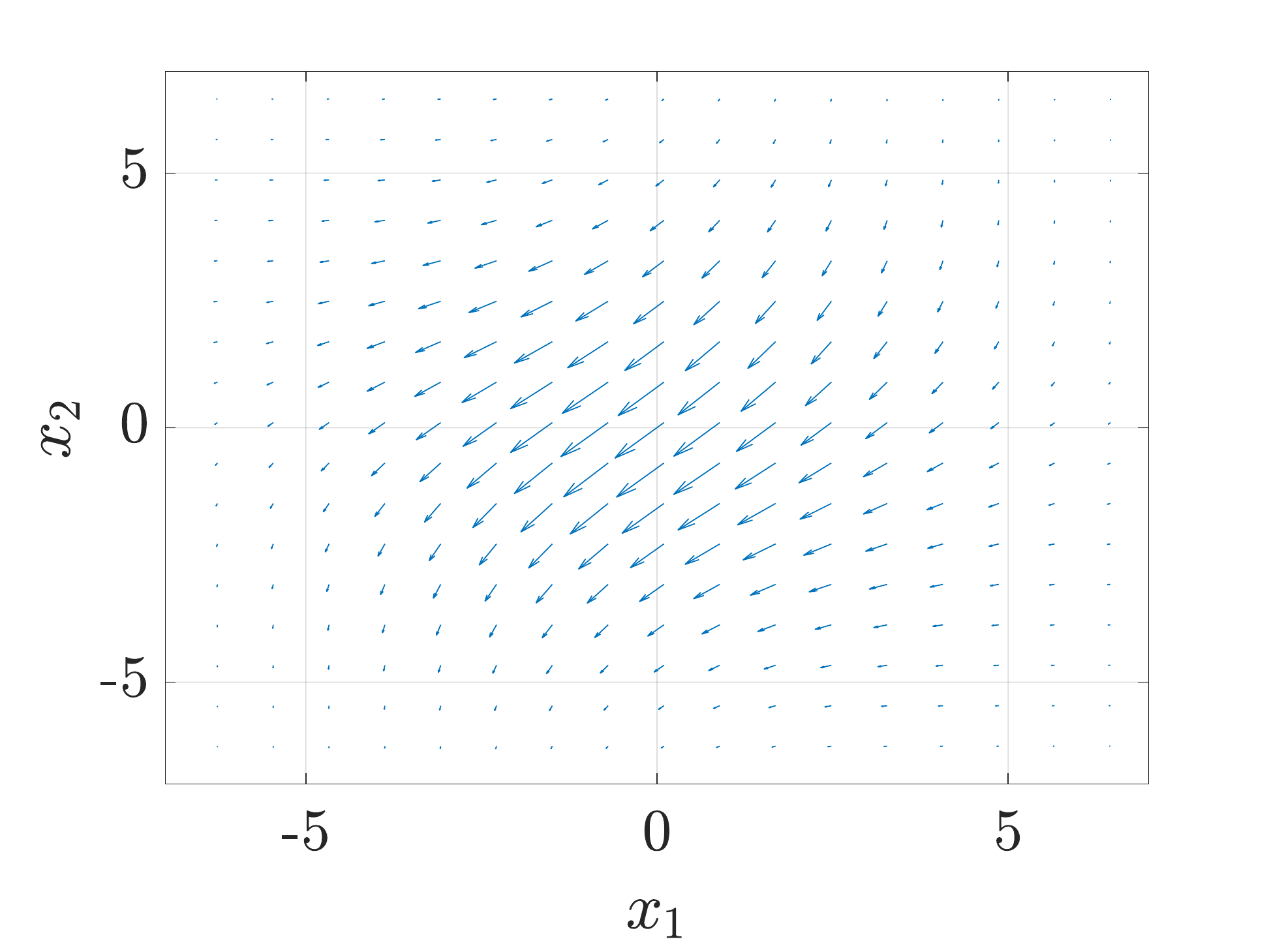} 
\caption{  (a) Vector fields 
used as coefficients in the two-dimensional linear advection 
equation \eqref{prototype-transport-PDE0}. 
The vector field defined in \eqref{vrtx_vf} is 
shown in (a) and the vector field defined in \eqref{vf2} 
is shown in (b).}
\label{fig:vf}
\end{figure}
%
%

%
%
We ran two simulations of \eqref{prototype-transport-PDE0} 
with coefficients \eqref{vf2} using the 
step-truncation FTT integrator 
\eqref{step_truncation_scheme} based 
on Adams-Bashforth 3 with step-size $\Delta t = 10^{-3}$, 
truncation tolerance $\delta = 10^{-6}$, and final integration 
time $t=5$. 
%
%
%
In the first simulation we solved the PDE 
with a step-truncation 
tensor method in fixed 
Cartesian coordinates.
In the second simulation we solved the 
PDE in coordinates $\bm y = \bm R^{-1} \bm x$, 
using the same step-truncation
tensor method. 
In order to verify the accuracy of our FTT simulations we also 
computed a benchmark solution on a full tensor product grid 
in two dimensions. 
%
%
We then mapped the transformed solution back to
Cartesian coordinates at each time step 
and compared it with the benchmark solution. 
We found that the global $L^{\infty}$ error of both 
low-rank simulations is bounded by 
$8 \times 10^{-4}$. 
In Figure \ref{fig:advection_solutions} 
we plot the FTT solution in 
Cartesian coordinates and 
the FTT-ridge solution 
in low-rank coordinates at 
time $t=5$. 
We observe that the PDE operator expressed in Cartesian coordinates 
advects the solution at an angle relative to the 
underling coordinate system while the 
operator expressed in coordinates $\bm y = \bm R^{-1} \bm x$ 
advects the solution directly along a coordinate axis, 
hence the low rank dynamics.
In Figure \ref{fig:2d_pde_ranks}(a) 
we plot the solution ranks versus time. 
Note that even though the operator in Cartesian coordinates 
has significantly larger rank than the operator in 
low-rank coordinates, the solution ranks follow 
the same trend during temporal integration with 
the FTT-ridge rank only slightly smaller 
than the FTT rank.}

{ 
\paragraph{Computational cost}
The CPU-time 
of integrating the advection equation 
\eqref{prototype-transport-PDE0} with 
coefficients \eqref{vf2} 
from $t = 0$ to $t=5$ is 72 seconds when computed 
with FTT in Cartesian coordinates and 
31 seconds when computed with FTT-ridge 
in low-rank coordinates. 
Note that although the ranks of the low-rank 
solutions are similar at each time step 
(Figure \ref{fig:2d_pde_ranks}(a)), the computational 
speed-up of the FTT-ridge simulation is 
due to the operator rank, which is $2$ for FTT-ridge 
and $16$ for FTT in Cartesian coordinates.}
\begin{figure}[!t]
\centerline{\footnotesize\hspace{.0cm} (a)  \hspace{6.5cm} (b) }
	\centering
\includegraphics[scale=0.41]{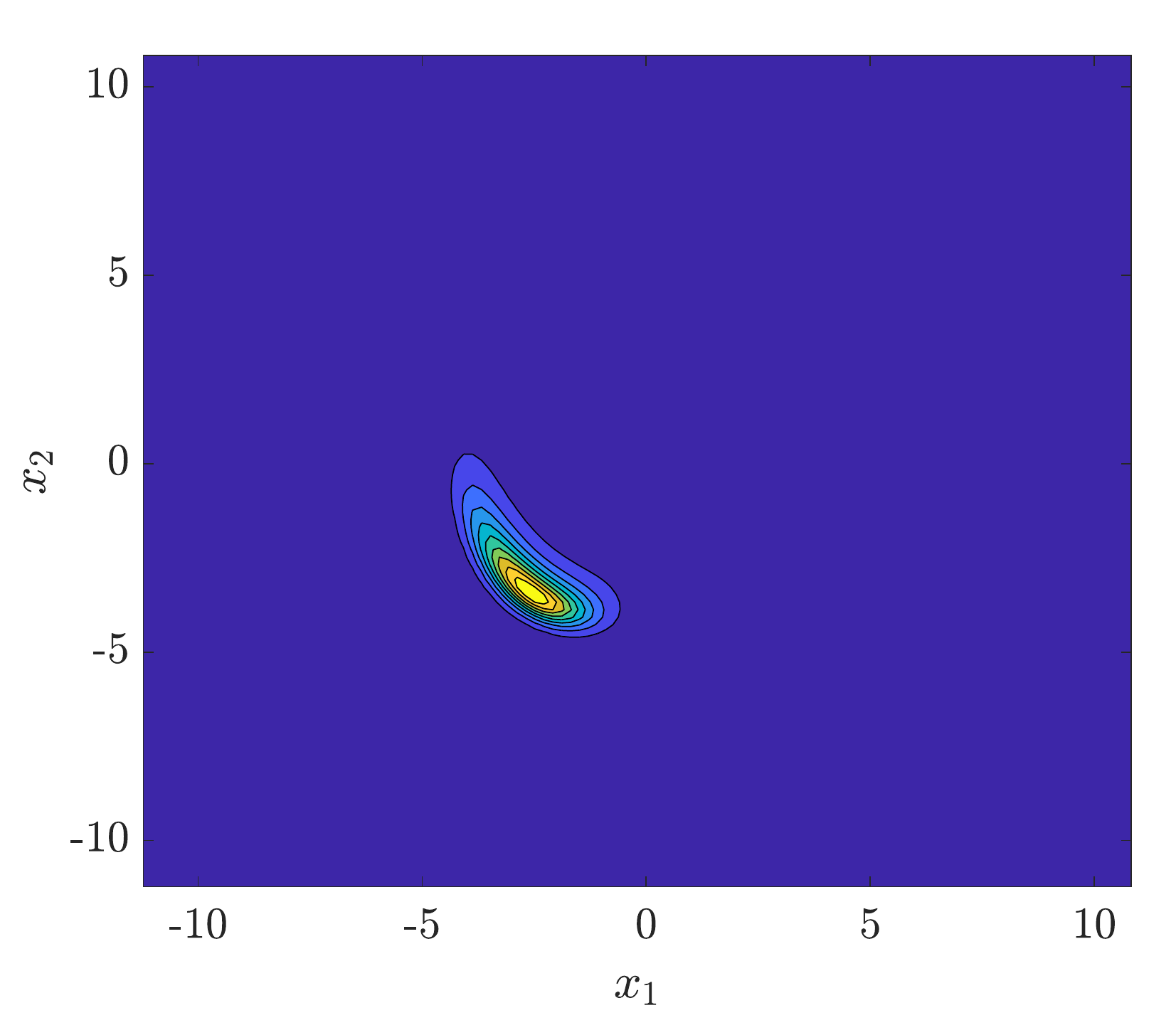} \hspace{-.2cm}
\includegraphics[scale=0.41]{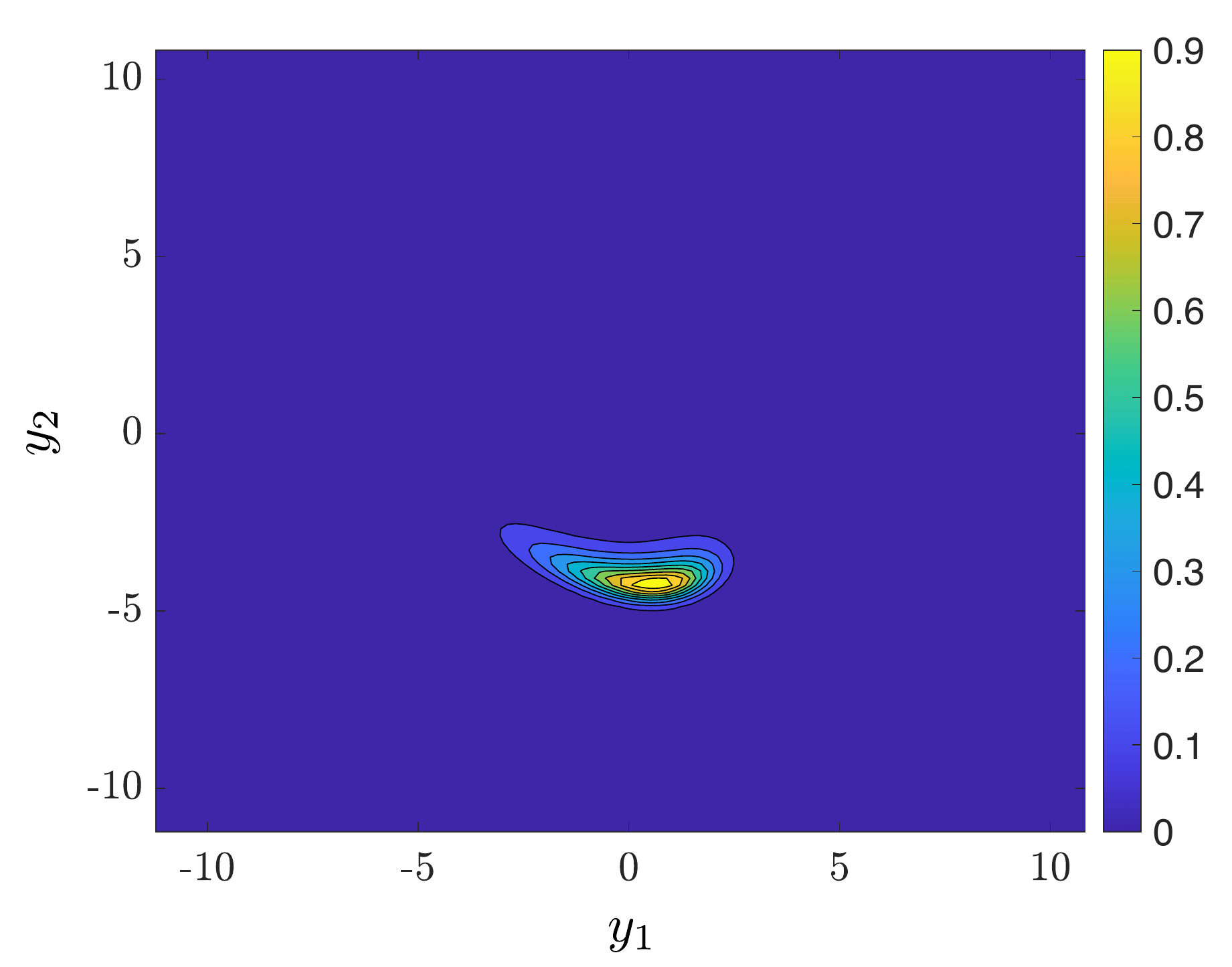} \\
\caption{  Solution to the two-dimensional advection equation \eqref{prototype-transport-PDE0} with coefficients \eqref{vf2} 
at time $t=5$. (a) FTT solution computed in Cartesian coordinates. 
(b) FTT-ridge computed low-rank coordinates.}
\label{fig:advection_solutions}
\end{figure}
{  
\subsection{Allen-Cahn equation}
\label{sec:Allen-Cahn}
Next we demonstrate coordinate-adaptive tensor 
integration on a simple nonlinear PDE. 
The Allen-Cahn eqaution is a reaction-diffusion PDE, which, in its simplest 
form includes a low-order polynomial non-linearity (reaction term) and 
a diffusion term \cite{Trefethen_2005} 
\begin{equation}
\label{Allen-Cahn}
\begin{cases}\displaystyle
\frac{\partial u(\bm x,t)}{\partial t} = \alpha \Delta u(\bm x,t) + u(\bm x,t) - u(\bm x,t)^3, \\
u(\bm x,0) = u_0(\bm x).
\end{cases}
\end{equation}
In two spatial dimensions the Laplacian in coordinates 
$\bm y = \bm \Gamma \bm x$ is given by 
%
\begin{equation}
\label{generalized_laplacian}
\begin{aligned}
\Delta_{\bm \Gamma} = 
&\left(\Gamma_{11}^2 + \Gamma_{12}^2 \right) \frac{\partial^2 }{\partial y_1^2} + 
\left(\Gamma_{21}^2 + \Gamma_{22}^2 \right) \frac{\partial^2 }{\partial y_2^2} +
2\left(\Gamma_{11} \Gamma_{21} + \Gamma_{12} \Gamma_{22} \right) \frac{\partial^2 }{\partial y_1 \partial y_2 }  
\end{aligned}
\end{equation}
which allows us to write the nonlinear PDE \eqref{Allen-Cahn} in 
the coordinate system $\bm y = \bm \Gamma \bm x$ 
with only a small increase in the rank of the 
Laplacian operator\footnote{
  In general a 
$d$-dimensional Laplacian 
$\Delta$ is a rank-$d$ operator and 
the corresponding operator 
$\Delta_{\bm \Gamma}$ in (linearly) 
transformed coordinates is rank $(d^2+d)/2$.}. 
The FTT rank of the cubic term appearing in the 
PDE operator of the Allen-Cahn equation 
\eqref{Allen-Cahn} 
is determined by the rank of the 
FTT solution. 
Standard algorithms for multiplying 
two FTT tensors $u_\T$ and $v_\T$ 
with ranks $\bm r_1$ and $\bm r_2$ 
results in a FTT tensor 
with (non-optimal) rank equal 
to the Hadamard product of the two 
ranks $\bm r_1 \circ \bm r_2$. 
Hence, by reducing the solution rank 
with a coordinate transformation, we can 
reduce the computational cost of computing 
the nonlinear term in \eqref{Allen-Cahn}. 
We set the diffusion coefficient $\alpha=0.2$ 
and the initial condition $u_0(\bm x)$ as 
the rotated Gaussian from equation 
\eqref{rotated_gaussian} with  $\epsilon = \pi/3$. 
%

%
We ran two FTT simulations of \eqref{Allen-Cahn} 
for time $t \in [0,5]$.
%
The first simulation was computed in fixed Cartesian 
coordinates. 
%
In the second simulation we used the coordinate transformation 
$\bm \Phi(\pi/3)^{-1}$ (which in this case we have available 
analytically) to transform the initial condition 
into a rank 1 FTT-ridge function. 
We integrated the rank 1 FTT-ridge initial condition forward 
in time using the corresponding transformed PDE, i.e., using the transformed Laplacian \eqref{generalized_laplacian}. 
In order to verify the accuracy of our low-rank simulations we 
also computed a benchmark solution on a full tensor product grid 
in two dimensions and mapped the transformed solution back to
Cartesian coordinates at each time step. 
%
%
%
%
We found that the global $L^{\infty}$ error of both low-rank 
solutions is bounded $5\times 10^{-5}$. 
In Figure \ref{fig:AC_solutions} 
we plot the FTT solution in 
Cartesian coordinates 
and the FTT-ridge solution in 
low-rank coordinates at 
time $t=5$. 
We observe that the profile of 
Gaussian functions are preserved 
as the solution moves from the unstable 
equilibrium at $u=0$ to the stable equilibrium 
at $u=1$. 
\begin{figure}[!t]
\centerline{\footnotesize\hspace{.0cm} (a) \hspace{6.6cm} (b)}
	\centering
\includegraphics[scale=0.41]{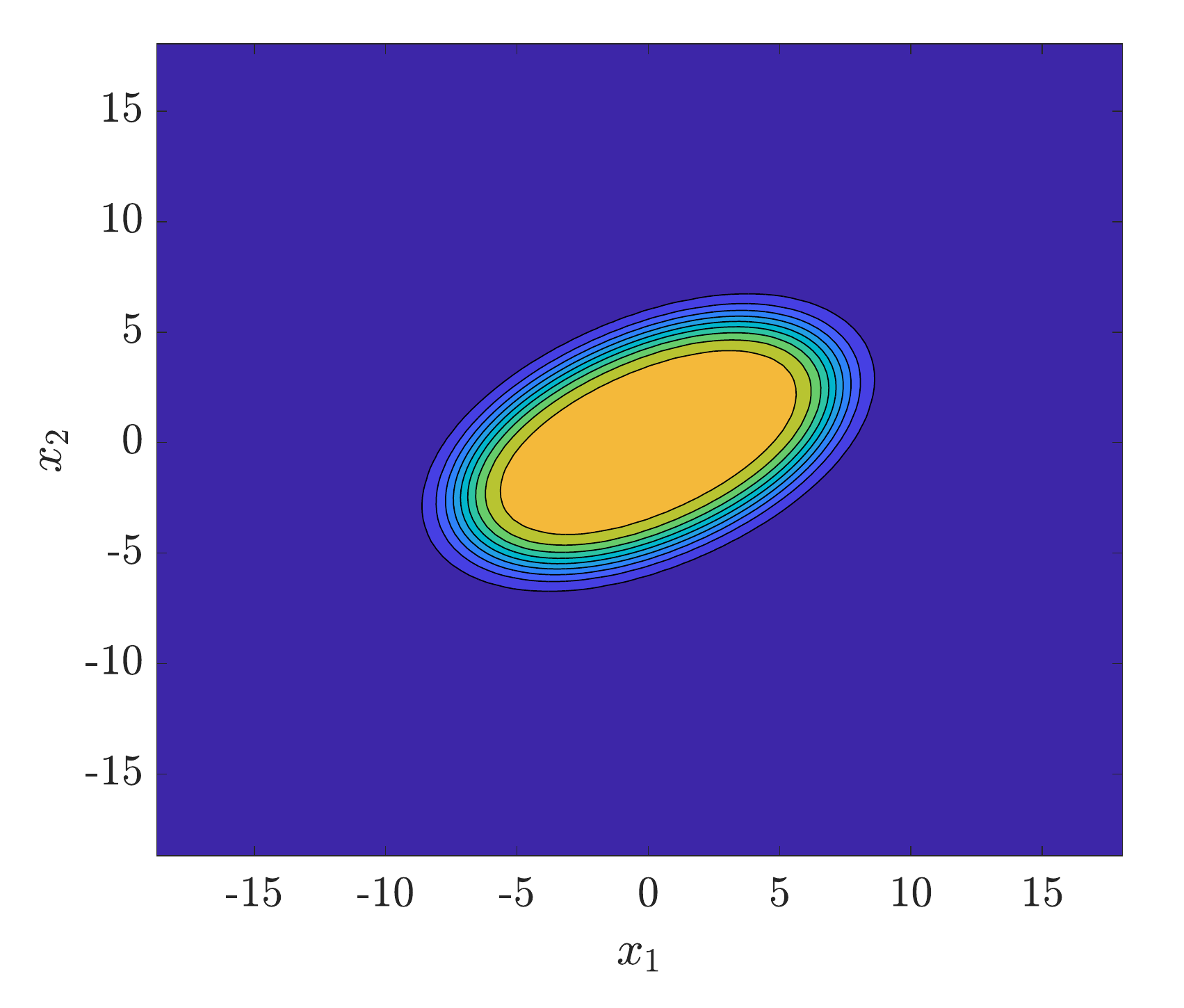} \hspace{-.2cm}
\includegraphics[scale=0.41]{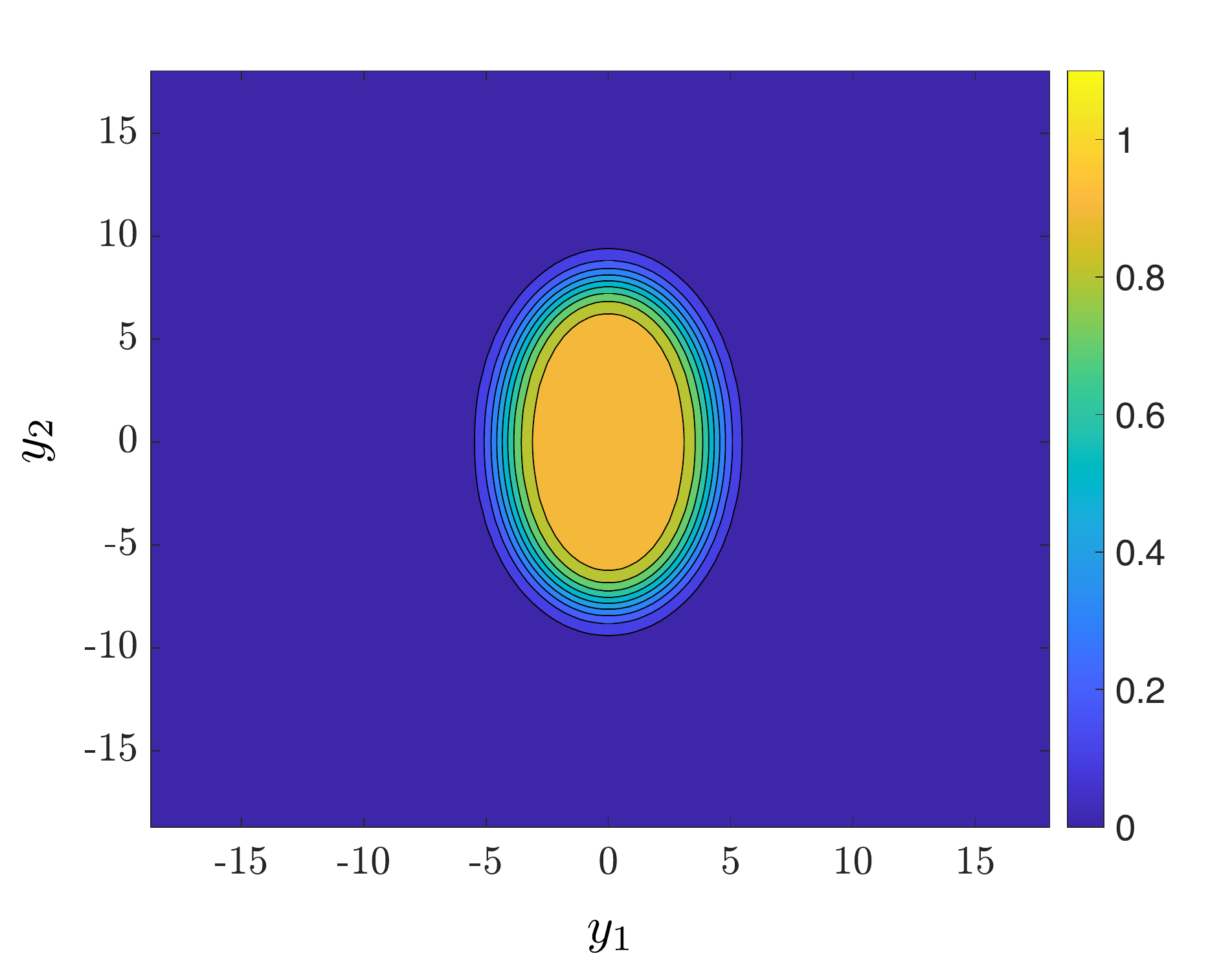} 
\caption{ 
Solution to the two-dimensional Allen-Cahn equation \eqref{Allen-Cahn} 
at time $t=5$. (a) FTT solution computed in Cartesian coordinates and (b) FTT-ridge solution computed in low-rank coordinates.}
\label{fig:AC_solutions}
\end{figure}
In Figure \ref{fig:2d_pde_ranks}(b) 
we plot the solution ranks versus time. 
We observe that the FTT solution rank is larger 
than the FTT-ridge solution rank at each time. 
}
{ 
\paragraph{Computational cost}
The CPU-time of integrating the Allen-Cahn equation 
\eqref{Allen-Cahn} from $t = 0$ to $t=5$ is 279 seconds 
when computed using FTT in Cartesian coordinates and 
72 seconds when computed using FTT-ridge in low-rank 
coordinates. 
The optimal coordinate transformation at time $t=0$ is known 
analytically so we do not need to compute it, thus 
these computational timings only include the temporal integration, 
and do not account for any computational time of changing the 
coordinate system. 
A significant amount of computational time in 
computing the FTT solutions comes from computing the cubic 
nonlinearity appearing in the Allen-Cahn equation at each 
time step. 
The lower rank FTT-ridge solution allows for this term to 
be computed significantly faster than the FTT solution 
in Cartesian coordinates. 
}

{ 
\subsection{3D and 5D linear advection equations} 
We also applied the rank-reducing coordinate-adaptive FTT integrators 
to the advection equation 
\begin{equation}
\label{Fokker_Planck}
\begin{cases}
\displaystyle\frac{\partial u(\bm x,t)}{\partial t} = 
\bm f(\bm x) \cdot \nabla u(\bm x,t),\vs  \\
u(\bm x,0) = u_0(\bm x),
\end{cases}
\end{equation}
in dimensions three and five with initial condition
$u_0(\bm x)$ defined as a Gaussian mixture 
\begin{equation}
\label{PDF_IC}
u_0(\bm x) = \frac{1}{m}
\sum_{i=1}^{N_g}\exp\left(- \sum_{j=1}^d 
\frac{1}{\beta_{ij}} \left( \bm R^{(i)}_{j} \cdot \bm x + t_{ij} \right)^2 \right) , 
\end{equation}}
where $\bm R_j{ ^{(i)}}$ is the $j$-th row of a $d\times d$ rotation 
matrix $\bm R{ ^{(i)}}$, $\beta_i \geq 0$, {  $t_{ij}$ are 
translations} and $m$ is the normalization factor 
{ 
\begin{equation}
\nonumber
m = \left\| \sum_{i=1}^{N_g}\exp\left(- \sum_{j=1}^d 
\frac{1}{\beta_{ij}} \left( \bm R^{(i)}_{j} \cdot \bm x + t_{ij} \right)^2 \right) \right\|_{L^1\left(\mathbb{R}^d\right)}.
\end{equation}}

\begin{figure}[!t]
\centerline{\footnotesize\hspace{.75cm} (a)  \hspace{8.4cm} (b) }
	\centering
\includegraphics[scale=.45]{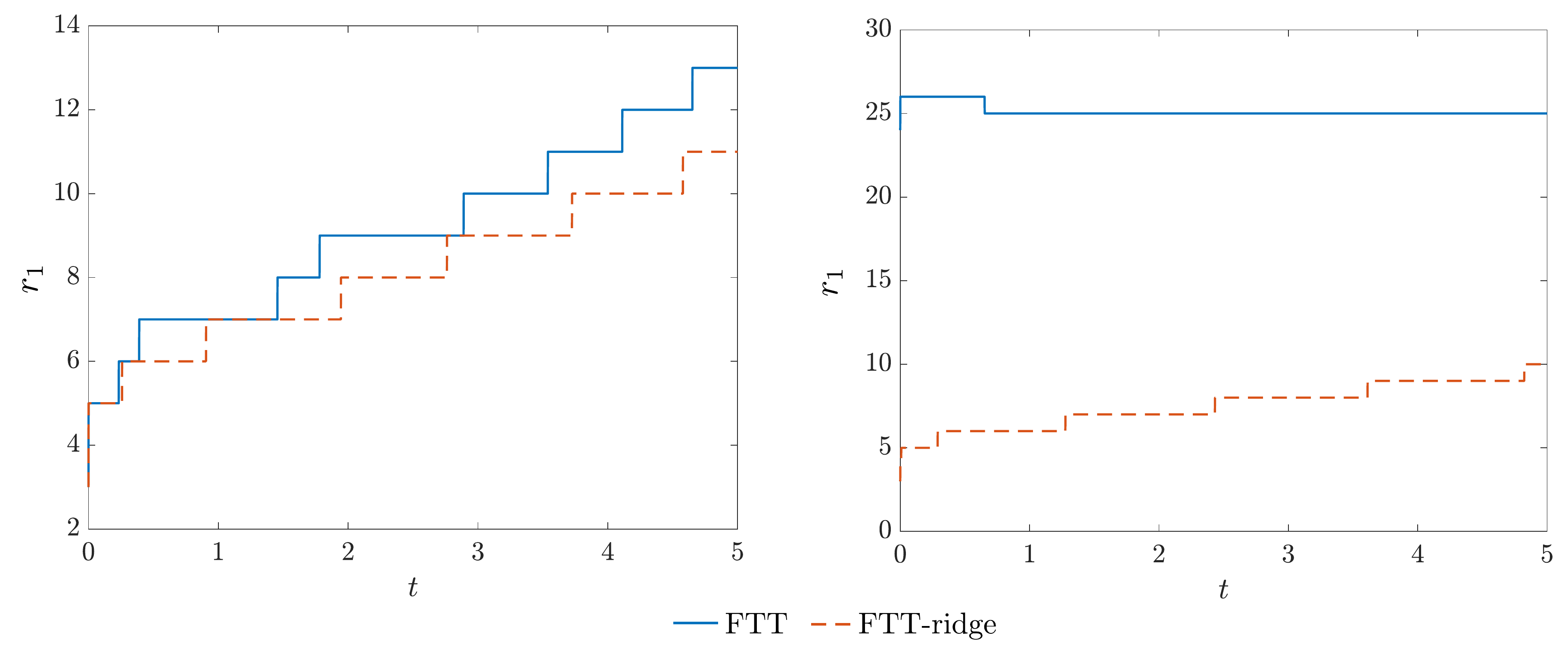} 
\caption{  (a) Rank versus time for the FTT and FTT-ridge solutions to the advection PDE \eqref{prototype-transport-PDE0} with coefficients \eqref{vf2}. 
(b) Rank versus time for the FTT and FTT-ridge solutions to the 
Allen-Cahn equation \eqref{Allen-Cahn}.}
\label{fig:2d_pde_ranks}
\end{figure}
\subsubsection{Three-dimensional simulation results}
First, we consider three spatial dimensions 
{  ($d=3$)}, and 
set the coefficients in \eqref{Fokker_Planck} as 
\begin{equation}
\label{drift_diffusion}
{  \bm f(\bm x)} = 
-\frac{1}{6} \begin{bmatrix}
2\sin(x_2) \\
3\cos(x_3) \\
3 x_1 \\
\end{bmatrix}, 
\end{equation}
resulting in the linear operator 
\begin{figure}[!t]
\centerline{\footnotesize\hspace{.85cm} $t = 0$ \hspace{1.65cm}    \hspace{5.5cm} $t = 1$ }
	\centering
	\rotatebox{90}{\hspace{1.5cm}\footnotesize FTT (Cartesian coordinates)}
\includegraphics[scale=0.41]{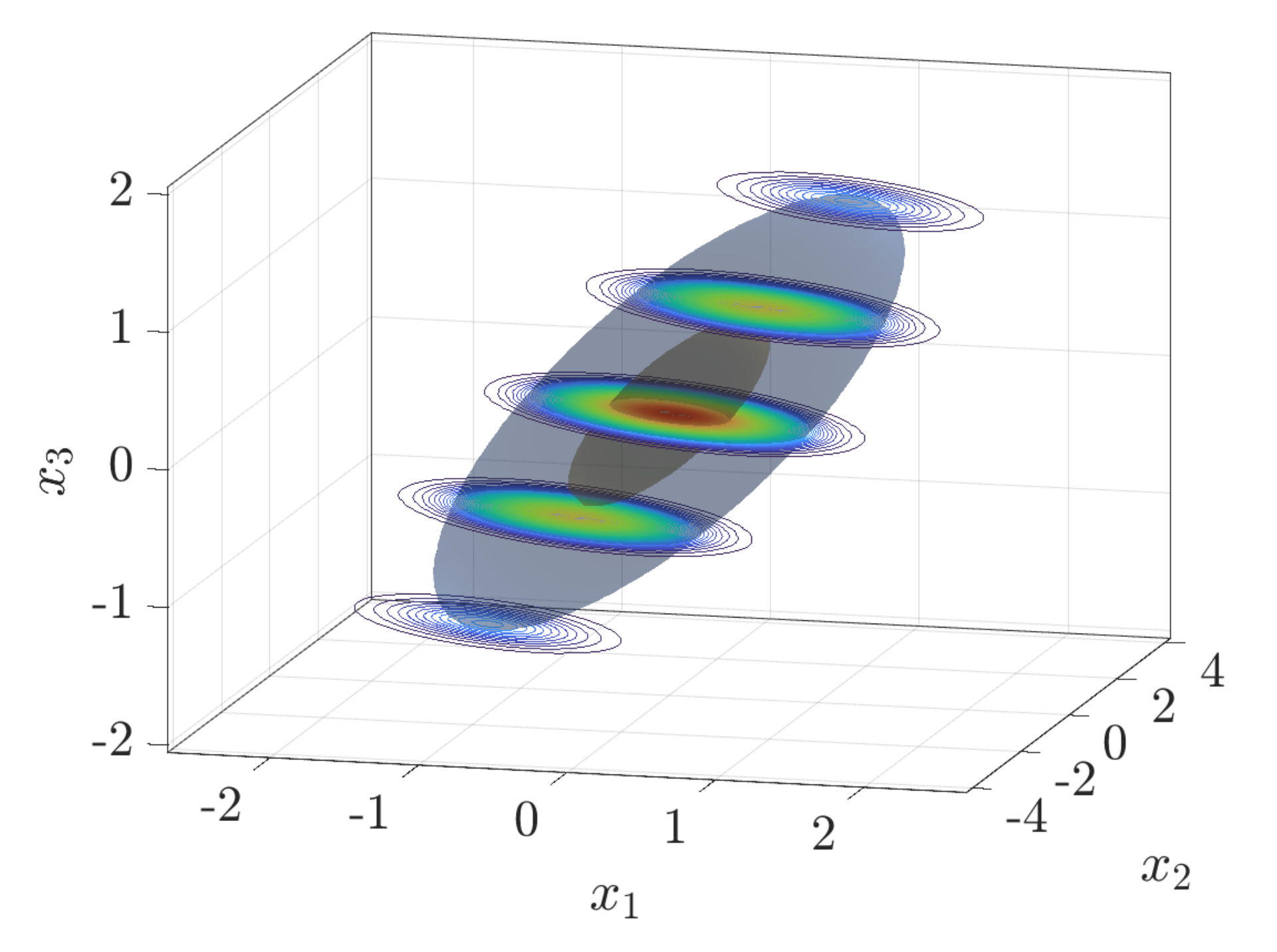} \hspace{-.2cm}
\includegraphics[scale=0.41]{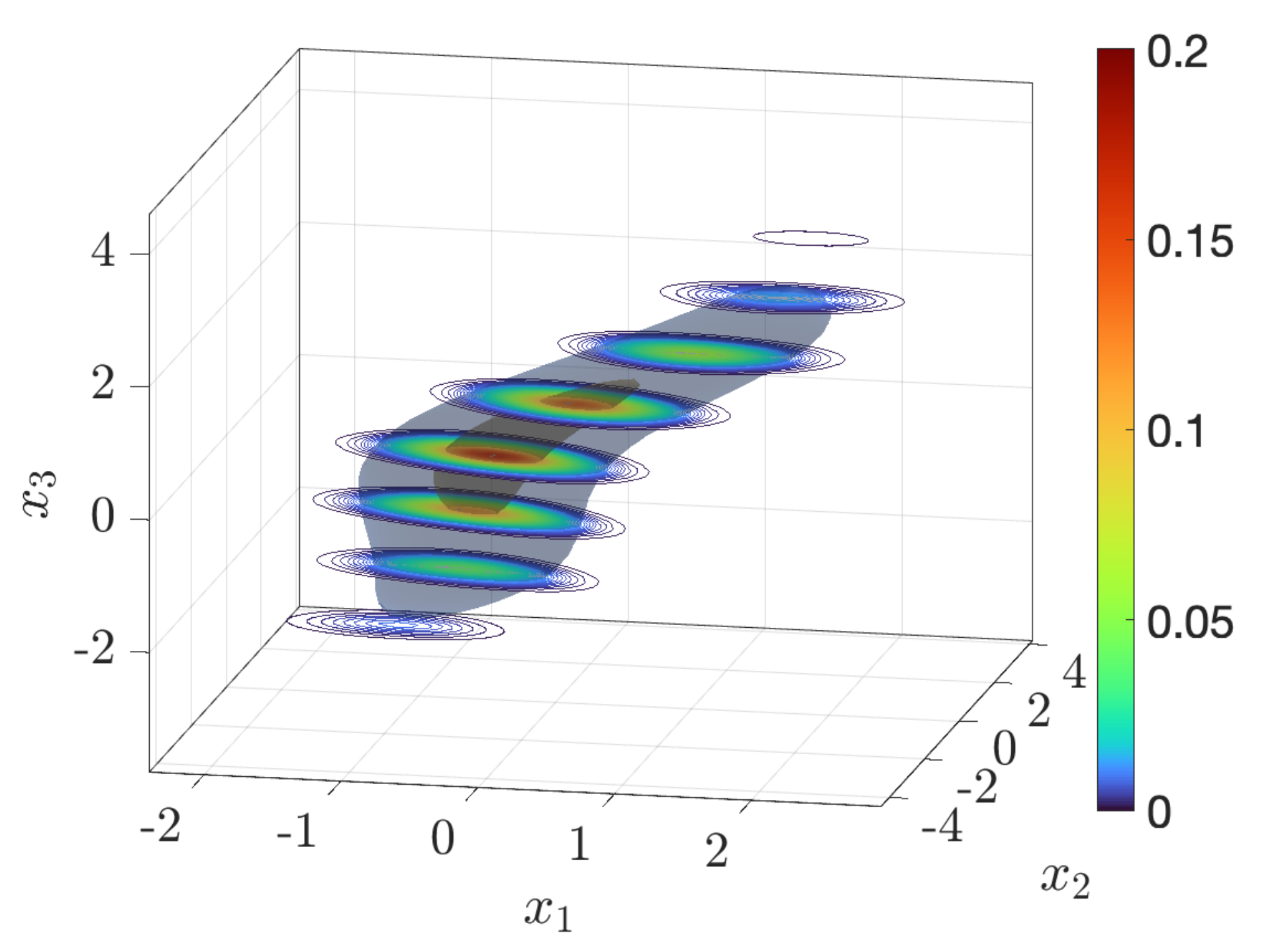} \\
	\centering
	\rotatebox{90}{\hspace{1cm}\footnotesize FTT-ridge (low-rank  coordinates)}
\includegraphics[scale=0.41]{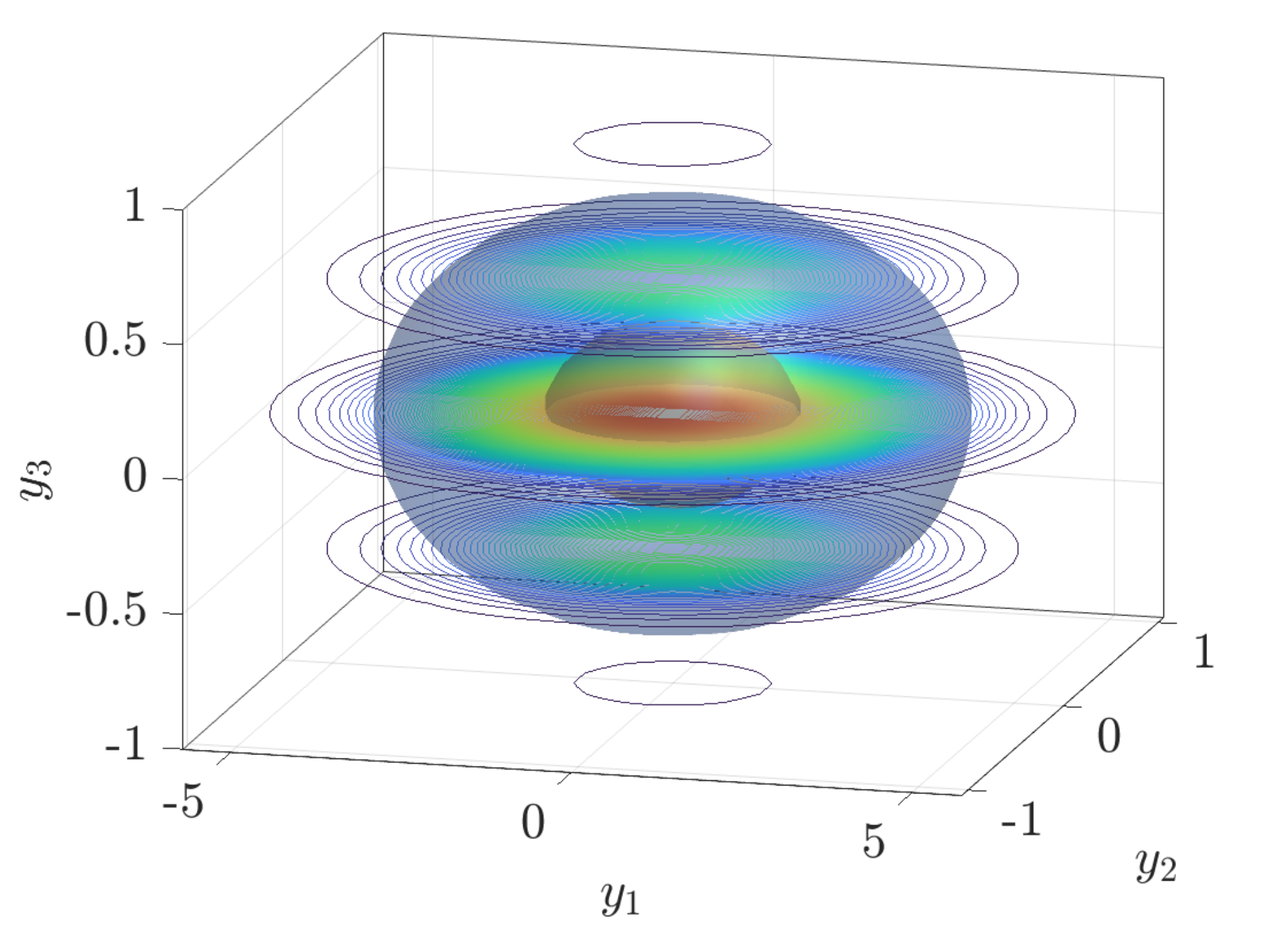} \hspace{-.3cm}
\includegraphics[scale=0.41]{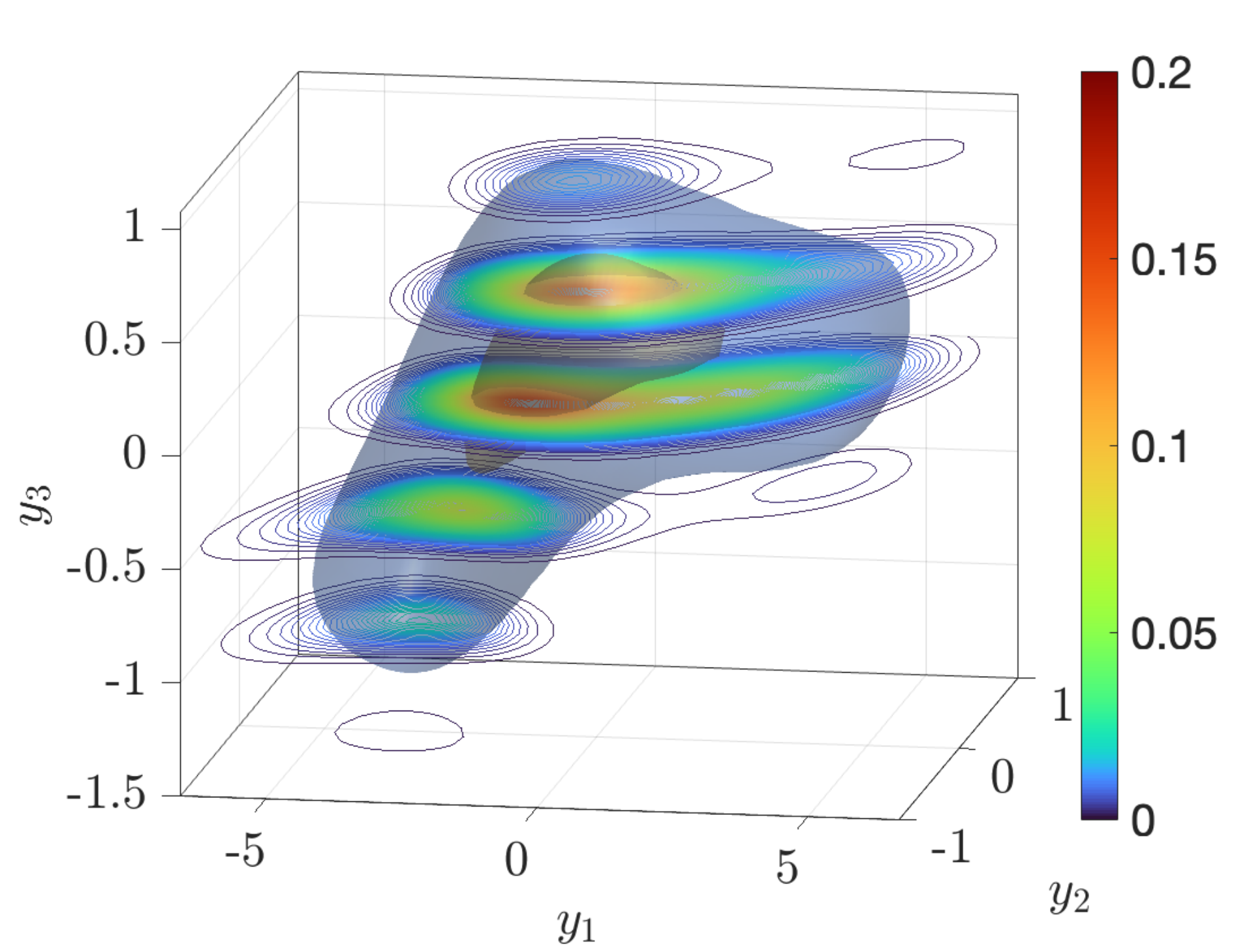} \\
\caption{Volumetric plot of the FTT solutions to the 
{  3D advection equation} \eqref{Fokker_Planck} at time 
$t = 0$ (left column) in Cartesian coordinates (top) and 
low-rank coordinates (bottom), and at time $t=1$ (right column) in Cartesian coordinates (top) and low-rank coordinates (bottom).}
\label{fig:PDE_countours}
\end{figure}
\begin{equation}
\label{Fokker_Planck_op}
\begin{aligned}
{  \bm f(\bm x) \cdot \nabla } = \frac{\sin(x_2)}{3} \frac{\partial }{\partial x_1} 
+ \frac{\cos(x_3)}{2} \frac{\partial }{\partial x_2} + \frac{x_1}{2} \frac{\partial }{\partial x_3}.
\end{aligned}
\end{equation}
In {  a} previous work 
\cite{Dektor_2020} we have demonstrated that 
variable coefficient advection problems with operators of 
the form \eqref{Fokker_Planck_op} 
can have solutions with multilinear rank that 
grows significantly over time. 
We set the parameters in the 
initial condition \eqref{PDF_IC} 
{  $N_g = 1$, }
\begin{equation}
\label{rotation_3D}
\bm R^{  (1)} = \exp\left(\frac{1}{28} \begin{bmatrix}
0 & 7 \pi & 4 \pi \\
-7 \pi  & 0 & 7 \pi \\
  -4 \pi & -7 \pi & 0
\end{bmatrix} \right), 
\qquad
\bm \beta = \begin{bmatrix}
  3 &  {1}/{10}  &  3 
\end{bmatrix}, 
\end{equation}
{  and $t_{ij} = 0$ for all $i,j$.} 
\begin{figure}[!t]
\centerline{\footnotesize\hspace{.5cm} (a) \hspace{5cm} (b) \hspace{5cm} (c) }
	\centering
\includegraphics[scale=0.29]{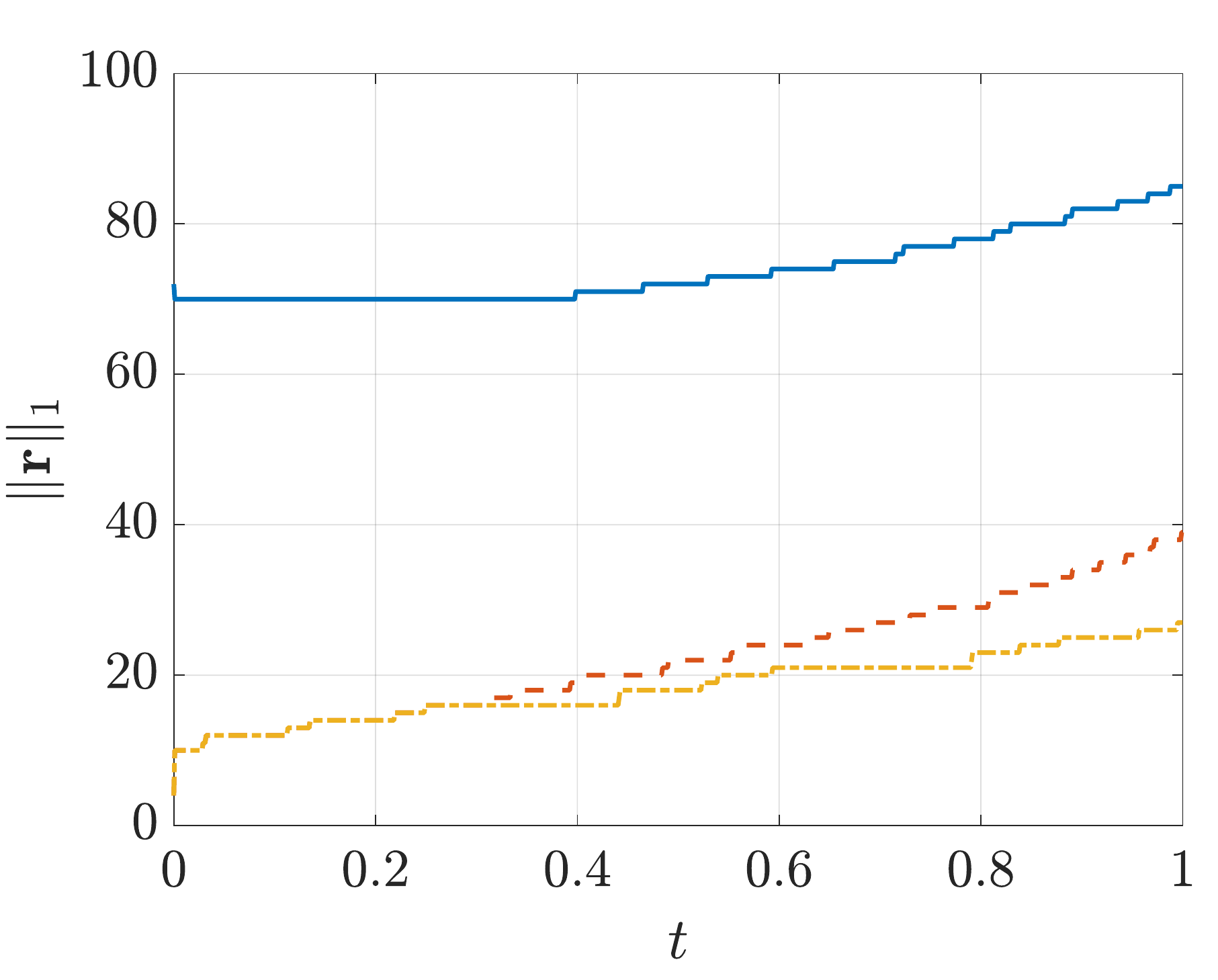} 
\includegraphics[scale=0.29]{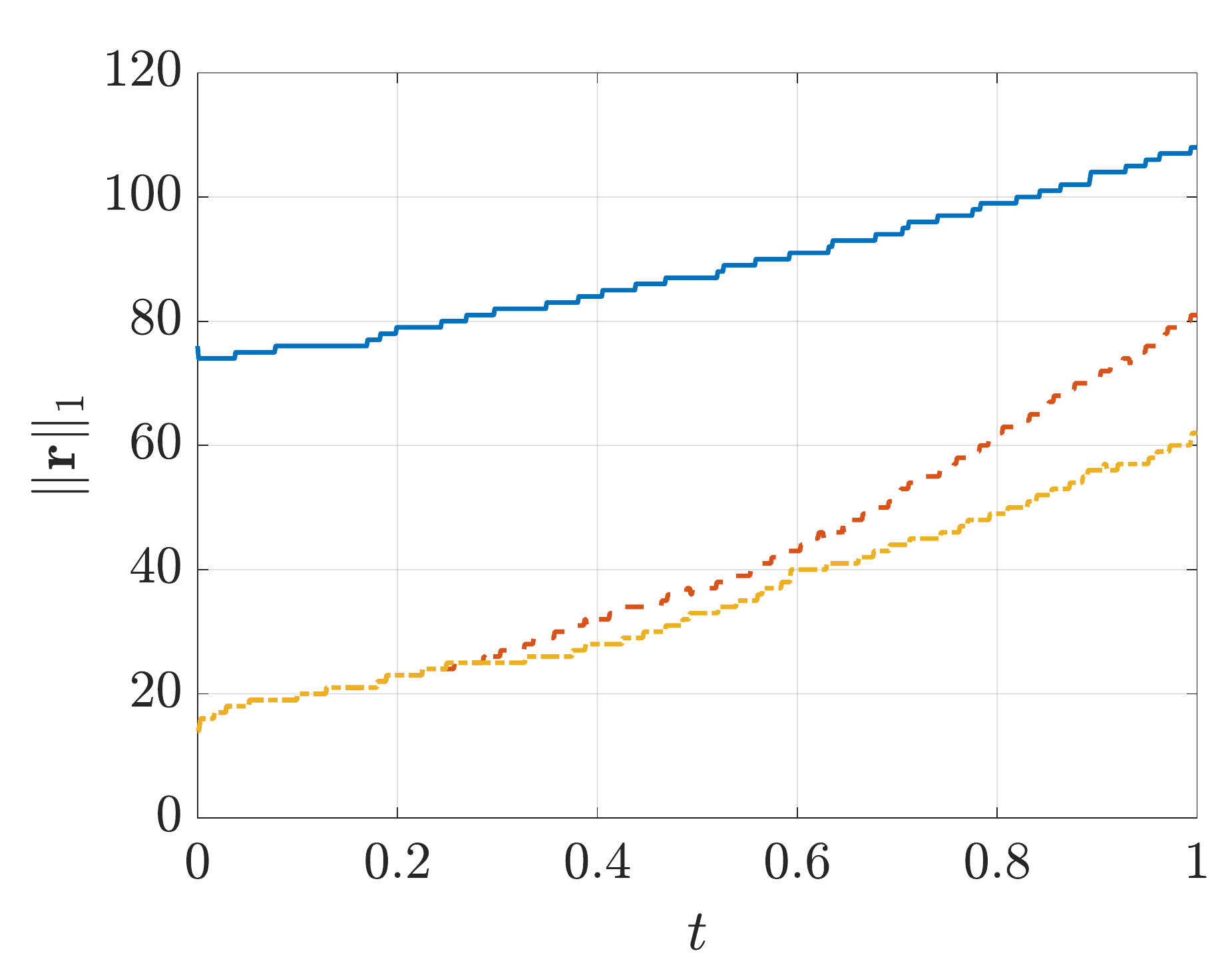} 
\includegraphics[scale=0.29]{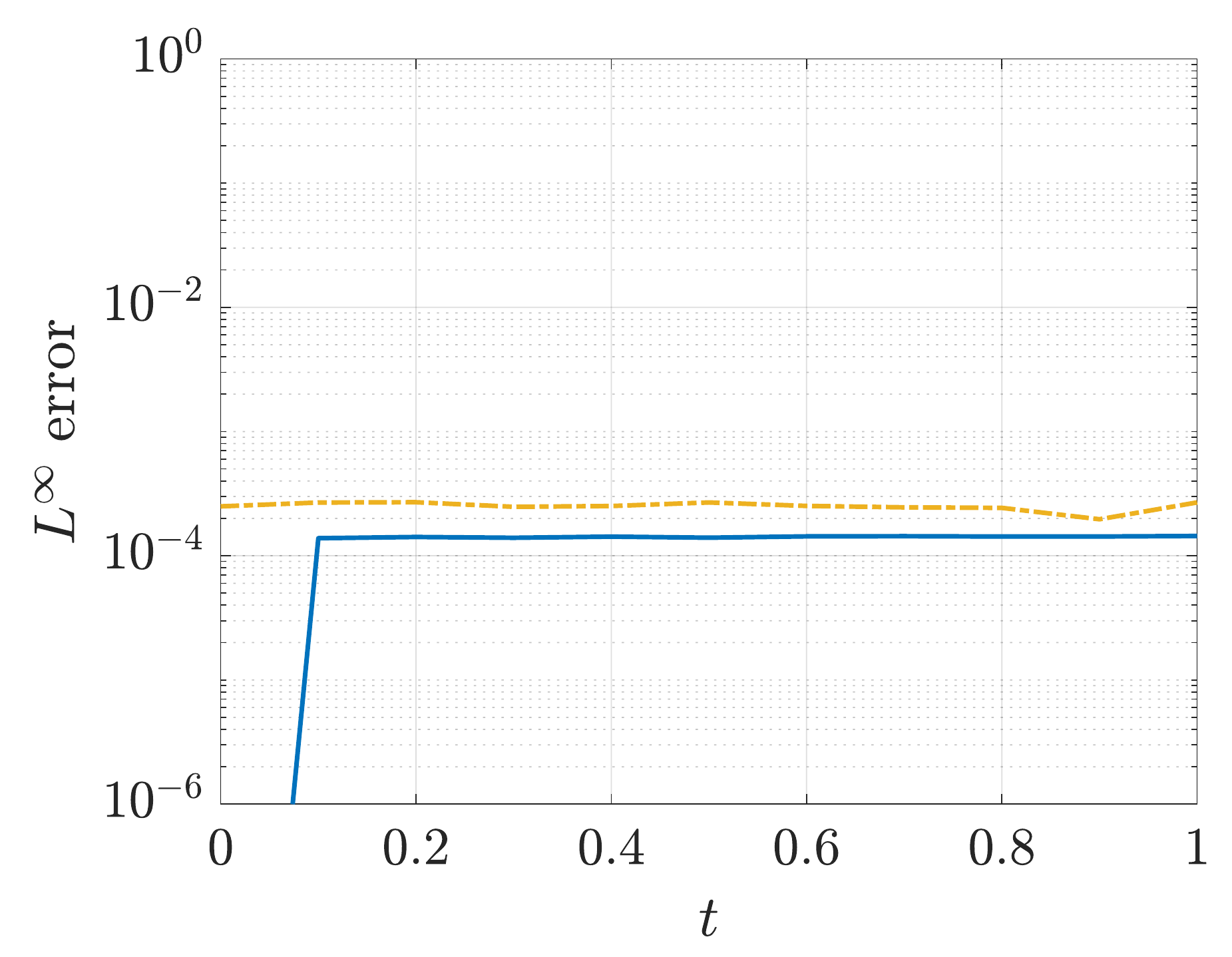} \\
\vspace{0cm}
\hspace{.3cm}
\includegraphics[scale=0.35]{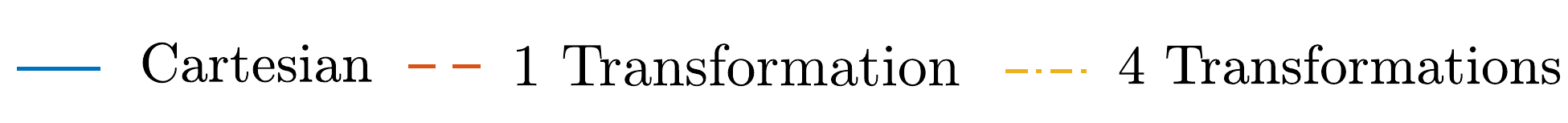} 
\caption{  3D advection 
equation \eqref{Fokker_Planck}: 
Multilinear rank of the PDE solution (a) and PDE 
right-hand-side (b) in Cartesian coordinates 
and transformed coordinates versus time. 
In (c) we plot the $L^{\infty}$ error of the FTT solutions relative 
to a benchmark solution.}
\label{fig:PDE_rank_vs_time}
\end{figure}

%
%
%
We ran three FTT simulations up to 
time $t = 1$. 
The first simulation is computed 
in fixed Cartesian coordinates. 
We then tested the FTT integrator with rank-reducing coordinate 
transformation in two different simulation settings. 
{  In the first one, we performed a 
coordinate transformation only at time $t=0$.}
Such a coordinate transformation is not done using 
the Riemannian gradient descent algorithm, since, 
in this case we have the optimal coordinate 
transformation available analytically and we can 
simply evaluate the FTT tensor on the 
low rank coordinates. 
{ 
In the second simulation we 
also performed a coordinate transformation 
at time $t=0$ 
(once again the coordinate transformation 
at time $t=0$ is not done using 
the Riemannian gradient descent algorithm) 
and then used the coordinate-adaptive integration Algorithm 
\ref{alg:coord_adaptive_integrator} with 
${\rm max\ rank} = 15$ and $k_r = 5$. 
With these parameters the coordinate-adaptive 
algorithm triggers three additional coordinate 
transformations at times $t \in\{0.25,0.59,0.9\}$ 
that are computed using 
the Riemannian gradient 
descent algorithm \ref{alg:gradient_descent} 
with step size $\Delta \epsilon = 10^{-4}$ and 
stopping tolerance ${  \eta} = 10^{-1}$. 
In Figure \ref{fig:PDE_spectra_final_time}(c) 
we plot the absolute value of the derivative of the 
cost function $C(\bm \Gamma(\epsilon))$ 
versus $\epsilon$ for {  the instances of 
gradient descent at times $t > 0$. }
We observe that the rate of change of the cost function 
becomes smaller as we iterate the gradient 
descent routine, i.e., the cost function is decreasing 
less per iteration after several iterations. 
This indicates that the cost function is approaching 
a flatter region and our gradient descent method 
is becoming less effective for reducing the cost function.
In Figure \ref{fig:PDE_rank_vs_time}(a) we plot the 
$1$-norm of the FTT solution rank vector versus 
time and in Figure \ref{fig:PDE_rank_vs_time}(b) 
we plot the $1$-norm of the FTT solution velocity 
(i.e., the PDE right hand side) for each FTT simulation. 
We observe that the FTT-ridge solutions and right 
hand side of the PDE have rank that is smaller than the 
corresponding ranks of the FTT solution in Cartesian 
coordinates.
We also observe that the adaptive coordinate 
transformations performed at times $t>0$ do not 
reduce the solution rank at the time of application,
but they do slow the rank increase as time integration 
proceeds. 
In Figure \ref{fig:PDE_spectra_final_time}(a)-(b) we 
plot the singular values of the FTT solutions at final 
time and note that both FTT-ridge solutions 
have singular values that decay significantly 
faster than the FTT solution in Cartesian coordinates. 
Moreover, the additional coordinate transformations 
performed by the coordinate-adaptive integrator 
causes the singular values of the FTT-ridge solution to decay 
faster than the other FTT-ridge solution that used only
one coordinate transformation at $t=0$. 
In Figure \ref{fig:PDE_countours} we provide 
volumetric plots of the PDE solution in Cartesian 
coordinates and in the transformed coordinate 
system computed with the coordinate-adaptive 
algorithm \ref{alg:coord_adaptive_integrator} 
at time $t = 1$. 
We observe that the reduced rank tensor ridge 
solution appears to be more symmetrical with 
respect to the underlying coordinate axes 
than the solution in Cartesian coordinates. 
}

{  It is important to note that the operator 
$G(\cdot,\bm x) = \bm f(\bm x) \cdot \nabla$ in \eqref{Fokker_Planck_op} 
is a separable operator of rank 
$\bm g = \begin{bmatrix} 1 & 3 & 3 & 1 \end{bmatrix}$. 
For a general linear coordinate transformation $\bm \Gamma$ 
the operator $G_{\bm \Gamma}$ 
(see \eqref{advection_operator_transformed}) 
acting in the transformed coordinate system can be obtained 
using trigonometric identities with (non-optimal) 
separation ranks. 
A more efficient representation of 
$G_{\bm \Gamma}$ 
can be obtained by FTT compression and 
then splitting $G_{\bm \Gamma}$ 
\eqref{split_operator} 
into a sum of three operators 
\begin{equation}
\label{3D_operator_splitting}
G_{\bm \Gamma} = G_{\bm \Gamma}^{(1)}  + 
G_{\bm \Gamma}^{(2)} + G_{\bm \Gamma}^{(3)}  , 
\end{equation}
where $G_{\bm \Gamma}^{(i)}$ 
have ranks 
$\bm g_{\bm \Gamma}^{(i)} = \begin{bmatrix} 1 & 4 & 4 & 1 \end{bmatrix}$ for each $i=1,2,3$. 
Then we apply the operator $G_{\bm \Gamma}$ to 
the FTT-ridge solution at each time 
using the procedure summarized in 
\eqref{split_operator_applied}. 
In this case, since the PDE operator $G$ in 
Cartesian coordinates is separable and low-rank, 
it is not surprising that a 
linear coordinate transformation 
increases the PDE operator rank. 
However by splitting the operator such as 
\eqref{3D_operator_splitting} and performing a FTT 
truncation operation after applying 
each lower rank operator $G_{\bm \Gamma}^{(i)}$ 
we can mitigate the computational cost.
}

\begin{figure}[!t]
\centerline{\footnotesize\hspace{1.3cm} (a)  \hspace{4.8cm}    (b) \hspace{5.05cm} (c) \hspace{.7cm}}
	\centering
\includegraphics[scale=0.27]{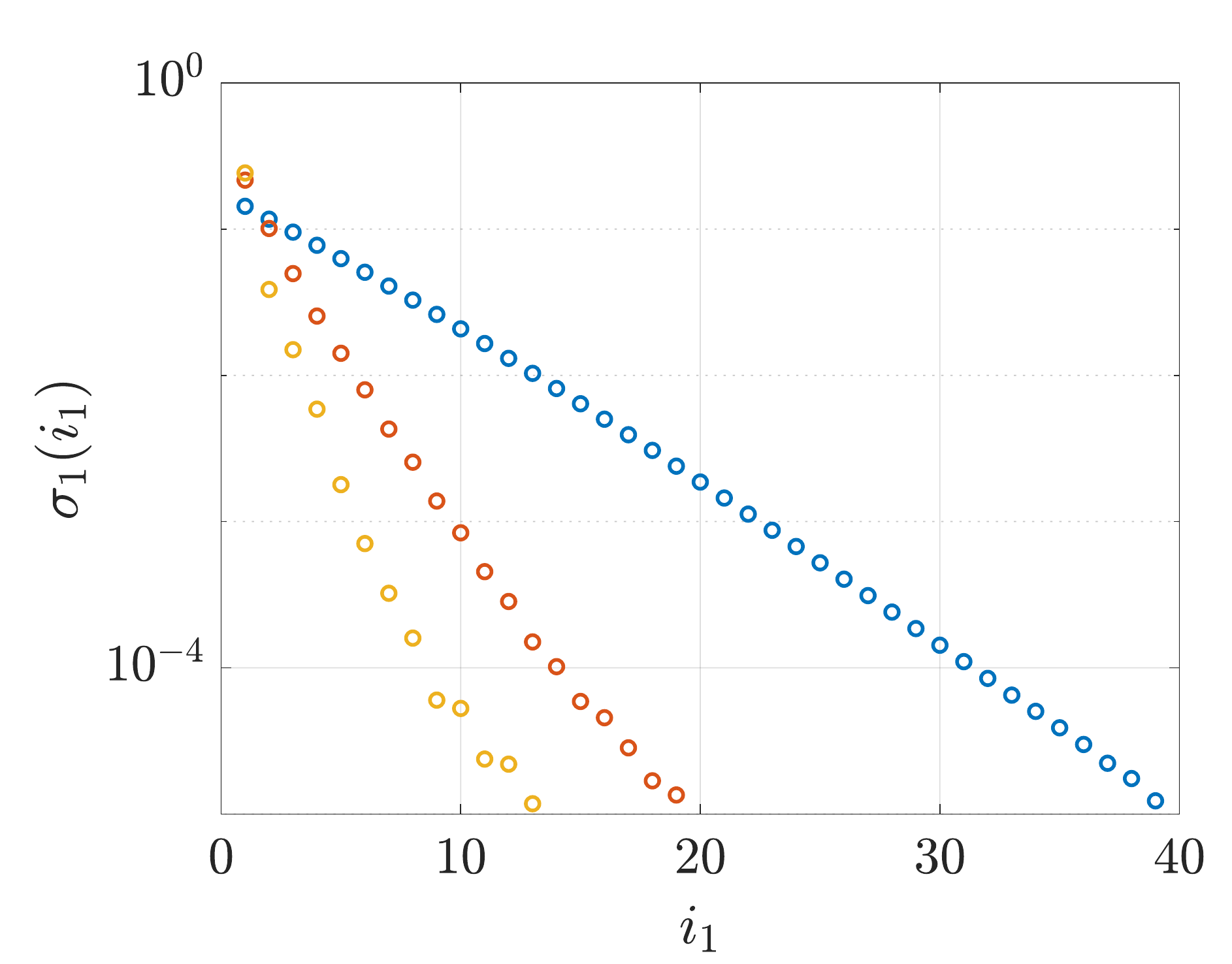} 
\includegraphics[scale=0.27]{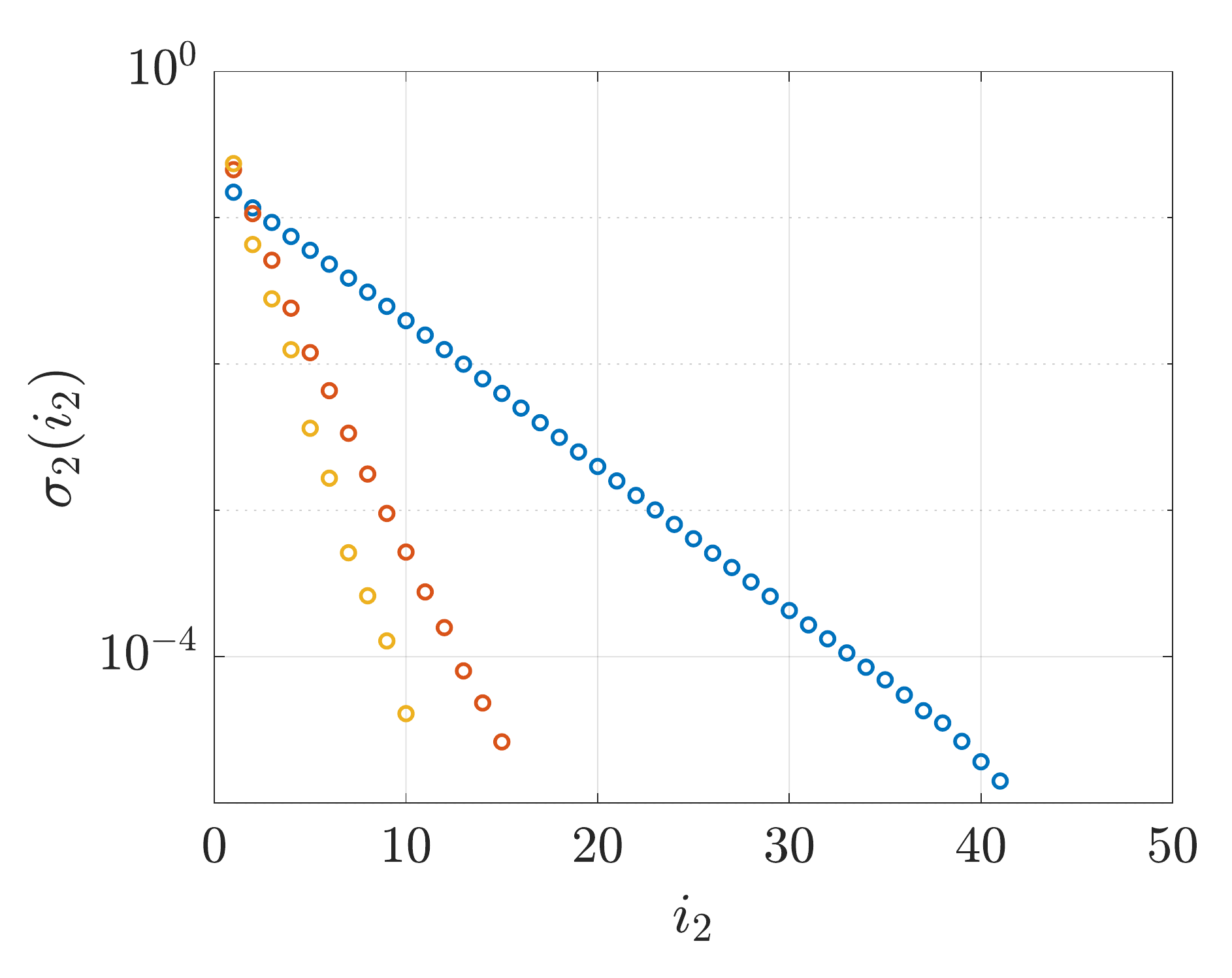} 
\includegraphics[scale=0.27]{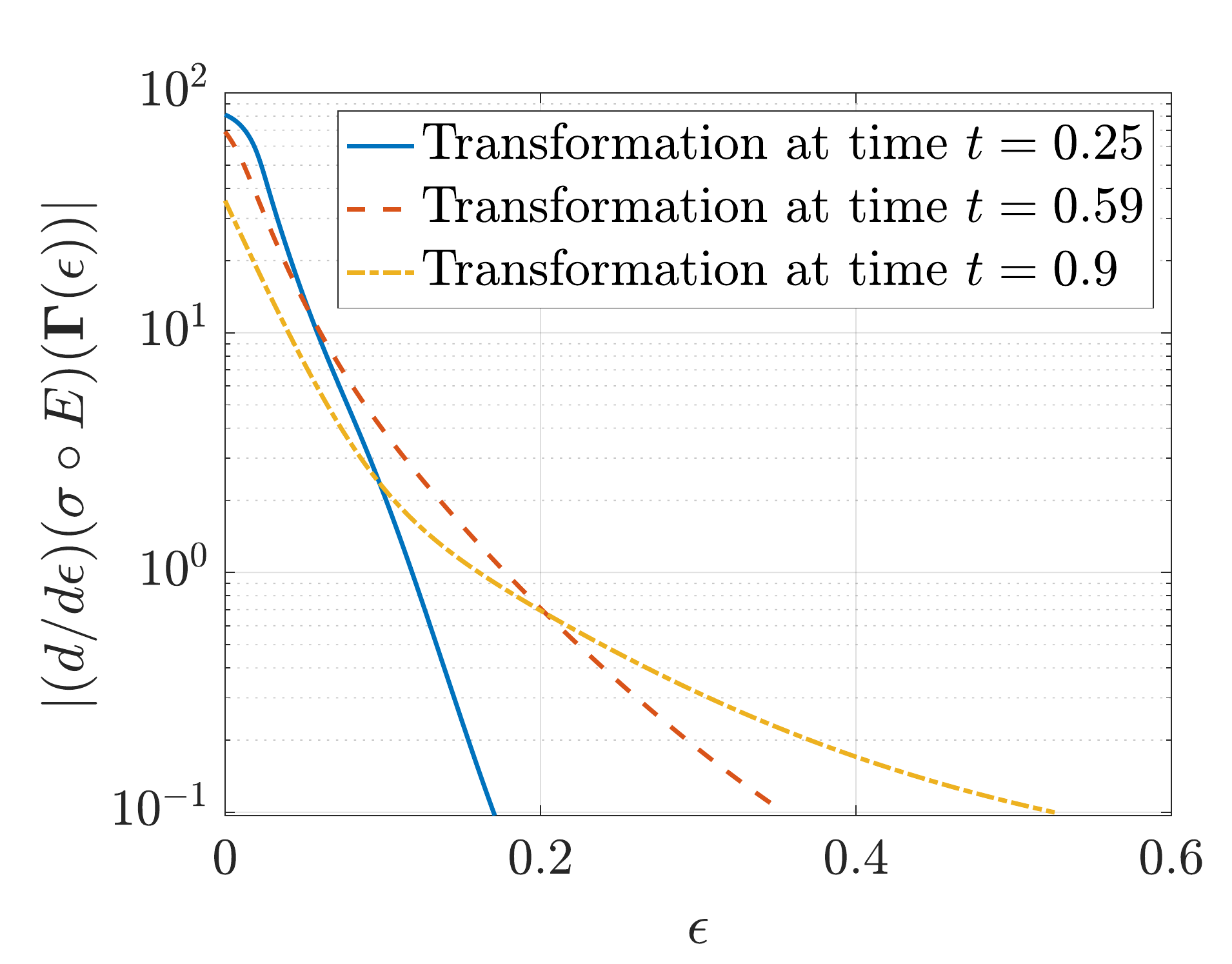}\\
\vspace{-.1cm}
\hspace{-5cm}
\includegraphics[scale=0.35]{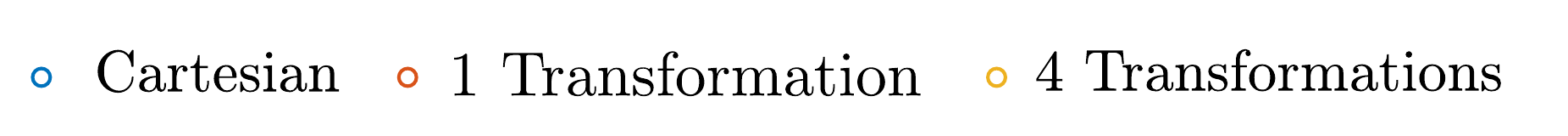} 
\caption{  
(a)-(b) Multilinear spectra of the FTT solutions 
to the 3D advection equation \eqref{Fokker_Planck} at time 
$t = 1$ in Cartesian coordinates and in low-rank coordinates. 
(c) Absolute value of the derivative of the cost function 
in \eqref{rank_min_problem} during gradient 
descent of the 3D advection equation \eqref{Fokker_Planck} 
solutions at times $t>0$.}
\label{fig:PDE_spectra_final_time}
\end{figure}

%
%
In Figure \ref{fig:PDE_rank_vs_time}(c) we 
plot the $L^{\infty}$ error between 
the transformed solutions 
and the FTT solution in 
Cartesian coordinates relative 
to the benchmark solution.  
We observe that the $L^{\infty}$ error of 
our coordinate-adaptive solution is very 
close to the {  $L^{\infty}$} error 
of the solution computed in Cartesian coordinates. This implies that 
the error incurred by transforming coordinates, 
integrating the PDE in the new 
coordinate system, and then transforming 
coordinates back is not significant.
{ 
\paragraph{Computational cost}
The CPU-time of integrating the three-dimensional 
advection equation \eqref{Fokker_Planck} 
from $t = 0$ to $t=1$ is 413 seconds when computed 
using FTT in Cartesian coordinates, 
173 seconds when computed using FTT-ridge in low-rank 
coordinates with one coordinate transformation at $t=0$, 
and 299 seconds when computed using the 
coordinate-adaptive FTT  Algorithm \ref{alg:coord_adaptive_integrator}. 
These timings do not include the coordinate 
transformations at time $t=0$ since they were 
not computed using Riemannian gradient descent. 
The timings do include the computation of 
the new coordinate systems at times 
$t>0$.}
\begin{figure}[!t]
\centerline{\footnotesize\hspace{.1cm} $t = 0$ \hspace{1.65cm}    \hspace{5.3cm} $t = 1$  }
	\centering
	\rotatebox{90}{\hspace{1.7cm}\footnotesize FTT (Cartesian coordinates)}
\includegraphics[scale=0.41]{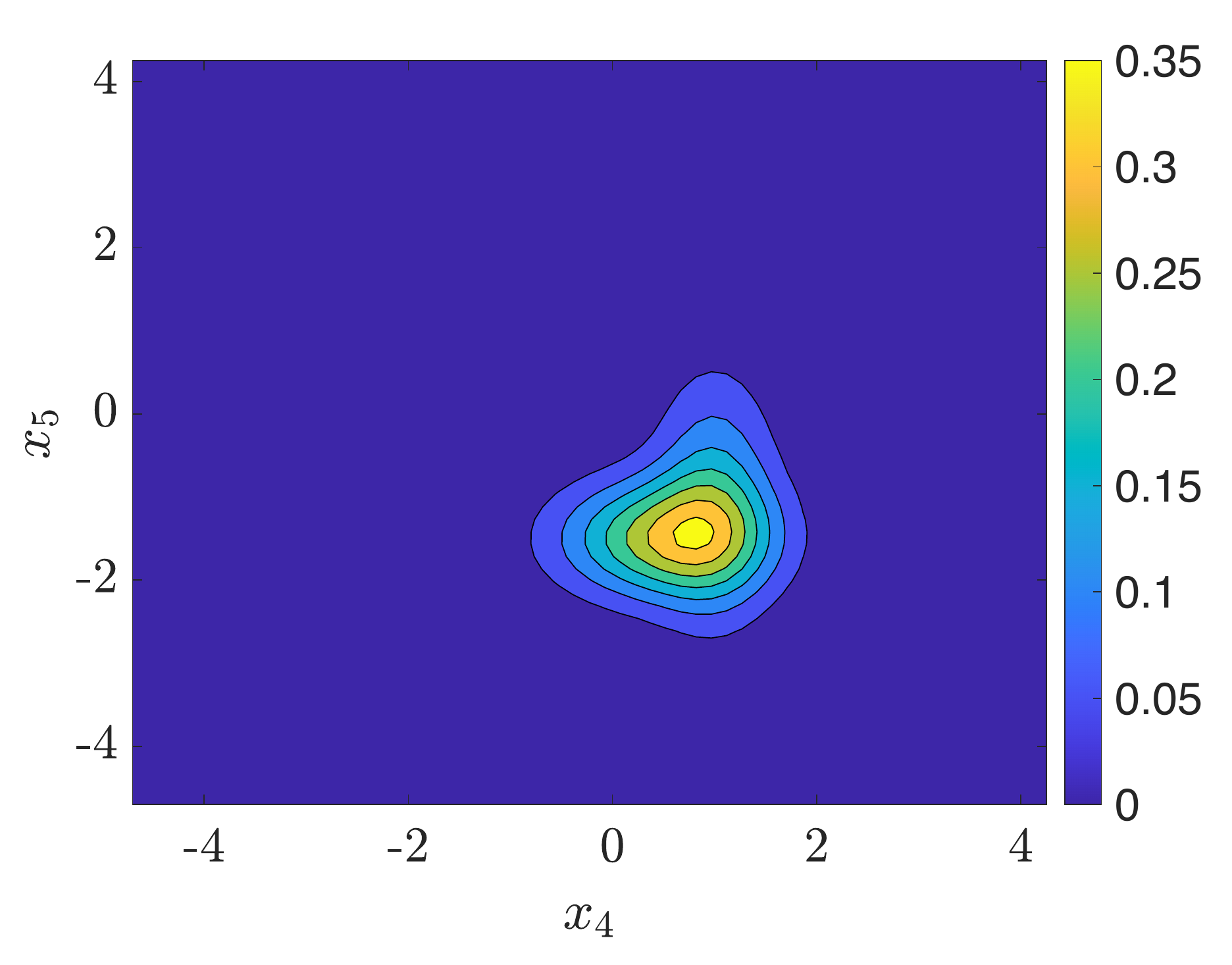} \hspace{-.2cm}
\includegraphics[scale=0.41]{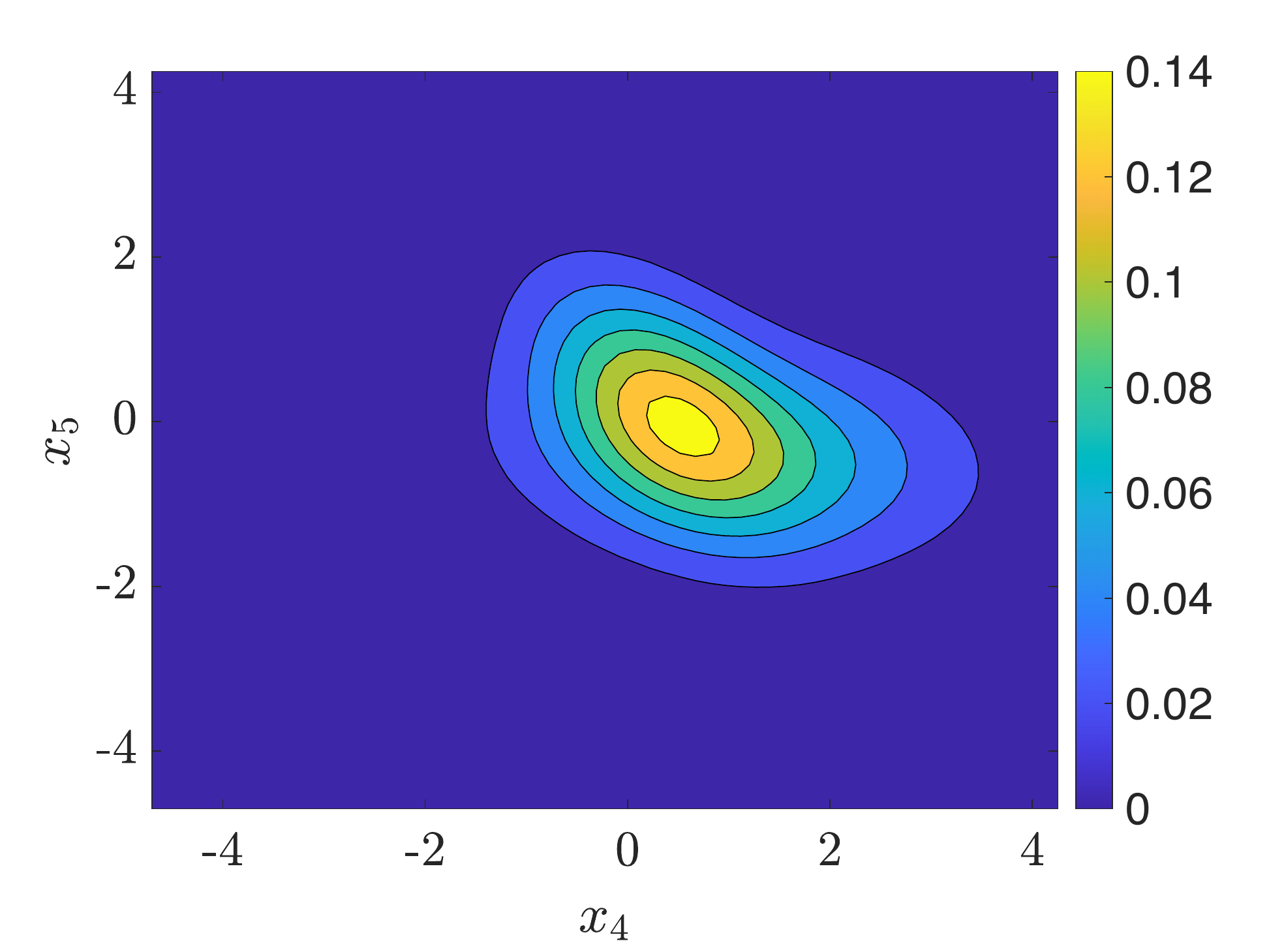} \\
\vspace{0.2cm}
	\centering
	\rotatebox{90}{\hspace{.9cm}\footnotesize FTT-ridge (transformed coordinates)}
\includegraphics[scale=0.41]{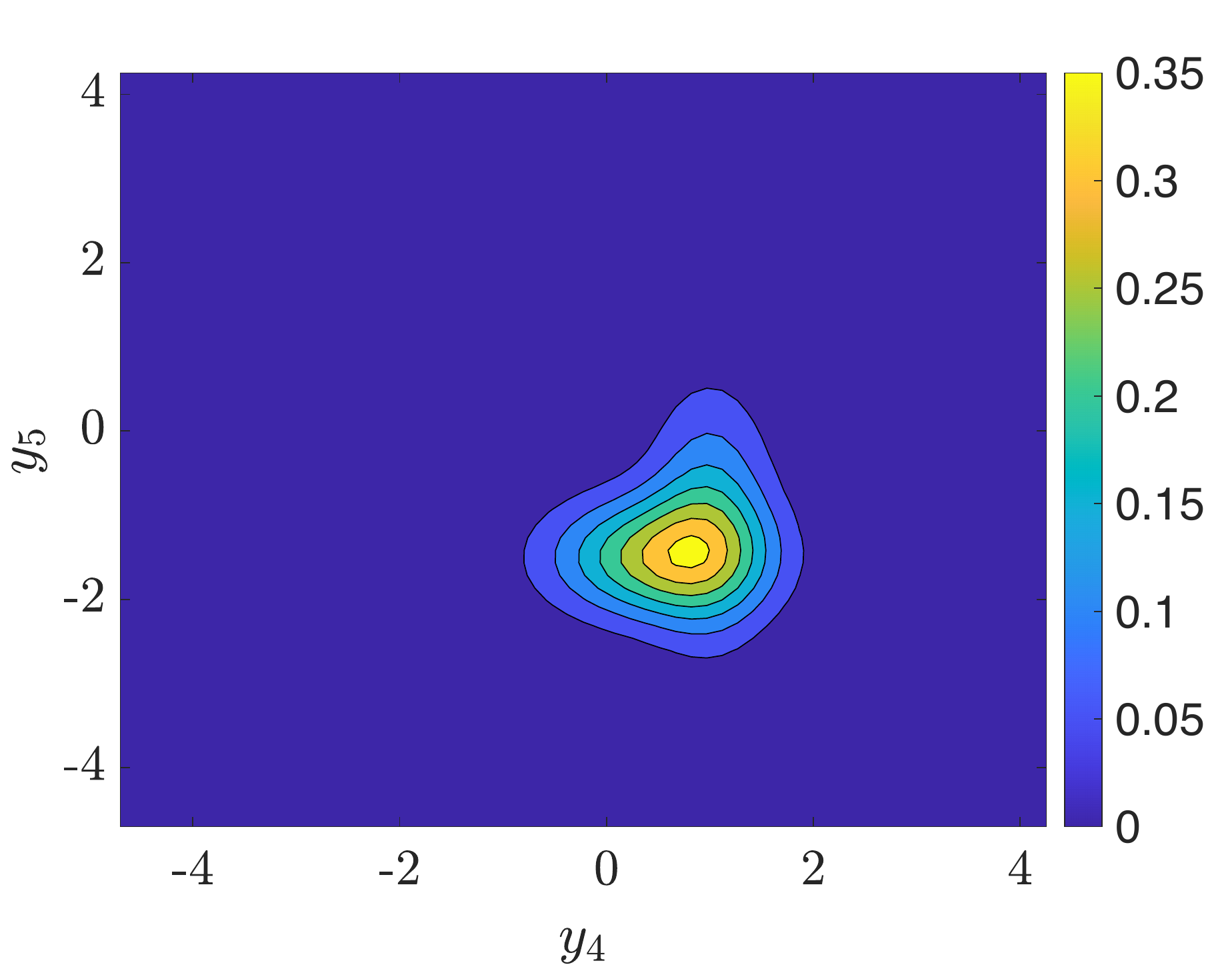} \hspace{-.0cm}
\includegraphics[scale=0.41]{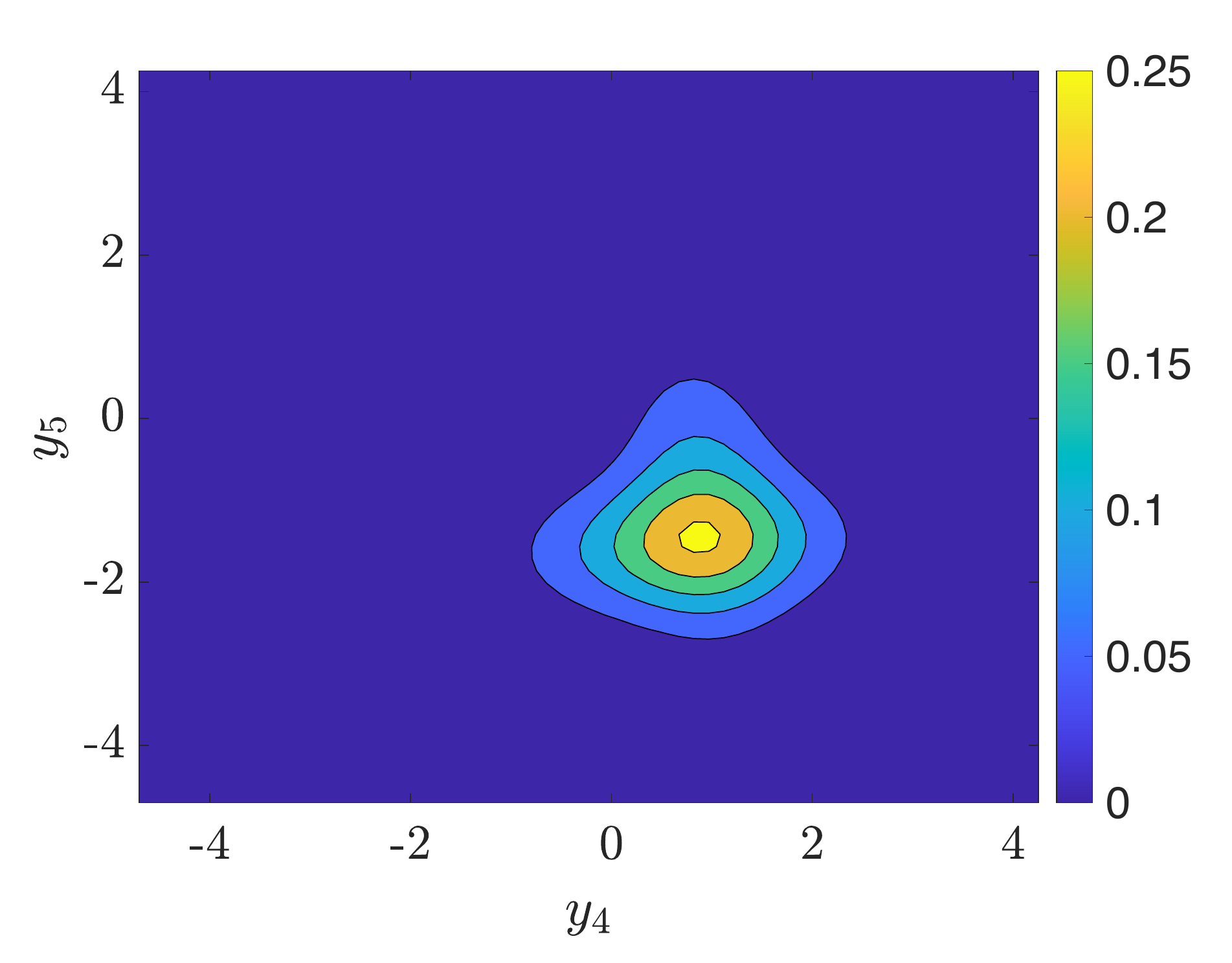} \\
\caption{  Marginal PDFs of the solution to the 
5D advection equation \eqref{Fokker_Planck} 
computed in Cartesian coordinates and low-rank 
adaptive coordinates at time $t=0$ and time 
$t = 1$. }
\label{fig:5D_PDE_marginals}
\end{figure}

{ 
\subsubsection{Five-dimensional simulation results}
Finally, we consider the advection equation  
\eqref{Fokker_Planck} in dimension five 
($d=5$) with coefficients 
\begin{equation}
\label{drift_5D}
\bm f(\bm x) = 
\begin{bmatrix}
-x_2 \\
x_3 \\
x_5 \\
-x_2 \\
-x_3
\end{bmatrix}.
\end{equation}
This allows us to test our coordinate-adaptive algorithm for 
a case in which we know the optimal ridge matrix. 
\begin{figure}[!t]
\centerline{\footnotesize\hspace{-0.1cm}   \hspace{5cm}    \hspace{5.5cm} }
	\centering
\includegraphics[scale=0.39]{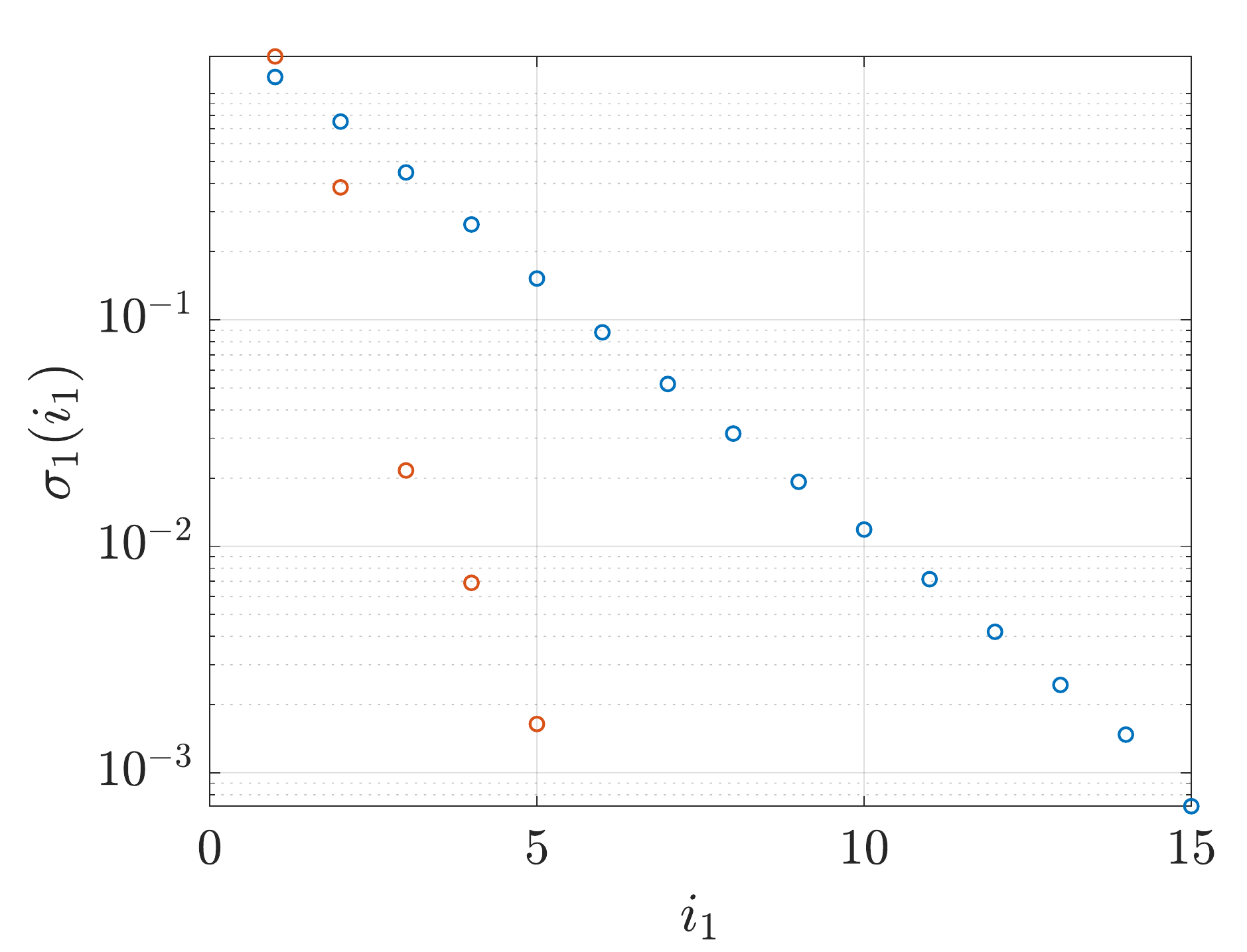} \hspace{.4cm}
\includegraphics[scale=0.39]{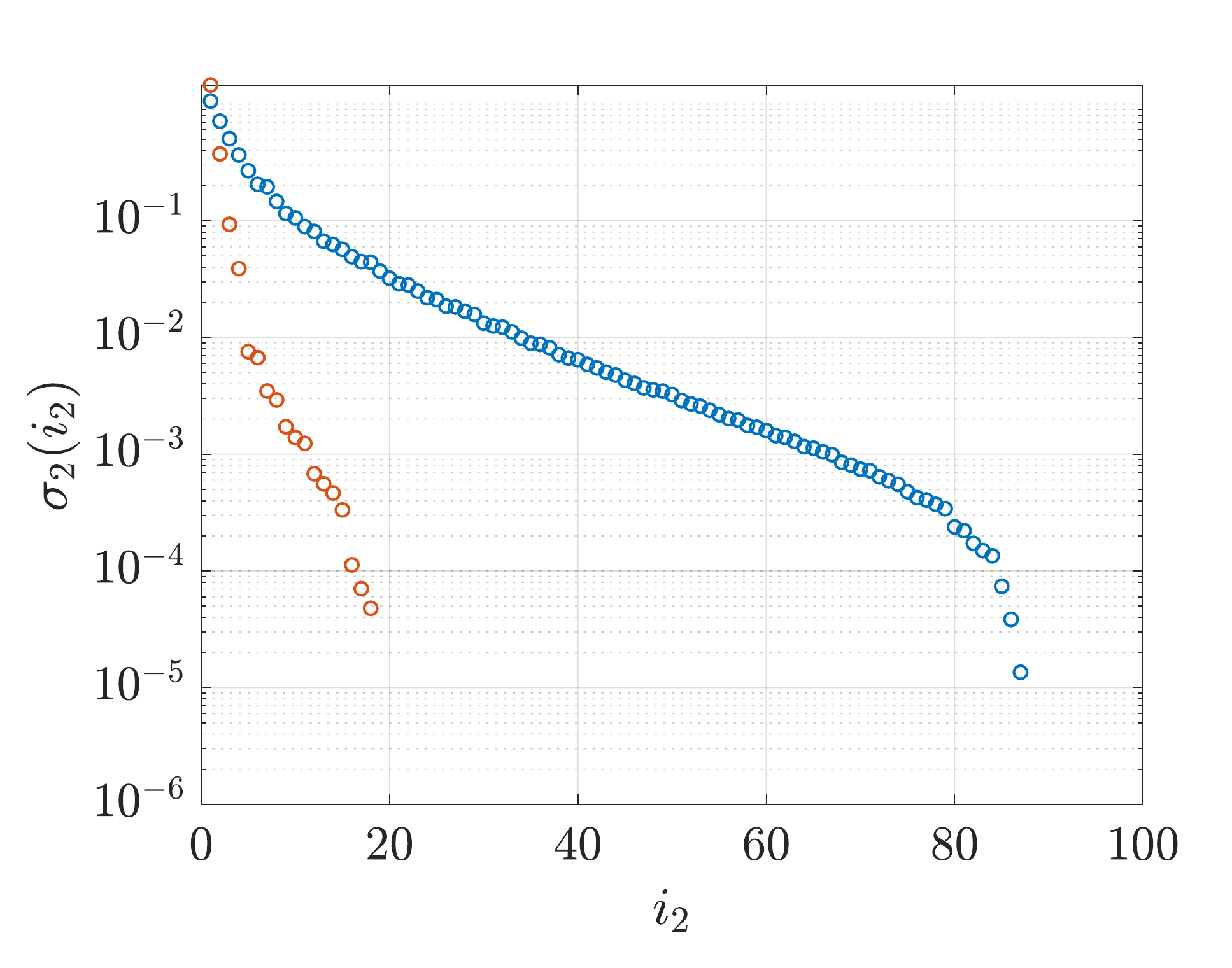} \\
\includegraphics[scale=0.39]{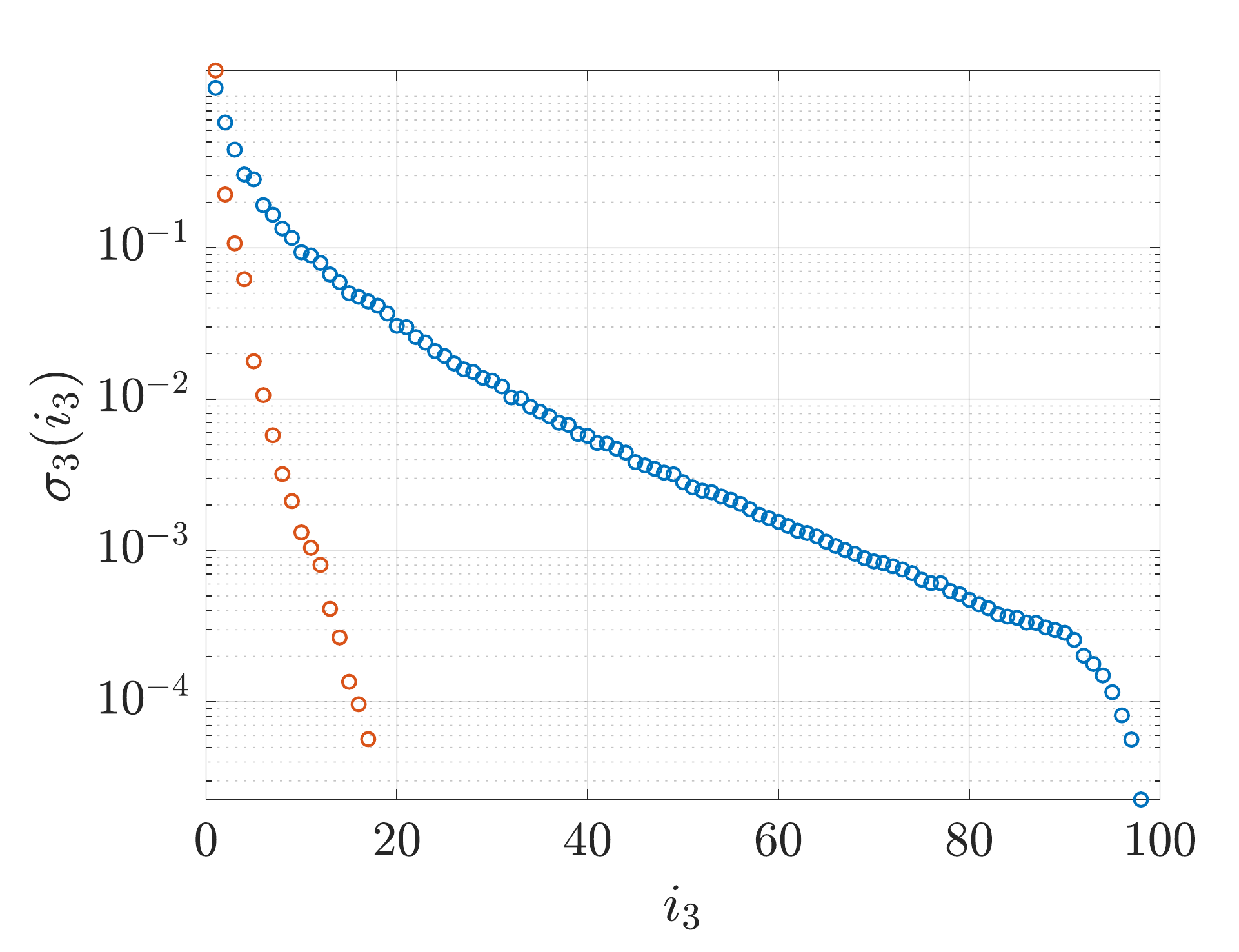} \hspace{.4cm}
\includegraphics[scale=0.39]{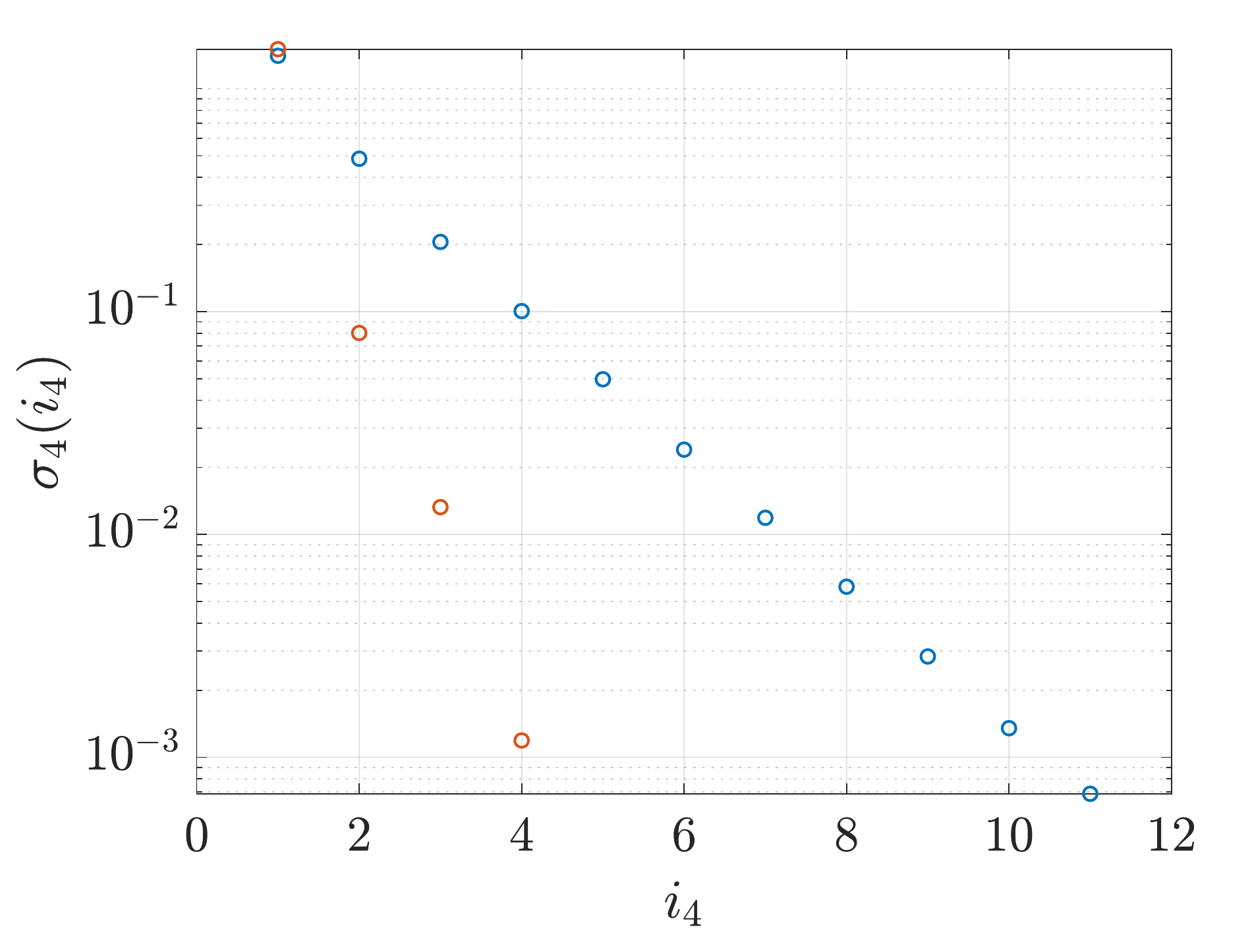} \\
\vspace{-.3cm}
\includegraphics[scale=0.45]{legend.pdf} 
\caption{  Multilinear spectra of the solution to the 
advection 5D PDE \eqref{Fokker_Planck} 
in Cartesian coordinates and transformed 
coordinates at time $t=1$.}
\label{fig:5D_PDE_spectra_final_time}
\end{figure}
In the initial condition \eqref{PDF_IC} we set the following parameters: 
$N_g = 2$, 
\begin{equation}
\label{rotation_5D}
\begin{aligned}
&\bm R^{(1)} = \bm R^{(2)}= \bm I_{5 \times 5}, \\ 
&\bm \beta = \begin{bmatrix}
 1/2 & 2 & 1/2 &  3  &  1/2 \\
 1 & 1/3 & 2 & 1 & 1/2 
 \end{bmatrix},
 \end{aligned}
\end{equation}
and 
\begin{equation}
\bm t = \begin{bmatrix}
 1 & 1 & 1 &  -1 & 1 \\
 0 & 0 & 3/2 & -1/2 & 1/2 
 \end{bmatrix},
\end{equation}
which results in an initial condition 
with FTT rank 
$\begin{bmatrix}
1 & 2 & 2 & 2 & 2 & 1
\end{bmatrix}$. 
Due to the choice of coefficients 
\eqref{drift_5D}, the 
analytical solution to the PDE 
\eqref{Fokker_Planck} 
can be written as a ridge 
function in terms of the PDE 
initial condition 
\begin{equation}
\label{5D_tensor_ridge_soln}
u(\bm x,t) = u_0\left(e^{t\bm B} \bm x \right), 
\end{equation}
where
\begin{equation}
\bm B = 
\begin{bmatrix}
0 & -1 & 0 & 0 & 0 \\
0 & 0 & 1 & 0 & 0 \\
0 & 0 & 0 & 0 & 1 \\
0 & -1 & 0 & 0 & 0 \\
0 & 0 & -1 & 0 & 0 
\end{bmatrix}.
\end{equation}
Thus \eqref{5D_tensor_ridge_soln} is a tensor 
ridge solution to the 5D advection equation 
\eqref{Fokker_Planck} with the same 
rank as the initial condition, i.e., 
in this case there exists a tensor ridge solution 
at each time with FTT rank equal to 
$\begin{bmatrix}
1 & 2 & 2 & 2 & 2 & 1
\end{bmatrix}$. 

%
%
%

We ran two simulations of the PDE \eqref{Fokker_Planck} 
up to $t = 1$.
The first simulation is computed with a step-truncation 
method in fixed Cartesian coordinates. 
The second simulation is computed with the 
coordinate-adaptive FTT-ridge tensor method 
that performs coordinate corrections at each time step 
(Algorithm \ref{alg:one_step_one_step}) 
with $\Delta \epsilon = 5 \times 10^{-4}$ and 
$M_{\rm iter} = 1$.  
In Figure \ref{fig:5D_PDE_marginals} we plot 
marginal PDFs of the FTT solution and the 
FTT-ridge solution at initial time $t=0$ 
and final time $t = 1$. 
We observe that the reduced rank FTT-ridge 
solution appears to be more symmetrical with 
respect to the coordinate axes 
than the corresponding function 
on Cartesian coordinates. 
In Figure \ref{fig:5D_PDE_rank_and_error}(a) 
we plot the $1$-norm of the 
FTT solution rank versus time 
and in Figure  \ref{fig:5D_PDE_rank_and_error}(b) 
we plot the FTT rank of the PDE velocity 
vector (right hand side of the PDE) versus 
time for both FTT solutions. 
We observe that the FTT solution rank in Cartesian 
coordinates grows quickly compared to the FTT-ridge 
solution in Cartesian coordinates. 
This is expected due to the existence 
of a low-rank FTT-ridge solution 
\eqref{5D_tensor_ridge_soln}. 
Note that the coordinate-adaptive algorithm 
\ref{alg:one_step_one_step} produces a FTT-ridge 
solution with ridge matrix that is different 
than $e^{t\bm B}$ in \eqref{5D_tensor_ridge_soln}.
The reason can be traced back to the cost function 
we are minimizing, i.e., the Schauder norm (see section 
\ref{sec:non-convex-relaxation}), and the 
fact that we do not fully determine the minimizer 
at each step, but rather perform only one 
$\epsilon$-step in the direction of the Riemannian gradient. 
Although the rank of the FTT-ridge solution 
computed with algorithm \ref{alg:one_step_one_step} 
is larger than the rank of the analytical 
solution \eqref{5D_tensor_ridge_soln}, 
the algorithm still controls the FTT 
solution rank during time integration. 
}
%
%
%
In Figure \ref{fig:5D_PDE_spectra_final_time} we plot 
the multilinear spectra of the two FTT solutions at 
time $t = 1$. 
{  We observe that the multilinear spectra 
of the FTT-ridge function
decay significantly faster than the 
spectra of the FTT solution
in Cartesian coordinates. }
%
%
%
%
%
%

{  For this problem it} is not straightforward 
to compute a benchmark solution 
on a full tensor product grid. If we were to use the 
same resolution as the FTT solutions, i.e., 
$200$ points in each dimension, then each time 
snapshot of the benchmark solution would be an array containing 
$200^5 \approx 3.2\times 10^{11}$ double precision floating point 
numbers. This requires $2.56$ terabytes of memory storage 
per time snapshot.
In lieu of comparing our FTT solutions with a benchmark 
solution, we compared the two FTT solutions with each other. 
{ 
To do so we mapped 
the FTT-ridge solution back to 
Cartesian coordinates every $250$ time steps 
by solving a PDE of the form \eqref{PDE_for_tensor} numerically.
We compared the $(x4,x5)$-marginal PDFs of the 
two solutions and found that the global $L^{\infty}$ norm 
of the difference of the two solution PDFs 
is bounded by $6 \times 10^{-4}$. 
}
%

%
%
\begin{figure}[!t]
\centerline{\footnotesize\hspace{5.5cm} (a)  \hspace{7cm}    (b) \hspace{5.05cm} }
	\centering
\includegraphics[scale=0.4]{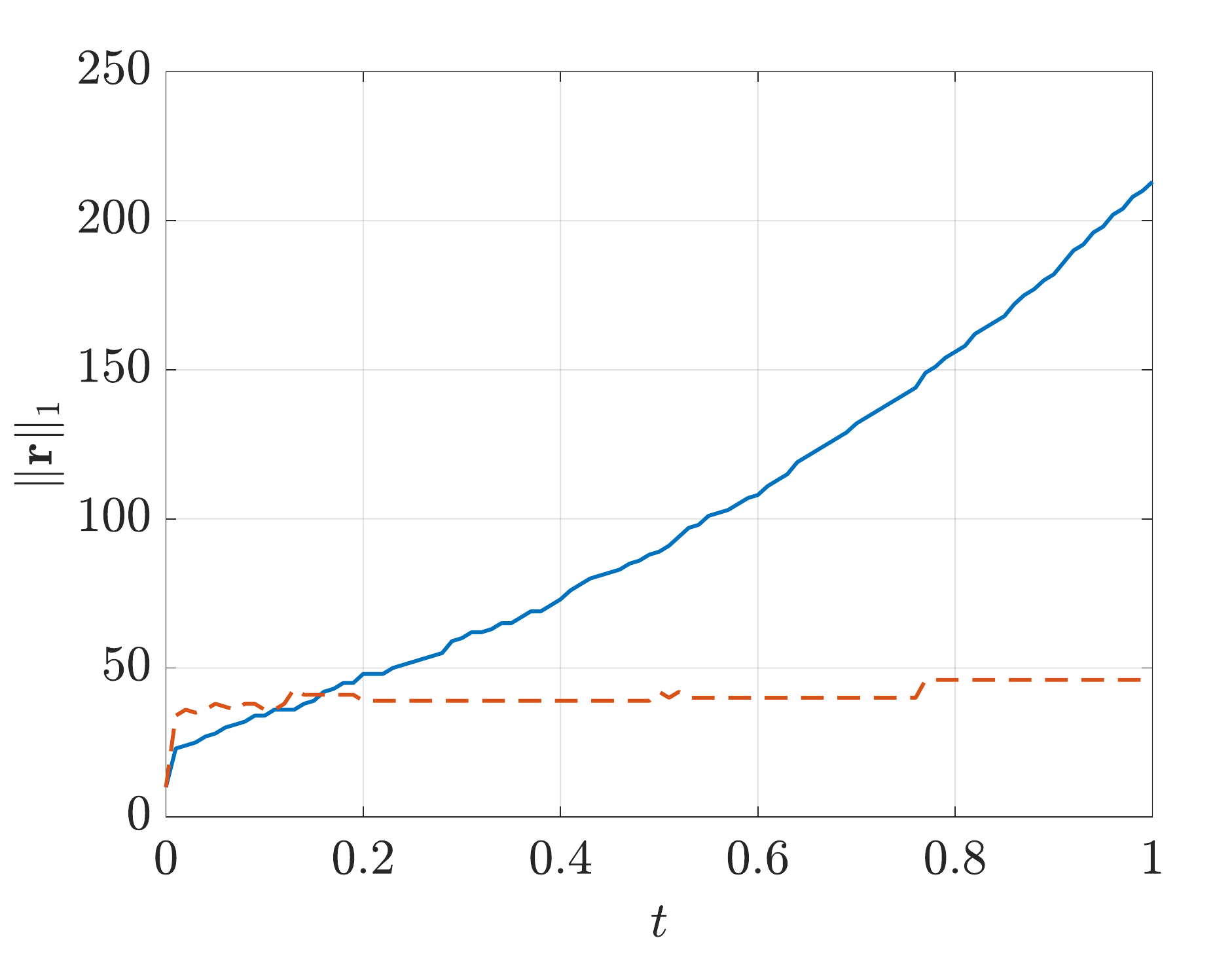}
\includegraphics[scale=0.4]{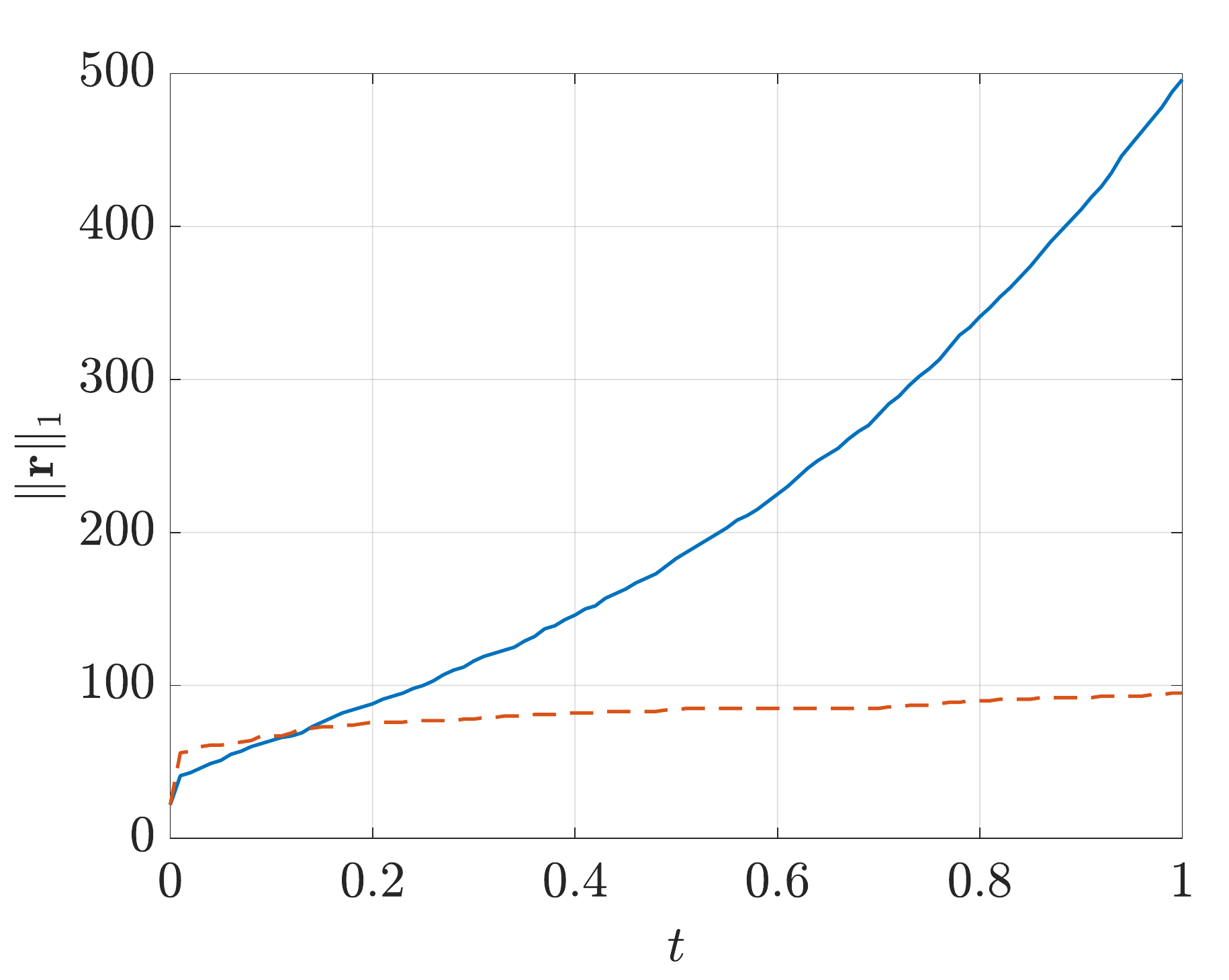}\\
\vspace{-.1cm}
\includegraphics[scale=0.45]{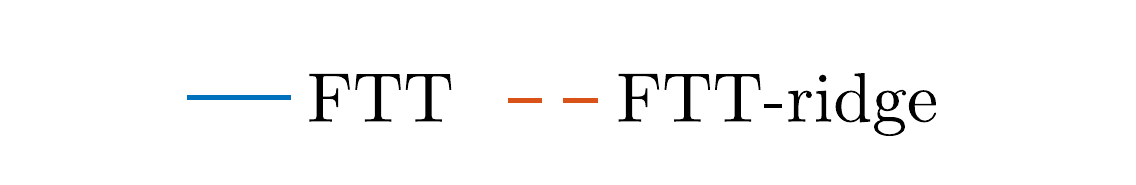} 
\caption{  (a) 1-norm of the solution rank vectors versus time. 
(b) 1-norm of the solution velocity (PDE right hand side) rank vectors versus time. }
\label{fig:5D_PDE_rank_and_error}
\end{figure}

{ 
\paragraph{Computational cost}
The CPU-time of integrating the five-dimensional 
advection equation \eqref{Fokker_Planck} 
from $t = 0$ to $t=1$ is 2046 seconds when computed 
using FTT in Cartesian coordinates and 
2035 seconds when computed using FTT-ridge in low-rank 
coordinates with coordinate corrections at each time step.}

\section*{Acknowledgements}
\noindent
This research was supported by the U.S. Air Force 
Office of Scientific Research grant FA9550-20-1-0174,
and by the U.S. Army Research Office 
grant W911NF1810309.

\appendix

\section{The Riemannian manifold of coordinate transformations}
\label{app:Riemann}
We endow the search space ${\rm SL}_d(\mathbb{R})$ in 
\eqref{relaxed_rank_min} with a Riemannian manifold 
structure. To this end, we first notice that ${\rm SL}_d(\mathbb{R})$ 
is a matrix Lie group over $\mathbb{R}$ and in particular 
is a smooth manifold. 
A point $\bm A \in {\rm SL}_d(\mathbb{R})$ 
corresponds to a linear coordinate 
transformation of $\mathbb{R}^d$ 
with determinant equal to $1$.
A smooth path $\bm \Theta(\epsilon)$ on the manifold 
${\rm SL}_d(\mathbb{R})$ paramaterized by $\epsilon \in (-\delta, \delta)$ 
is a collection of smoothly varying linear coordinate 
transformations with determinant equal to $1$ for 
all $\epsilon \in (-\delta,\delta)$. 
Denote by $\mathcal{C}^1\left( (-\delta,\delta) , 
{\rm SL}_d(\mathbb{R}) \right)$ the collection of 
all continuously differentiable paths $\bm \Theta(\epsilon)$ 
on the manifold ${\rm SL}_d(\mathbb{R})$ parameterized by 
$\epsilon \in (-\delta, \delta)$. 
The tangent space of ${\rm SL}_d(\mathbb{R})$ at the 
point $\bm A \in {\rm GL}_d(\mathbb{R})$ is defined to 
be the collection of equivalence classes of velocities 
associated to all possible curves on ${\rm SL}_d(\mathbb{R})$ 
passing through the point $\bm A$
\begin{equation}
\label{tangent_space}
T_{\bm A} {\rm SL}_d(\mathbb{R}) = \left\{  
\frac{d \bm \Theta(\epsilon)}{d \epsilon}\biggr\vert_{\epsilon=0}: 
\quad \bm \Theta \in \mathcal{C}^1\left( (-\delta,\delta) , 
{\rm SL}_d(\mathbb{R}) \right), \quad  
\bm \Theta(0) = \bm A \right\}.
\end{equation}
It is well-known (e.g., \cite{gradient_flow}) that 
the tangent space of ${\rm SL}_d(\mathbb{R})$ 
at the point $\bm A$ is given by 
\begin{equation}
\label{lie_alg_coset}
T_{\bm A} {\rm SL}_d(\mathbb{R}) = \mathfrak{sl}_d(\mathbb{R}) \bm A = \{ \bm N \bm A : \bm N \in \mathfrak{sl}_d(\mathbb{R}) \},
\end{equation}
where $\mathfrak{sl}_d(\mathbb{R})$ denotes the collection 
of all $d \times d$ real matrices 
with vanishing trace. 
We can easily verify that if 
$\bm \Theta(\epsilon) $ is a smooth collection 
of matrices parameterized by $\epsilon \in (-\delta,\delta)$ with 
$\bm \Theta(0) \in {\rm SL}_d(\mathbb{R}) = \mathfrak{sl}_d(\mathbb{R})$ and 
$d\bm \Theta(\epsilon)/d \epsilon \in T_{\bm \Theta(\epsilon)} {\rm SL}_d(\mathbb{R})$ 
for all $\epsilon$, then $\det(\bm \Theta(\epsilon)) = 1$ for all $\epsilon$, i.e., 
$\bm \Theta(\epsilon) \in {\rm SL}_d(\mathbb{R})$ for all $\epsilon$. 
The proof of this result is a direct application of 
Jacobi's formula \cite{Magnus2019} which states 
\begin{equation}
\label{Jacobi_formula}
\begin{aligned}
\frac{d }{d \epsilon} \det\left(\bm \Theta(\epsilon)\right) &= \det\left(\bm \Theta(\epsilon)\right) {\rm trace}\left(\frac{d \bm \Theta(\epsilon)}{d \epsilon}  \bm \Theta^{-1}(\epsilon) \right) \\
&= \det\left(\bm \Theta(\epsilon)\right) {\rm trace}\left( \bm N \bm \Theta(\epsilon)   \bm \Theta^{-1}(\epsilon) \right) \\
&=  \det\left(\bm \Theta(\epsilon)\right) {\rm trace}\left( \bm N \right) \\
&= 0 , 
\end{aligned}
\end{equation}
since $\bm N \in \mathfrak{sl}_d(\mathbb{R})$. 
Thus the determinant of $\bm \Theta(\epsilon)$ is constant 
and since $\det( \bm \Theta(0) ) = 1$ it follows that $\det(\bm \Theta(\epsilon)) = 1$ 
for all $\epsilon \in (-\delta,\delta)$. 
In the language of abstract differential equations, 
$\mathfrak{sl}_d(\mathbb{R})$ is referred 
to as the Lie algebra associated with the 
Lie group ${\rm SL}_d(\mathbb{R})$. 
In Figure \ref{fig:manifold_sketches}(a)
we provide an illustration of 
a path $\bm \Theta(\epsilon)$ on 
the manifold ${\rm SL}_d(\mathbb{R})$ 
passing through the point $\bm A$ and 
the tangent space 
$T_{\bm A} {\rm SL}_d(\mathbb{R})$ of 
${\rm SL}_d(\mathbb{R})$ at $\bm A$. 
Also depicted in Figure \ref{fig:manifold_sketches} 
is a smooth function $f$ from ${\rm SL}_d(\mathbb{R})$ to another 
smooth manifold $\mathcal{M}$. The image of the 
path $\bm \Theta(\epsilon)$ on 
$\rm{SL}_d(\mathbb{R})$ under $f$ is a 
path $f(\bm \Theta(\epsilon))$ on 
$\mathcal{M}$.
Under this mapping of curves,
we can associate the tangent vector 
$d\bm \Theta(\epsilon)/d \epsilon \vert_{\epsilon=0}$ 
in $T_{\bm A} {\rm SL}_d(\mathbb{R})$ 
with a tangent vector 
$d f(\bm \Theta(\epsilon))/d\epsilon\vert_{\epsilon=0}$ in 
$T_{f(\bm A)} \mathcal{M}$. 
This association gives rise to the notion of 
the directional derivative of a function 
$f : {\rm SL}_d(\mathbb{R}) \to \mathcal{M}$
which we now define.

\vs
\noindent
{\bf Definition A1. }{\em 
Let $f$ be a smooth function from the Riemannian manifold
${\rm SL}_d(\mathbb{R})$ to a smooth manifold $\mathcal{M}$. 
The directional derivative of $f$ at the point $\bm A \in {\rm SL}_d(\mathbb{R})$ 
in the direction $\bm V \in T_{\bm A} {\rm SL}_d(\mathbb{R})$ is defined as
\begin{equation}
\label{directional_deriv}
(d_{\bm A} f ) \bm V = \frac{\partial  f(\bm \Theta(\epsilon))}{\partial \epsilon} \biggr\vert_{\epsilon=0},
\end{equation}
where $\bm \Theta(\epsilon)$ is a smooth curve on ${\rm SL}_d(\mathbb{R})$ passing through the 
point $\bm A$ at $\epsilon=0$ with velocity $\bm V$.}

\vs
\noindent
It is a standard exercise of differential 
geometry to verify that the directional derivative
\eqref{directional_deriv} is independent of the choice 
of curve $\bm \Theta(\epsilon)$. 
The map $d_{\bm A} f$ appearing in 
\eqref{directional_deriv} is a linear 
map from $T_{\bm A} {\rm SL}_d(\mathbb{R})$ to $T_{f(\bm A)} \mathcal{M}$ known as 
the differential of $f$. 
\begin{figure}[!t]
\centerline{\footnotesize\hspace{-0.1cm}   \hspace{5cm}    \hspace{5.5cm} }
\centerline{
\includegraphics[scale=0.4]{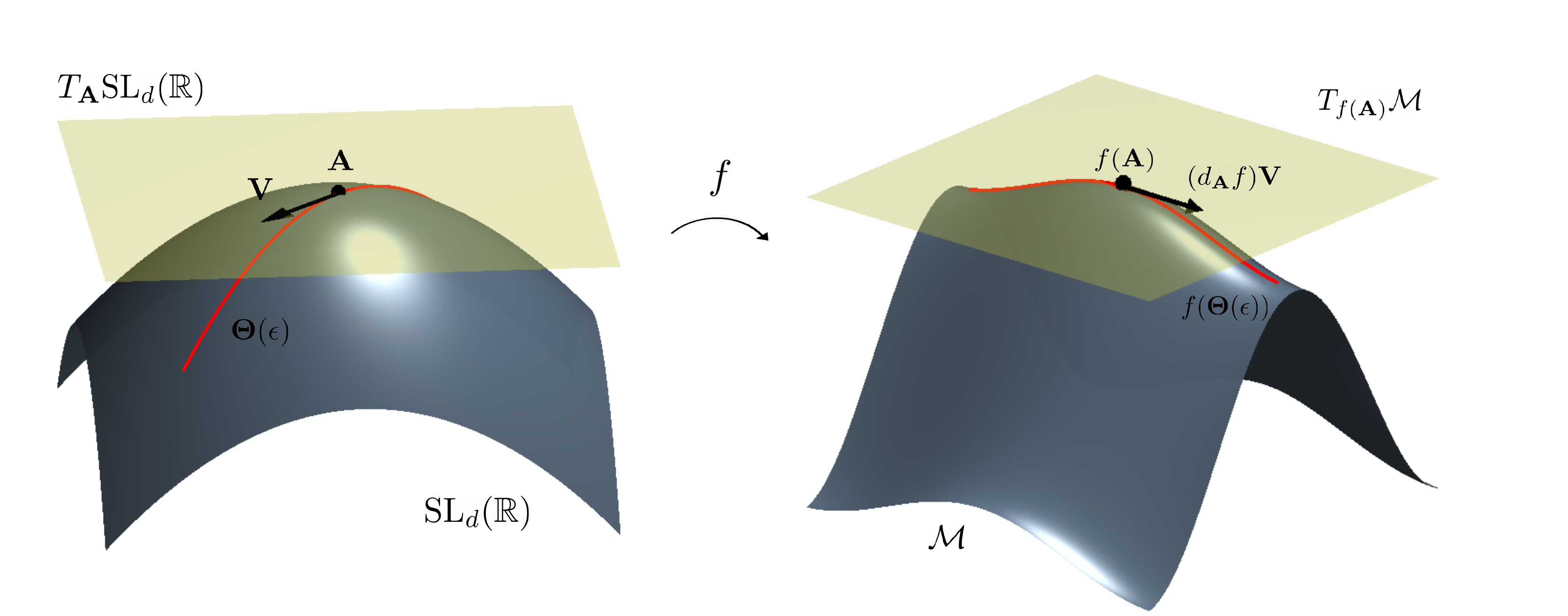}
}
\caption{An illustration of the 
directional derivative of the function 
$f : {\rm SL}_d(\mathbb{R}) \to 
\mathcal{M}$. The curve 
$\bm \Theta(\epsilon)$ on 
${\rm SL}_d(\mathbb{R})$ is 
mapped to the curve 
$f(\bm \Theta(\epsilon))$ on 
$\mathcal{M}$ and the 
directional derivative of 
$f$ at $\bm A \in {\rm SL}_d(\mathbb{R})$ 
in the direction 
$\bm V \in T_{\bm A} 
{\rm SL}_d(\mathbb{R})$ is the 
velocity of the curve $f(\bm \Theta(\epsilon))$ 
at $f(\bm A)$.}
\label{fig:manifold_sketches}
\end{figure}
In Figure \ref{fig:manifold_sketches} 
we provide an illustration of the mapping 
$f$ and its differential. 
The differential and 
directional derivative allow 
us to understand the change in 
$f$ when moving from the point 
$\bm A \in {\rm SL}_d(\mathbb{R})$ 
in the direction of 
the tangent vector $\bm V  \in 
T_{\bm A} {\rm SL}_d(\mathbb{R})$. 
Next, we compute the Riemannian gradient of $f$, 
i.e., a specific tangent vector on ${\rm SL}_d(\mathbb{R})$ 
that points in a direction which makes the function $f$ 
vary the most. 

For functions defined on Euclidean 
space, the connection between 
directional derivative and gradient 
is understood by using the standard 
inner product defined for Euclidean spaces.  
A generalization of the Euclidean inner product for 
an abstract manifold such as ${\rm SL}_d(\mathbb{R})$ is 
the Riemannian metric $( \cdot, \cdot)_{\bm A}$, which is 
a collection of smoothly varying inner products on each tangent space 
$T_{\bm A} {\rm SL}_d(\mathbb{R})$. 
In particular, we define
\begin{equation}
\label{riemannian_metric}
(\bm V, \bm W)_{\bm A} = {\rm trace} \left[ \left( \bm V \bm A^{-1} \right) \left( \bm W \bm A^{-1} \right)^{\top} \right], \qquad \forall \bm A \in {\rm SL}_d(\mathbb{R}), \quad \forall \bm V, \bm W \in T_{\bm A} {\rm SL}_d(\mathbb{R})
\end{equation}
on ${\rm SL}_d(\mathbb{R})$. 
With this Riemannian metric, we can 
define the Riemannian gradient. 

\vs
\noindent
{\bf Definition A2. }{\em 
Let $f$ be a smooth function from 
${\rm SL}_d(\mathbb{R})$ to a smooth manifold $\mathcal{M}$. 
The Riemannian gradient of $f$ at the point 
$\bm A \in {\rm SL}_d(\mathbb{R})$ is the unique vector 
field ${\rm grad } f (\bm A)$ satisfying 
\begin{equation}
(d_{\bm A} f ) \bm V = ( {\rm grad } f (\bm A), \bm V)_{\bm A}, \qquad \forall \bm V \in T_{\bm A} {\rm SL}_d(\mathbb{R}).
\end{equation}}

\noindent
Analogous to the Euclidean case, 
the Riemannian gradient points in 
the direction which $f$ increases the most, 
and, the negative gradient points the 
the direction which $f$ decreases most. 
With this Riemannian geometric structure, 
we proceed by constructing 
a path on ${\rm SL}_d(\mathbb{R})$, known 
as a descent path, which converges to 
a local minimum of the cost function 
$S \circ E$ in \eqref{relaxed_rank_min}.

\section{Theorems and Proofs}
\noindent
    \setcounter{lemma}{0}
    \renewcommand{\thelemma}{\Alph{section}\arabic{lemma}}
        \setcounter{proposition}{0}
    \renewcommand{\theproposition}{\Alph{section}\arabic{proposition}}
\label{sec:appdx_proof}

\noindent First, we show that the tensor rank is invariant 
under translation of functions with compact support. This allows us 
to disregard translations when looking for rank reducing coordinate
flows.

\begin{proposition}
Let $u_{\T} \in L^2_{\mu}(\Omega)$ $(\Omega \subseteq \mathbb{R}^d)$ 
be a rank-$\bm r$ FTT 
with ${\rm supp}(u_{\T}) \subseteq \Omega$ and 
$C_{\bm t} : \mathbb{R}^d \to \mathbb{R}^d$
a coordinate translation, i.e., 
\begin{equation}
C_{\bm t}(\bm x) = \bm x + \bm t,
\end{equation}
where $\bm t \in \mathbb{R}^d$, such that 
${\rm supp}(u_{\T}) \subseteq C_{\bm t}(\Omega)$. 
Then $u_{\T}(C_{\bm t}(\bm x))$ is also a rank-$\bm r$ 
FTT tensor. 
\end{proposition}
\begin{proof}
Let 
\begin{equation}
u_{\T}(\bm x) = \bm Q_{\leq i}(\bm x_{\leq i} ) \bm \Sigma_i \bm Q_{>i}(\bm x_{>i})
\end{equation}
be an orthogonalized expansion of the the rank-$\bm r$ 
FTT $u_{\T}$ as in \eqref{orthog_at_i}. 
Then by a simple change of variables in the integrals 
it is easy to verify that the translated cores 
$\bm Q_{\leq i}^{\top}(\bm x_{\leq i} + 
\bm t_{\leq i}) ,\bm Q_{>i}(\bm x_{> i} + \bm t_{> i})$ 
also satisfy the orthogonality conditions 
\begin{equation}
\begin{aligned}
\left\langle \bm Q_{\leq i}^{\top}(\bm x_{\leq i} + \bm t_{\leq i}) \bm Q_{\leq i}(\bm x_{\leq i} + \bm t_{\leq i}) \right\rangle_{\leq i } = \bm I_{\bm r_i \times r_i}, \\
\left\langle \bm Q_{> i}(\bm x_{>i} + \bm t_{>i}) \bm Q_{> i}^{\top}(\bm x_{>i} + \bm t_{>i}) \right\rangle_{>i} = \bm I_{\bm r_i \times r_i},
\end{aligned}
\end{equation}
and thus 
\begin{equation}
u_{\T}(C_{\bm t}( \bm x)) = \bm Q_{\leq i}(\bm x_{\leq i} + \bm t_{\leq i}) 
\bm \Sigma_i \bm Q_{> i}(\bm x_{> i} + \bm t_{> i}) 
\end{equation}
is an orthogonalized FTT tensor. Hence, $u_{\T} \circ C_{\bm t}$ has 
the same multilinear rank $\bm r$ as $u_{\T}$. 
\end{proof}

\vs
\noindent
Next, we provide a proof of Proposition \ref{prop:riemannian_gradient}. 
To do so we first provide the differentials of the maps 
$E$, $S$ and $C=S \circ E$ in the 
following lemmas. 
\begin{lemma}
\label{lemma:diff_of_E}
The differential of the evaluation map $E$ 
corresponding to ${  u_{\T}}$ (see eqn. \eqref{eval_map}) 
at the point $\bm A \in {\rm GL}_d(\mathbb{R})$ 
in the direction $\bm V \in T_{\bm A}  {\rm GL}_d(\mathbb{R})$ is 
\begin{equation}
(d_{\bm A} E ) \bm V = \nabla v(\bm x)  \cdot  ( \bm V \bm x ) .
\end{equation}
\end{lemma}
\begin{proof}
This result is proven directly from the definitions. 
Let $\bm \Theta(\epsilon)$ be a smooth curve on $ {\rm GL}_d(\mathbb{R})$ 
passing through $\bm A$ with velocity $\bm V$ at $\epsilon = 0$. 
Then 
\begin{equation}
\begin{aligned}
(d_{\bm A} E) \bm V &= \frac{\partial }{\partial \epsilon} E(\bm \Theta(\epsilon))\biggr\vert_{\epsilon=0} \\
&= \frac{\partial }{\partial \epsilon} {  u_{\T}}(\bm \Theta(\epsilon) \bm x )\biggr\vert_{\epsilon=0} \\
&= \nabla {  u_{\T}}(\bm A \bm x) \cdot \left( \bm V \bm x\right) \\
&=\nabla v(\bm x) \cdot \left( \bm V \bm x\right) . 
\end{aligned}
\end{equation}
\end{proof}

\begin{lemma}
\label{lemma:diff_of_sigma}
The differential of $S$ {  (see \eqref{sum_of_multi_sings})} at the point $v \in L^2_{\mu}(\Omega)$ 
in the direction $w \in T_{  v} L^2_{\mu}(\Omega)$ is 
\begin{equation}
(d_{v} S) w = \sum_{i=1}^{d-1} \int_{\Omega} \bm Q_{\leq i} \bm Q_{>i}
 w(\bm x) d\mu(\bm x)
\end{equation}
where $\bm Q_{\leq i}, \bm Q_{>i}$ are FTT cores of $v$ as in eqn. \eqref{orthog_at_i} 
satisfying the orthogonality conditions \eqref{orthogonal_cores}. 
\end{lemma}
\begin{proof}
Let $\gamma(\epsilon)$ be a smooth curve on $L^2_{\mu}(\Omega)$ passing through $v$ 
at $\epsilon = 0$ with velocity $w$. 
At each $\epsilon$ the function $\gamma(\epsilon)$ admits 
orthogonalizations of the form 
\begin{equation}
\label{gamma_orthog}
\gamma(\epsilon) = \bm Q_{\leq i}(\epsilon) \bm \Sigma_i(\epsilon) \bm Q_{>i}(\epsilon) 
\end{equation}
for each $i = 1,2,\ldots,d-1$. 
Differentiating \eqref{gamma_orthog} with respect 
to $\epsilon$ we obtain 
\begin{equation}
\frac{\partial \gamma}{\partial \epsilon} = \frac{\partial \bm Q_{\leq i}}{\partial \epsilon} \bm \Sigma_i \bm Q_{>i}
+\bm Q_{\leq i} \frac{\partial \bm \Sigma_i}{\partial \epsilon} \bm Q_{>i}  +
\bm Q_{\leq i} \bm \Sigma_i \frac{\partial \bm Q_{>i}}{\partial \epsilon} .
\end{equation}
Multiplying on the left by $\bm Q_{\leq i}^{\top}$ and on the right by $\bm Q_{>i} ^{\top}$, applying the operator $\langle \cdot \rangle_{\leq i, >i}$, 
and evaluating at $\epsilon = 0$ we obtain 
\begin{equation}
\label{evolution_of_Sigma}
\frac{\partial \bm \Sigma_i}{\partial \epsilon} = \left\langle \bm Q_{\leq i}^{\top}  {  w}(\bm x)  \bm Q_{>i}^{\top} \right\rangle_{\leq i, >i}
- \left\langle \bm Q_{\leq i}^{\top} \frac{\partial \bm Q_{\leq i}}{\partial \epsilon} \right\rangle_{\leq i} \bm \Sigma_i - \bm \Sigma_i \left\langle \frac{\partial \bm Q_{>i}}{\partial \epsilon} \bm Q_{>i}^{\top} \right\rangle_{>i},
\end{equation}
where we used 
orthogonality of the 
FTT cores $\bm Q_{\leq i}$ and $\bm Q_{>i}$.
Differentiating the orthogonality constraints 
\eqref{orthogonal_cores} with respect 
to $\epsilon$ we obtain 
\begin{equation}
\label{skew-symm_FTT_cores}
\left\langle \frac{\partial \bm Q_{\leq i}^{\top}}{\partial \epsilon} \bm Q_{\leq i} \right\rangle_{\leq i} = - \left\langle \bm Q_{\leq i}^{\top} \frac{\partial \bm Q_{\leq i}}{\partial \epsilon} \right\rangle_{\leq i}, \qquad 
\left\langle \frac{\partial \bm Q_{>i}}{\partial \epsilon} \bm Q_{> i}^{\top}  \right\rangle_{>i} = - \left\langle \bm Q_{>i} \frac{\partial \bm Q_{> i}^{\top}}{\partial \epsilon} \right\rangle_{>i} ,  \quad \forall \epsilon,  
\end{equation}
which implies that the second two terms on the right hand 
side of \eqref{evolution_of_Sigma} side are 
skew-symmetric and thus have zeros on the diagonal. 
Hence the diagonal entries of $\displaystyle\frac{\partial \bm \Sigma_i}{\partial \epsilon}$ are 
the diagonal entries of the matrix 
$\left\langle \bm Q_{\leq i}^{\top}  {  w}(\bm x)  \bm Q_{>i}^{\top} \right\rangle_{\leq i, >i}$ or written element-wise 
\begin{equation}
\label{evolution_of_one_eig_val}
\frac{\partial S_{i}(\alpha_i)}{\partial \epsilon} = \int q_{\leq i}(\alpha_i)  {  w}(\bm x) 
q_{> i}(\alpha_i) d\mu(\bm x).
\end{equation}
Finally summing \eqref{evolution_of_one_eig_val} over $i = 1,2,\ldots, d-1$ and 
$\alpha_i = 1,2,\ldots, r_i$ and using matrix product notation for the latter summation we 
obtain 
\begin{equation}
\sum_{i=1}^{d-1} \sum_{\alpha_i=1}^{r_i} \frac{\partial S_i(\alpha_i)}{\partial \epsilon} = \int_{\Omega} \bm Q_{\leq i} \bm Q_{>i} {  w}(\bm x) d \mu(\bm x),
\end{equation}
proving the result.
\end{proof}

\vs
\noindent
Combining the results of Lemma \ref{lemma:diff_of_E} and Lemma \ref{lemma:diff_of_sigma} 
with a simple application of the chain rule for differentials we prove the following Lemma. 
\begin{lemma}
\label{lemma:diff_of_composition}
The differential of the function $C = S \circ E$ at the 
point $\bm A \in {\rm GL}_d(\mathbb{R})$ in the direction $\bm V \in T_{\bm A} {\rm GL}_d(\mathbb{R})$ is 
\begin{equation}
d_{\bm A} ( S \circ E ) \bm V = \sum_{i=1}^{d-1} \int_{\Omega} \bm Q_{\leq i} \bm Q_{>i}  \nabla v(\bm x) \cdot \left( \bm V  \bm x \right)  d \mu (\bm x), 
\end{equation}
where $\bm Q_{\leq i}, \bm Q_{>i}$ are 
orthogonal FTT cores of $v(\bm x)$.
\end{lemma}
Next we provide the Riemannian gradient 
of $C = S \circ E$ when its domain is 
${\rm GL}_d(\mathbb{R})$.
\begin{proposition}
The Riemannian gradient of 
$(S \circ E) : {\rm GL}_d(\mathbb{R}) \to 
\mathbb{R}$ at the point $\bm A \in {\rm GL}_d(\mathbb{R})$ 
is given by  
\begin{equation}
{\rm grad} (S \circ E) (\bm A) = 
\widehat{\bm D} \bm A,
\end{equation}
where
\begin{equation}
\label{Dhat}
\widehat{\bm D} = \sum_{i=1}^{d-1} \int \bm Q_{\leq i} \bm Q_{>i}  \nabla v(\bm x) \left(\bm A\bm x\right)^{\top} d  \mu(\bm x) 
\end{equation}
\end{proposition}
\begin{proof}
To prove this result we check directly using the 
definition of Riemannian gradient. For 
any $\bm A \in {\rm GL}_d(\mathbb{R})$ and 
$\bm V \in T_{\bm A} {\rm GL}_d(\mathbb{R})$ we have 
\begin{equation}
\begin{aligned}
\left( \widehat{\bm D} \bm A, \bm V\right)_{\bm A} &= 
\left( \left[\sum_{i=1}^{d-1} \int \bm Q_{\leq i} \bm Q_{>i}  \nabla v(\bm x) \left( \bm A \bm x \right)^{\top} d \mu(\bm x) \right] \bm A, \bm V \right)_{\bm A} \\ 
&= {\rm trace} \left( \sum_{i=1}^{d-1} \int \bm Q_{\leq i} \bm Q_{>i} \nabla v(\bm x) \bm x^{\top} \bm A^{\top} d \mu(\bm x)  \bm A^{-\top} \bm V^{\top} \right) \\ 
&= \sum_{i=1}^{d-1} \int \bm Q_{\leq i} \bm Q_{>i} {\rm trace}\left( \nabla v(\bm x)  \bm x^{\top} \bm V^{\top} \right) d\mu( \bm x) \\
&= \sum_{i=1}^{d-1} \int \bm Q_{\leq i} \bm Q_{>i}  \nabla v(\bm x) \cdot \left( \bm V  \bm x \right) d\mu(\bm x), 
\end{aligned}
\end{equation}
where in the last equality we used the fact that ${\rm trace}( \bm v \bm w^{\top} ) = \bm v \cdot \bm w$ 
for any column vectors $\bm v, \bm w \in \mathbb{R}^d$. 
\end{proof}

\vs
\noindent
In general, the trace of $\widehat{\bm D}$ is not 
equal to zero and thus $\widehat{\bm D} \bm A$ is not 
an element of the tangent space $T_{\bm A} {\rm SL}_d(\mathbb{R})$ (see eqn. \eqref{lie_alg_coset}). 
In order to obtain the Riemannian gradient 
$\bm D \bm A$ of 
$S \circ E: {\rm SL}_d(\mathbb{R}) \to 
\mathbb{R}$, we modify the 
diagonal entries of $\widehat{\bm D}$ 
to ensure that $\bm D \bm A$ satisfies the properties 
of Riemannian gradient and also belongs 
to the tangent space $T_{\bm A}{\rm SL}_d(\mathbb{R})$. 

\vs
\begin{proof}[Proposition \ref{prop:riemannian_gradient}]
First we prove that $\bm D \bm A$ with 
$\bm D$ defined in 
\eqref{descent_generator} is an element of 
$T_{\bm A}{\rm SL}_d(\mathbb{R})$, 
i.e., we prove that ${\rm trace}(\bm D) = 0:$
\begin{equation}
\begin{aligned}
{\rm trace}(\bm D) &= \sum_{i=1}^{d-1} \int \bm Q_{\leq i} \bm Q_{>i} 
{\rm trace}\left( \nabla v(\bm x) \left( \bm A \bm x\right)^{\top} - \frac{\nabla v(\bm x)^{\top} \bm A \bm x}{d} \bm I_{d \times d} \right) d \mu(\bm x) \\
&= \sum_{i=1}^{d-1} \int \bm Q_{\leq i} \bm Q_{>i}  \left[
{\rm trace}\left( \nabla v(\bm x) \left(\bm A \bm x\right)^{\top} \right) - {\rm trace} \left(\frac{\nabla v(\bm x)^{\top} \bm A \bm x}{d} \bm I_{d \times d} \right) \right] d \mu(\bm x). 
\end{aligned}
\end{equation} 
It is easy to verify that ${\rm trace}\left( \nabla v(\bm x) \left(\bm A \bm x\right)^{\top}  \right) = {\rm trace} \left(\displaystyle\frac{\nabla v(\bm x)^{\top} \bm A \bm x}{d} \bm I_{d \times d} \right)$ 
and hence ${\rm trace}(\bm D) = 0$. 
Next we show that $(\bm D \bm A, \bm V)_{\bm A} = d_{\bm A} (S \circ E) \bm V$ for all $\bm A \in {\rm SL}_d(\mathbb{R})$ and $\bm V \in T_{\bm A} {\rm SL}_d(\mathbb{R})$. 
Indeed, for any 
$\bm A \in {\rm SL}_d(\mathbb{R})$ and 
$\bm V \in T_{\bm A} {\rm SL}_d(\mathbb{R})$ 
we have 
\begin{equation}
\begin{aligned}
(\bm D \bm A, \bm V)_{\bm A} &= \sum_{i=1}^{d-1} \int 
\bm Q_{\leq i} \bm Q_{>i} 
{\rm trace} \left[ \left( \nabla v(\bm x) \left( \bm A \bm x \right)^{\top} - \frac{\nabla v(\bm x)^{\top} \bm A \bm x}{d} \bm I_{d \times d} \right)  \left(\bm V \bm A^{-1}\right) ^{\top} \right] d \mu(\bm x) \\
&= \sum_{i=1}^{d-1} \int \bm Q_{\leq i} \bm Q_{>i} \left[
{\rm trace} \left(\nabla v(\bm x) \left(\bm A\bm x\right)^{\top}
 \left( \bm V \bm A^{-1}\right)^{\top} \right) 
-  \frac{\nabla v(\bm x)^{\top} \bm A \bm x}{d} {\rm trace}\left(\left( \bm V \bm A^{-1} \right)^{\top}\right) \right] d \mu(\bm x).
\end{aligned}
\end{equation}
Since $\bm V \in T_{\bm A} {\rm SL}_d(\mathbb{R})$ we have that 
$\bm V = \bm W \bm A$ for some real matrix $\bm W$ with ${\rm trace}(\bm W) = 0$. 
Using this in the preceding equation we have 
\begin{equation}
\begin{aligned}
(\bm D(\bm A) \bm A, \bm V)_{\bm A} 
&= \sum_{i=1}^{d-1} \int \bm Q_{\leq i} \bm Q_{>i} \left[
{\rm trace} \left(\nabla v(\bm x) \bm x^{\top} \bm A^{\top} \bm A^{-\top} \bm V^{\top} \right) 
-  \frac{\nabla v(\bm x)^{\top} \bm A \bm x }{d} {\rm trace}\left( \bm W^{\top} \right)	\right] d \mu(\bm x) \\
&=  \sum_{i=1}^{d-1} \int \bm Q_{\leq i} \bm Q_{>i}
{\rm trace} \left(\nabla v(\bm x) \bm x^{\top}  \bm V^{\top} \right) d\mu(\bm x) \\
&= \sum_{i=1}^{d-1} \int \bm Q_{\leq i} \bm Q_{>i}  \nabla v(\bm x) \cdot \left( \bm V \bm x \right) d\mu(\bm x) \\
 &= d_{\bm A} (S \circ E) \bm V,
\end{aligned}
\end{equation}
completing the proof.
\end{proof}

\begin{lemma}
\label{lemma:continuity_of_cost}
Let $\sigma_i(\alpha_i)$ ($i = 1,2,\ldots,d$, $\alpha_i = 1,2,\ldots,r_i$) be the 
multilinear spectrum of the FTT $v_{\T}(\bm x) \approx u_{\T}(\bm A \bm x)$  and 
assume that for each $i=1,2,\ldots,d$ the real numbers $\sigma_i(\alpha_i)$ 
are distinct for all $\alpha_i=1,2,\ldots,r_i$. Then the cost function 
$(S \circ E)$ is a differentiable at the point $\bm A$. 
\end{lemma}
\begin{proof}
Let $\bm \Theta(\epsilon) \in \mathcal{C}^1\left( (-\delta,\delta) , 
{\rm SL}_d(\mathbb{R}) \right)$ with $\bm \Theta(0) = \bm A$. 
The $\epsilon$-dependent multilinear spectrum 
$\sigma_i(\alpha_i;\epsilon)$ ($i = 1,2,\ldots,d$, $\alpha_i = 1,2,\ldots,r_i$) 
of $u_{\T}(\bm \Theta(\epsilon) \bm x)$ 
are given by the eigenvalues of an 
$\epsilon$-dependent self-adjoint compact Hermitian operator. 
In a neighborhood of $\epsilon = 0$ the eigenvalue 
$\sigma_i(\alpha_i;\epsilon)$ admits a series expansion \cite[p. xx]{kato}
\begin{equation}
\sigma_i(\alpha_i;\epsilon) = \sigma_i(\alpha_i ; 0) + \epsilon \hat{\sigma}_i(\alpha_i), 
\end{equation}
and thus $\sigma_i(\alpha_i; \epsilon)$ is differentiable with respect to $\epsilon$ at 
$\epsilon = 0$. 
Hence the sum of all eigenvalues 
$$\displaystyle\sum_{i=1}^{d-1} \sum_{\alpha_i = 1}^{r_i} \sigma_i(\alpha_i;\epsilon)$$ 
is differentiable at $\epsilon = 0$ and thus the cost function $C$ is 
differentiable at $\bm A \in {\rm SL}_d(\mathbb{R})$. 
\end{proof}

\bibliographystyle{plain}
\bibliography{bibliography_file}

\end{document}